\documentclass[11pt, a4paper,oneside, reqno]{amsart}
\usepackage[english]{babel}
\usepackage{amsmath, amsthm, amsfonts, mathrsfs, amssymb, amscd}
\usepackage{bm}
\usepackage{mathtools}
\mathtoolsset{centercolon}
\usepackage{mathabx}
\usepackage{bbm}
\usepackage{stmaryrd}
\usepackage{accents}
\usepackage[toc]{appendix}
\usepackage[shortlabels]{enumitem}
\usepackage[colorlinks, citecolor = blue, urlcolor={red}]{hyperref}
\usepackage{hyperref}
\usepackage{geometry}
\allowdisplaybreaks

\usepackage[colorinlistoftodos,prependcaption,textsize=small,textwidth=25mm]{todonotes}

\newtheorem{theorem}{Theorem}

\newtheorem{lemma}[theorem]{Lemma}

\newtheorem{proposition}[theorem]{Proposition}

\theoremstyle{remark} 
\newtheorem{remark}[theorem]{Remark}

\theoremstyle{definition} 
\newtheorem{definition}[theorem]{Definition}

\numberwithin{theorem}{section}
\numberwithin{equation}{section}

\def\R{{\mathbb R}}

\def\F{{\mathbb F}}

\def\S{{\mathbb S}}

\renewcommand{\Re}{\operatorname{Re}}

\renewcommand{\hat}{\widehat}
\newcommand{\supp}{\text{\rm supp\,}}

\newcommand{\mc}{\mathcal}

\newcommand{\Hinf}{H^{\infty}}
\newcommand{\RR}{\mathbb{R}}

\newcommand{\CC}{\mathbb{C}}
\newcommand{\NN}{\mathbb{N}}

\newcommand{\OO}{\mathcal{O}}

\renewcommand{\SS}{\mathcal{S}}

\newcommand{\BB}{\mathcal{B}}
\newcommand{\FF}{\mathcal{F}}

\newcommand{\RRdh}{\RR^d_+}
\newcommand{\RRd}{\RR^d}
\newcommand{\Cc}{C_{\mathrm{c}}}
\renewcommand{\d}{\partial}
\newcommand{\del}{\Delta}
\newcommand{\grad}{\nabla}

\newcommand{\eps}{\varepsilon}
\newcommand{\ph}{\varphi}

\newcommand{\gam}{\gamma}

\newcommand{\Tr}{\operatorname{Tr}}
\newcommand{\ext}{\operatorname{ext}}
\newcommand{\bTr}{\overline{\operatorname{Tr}}}
\newcommand{\bext}{\overline{\operatorname{ext}}}
\newcommand{\loc}{{\rm loc}}
\renewcommand{\b}{{\rm b}}
\newcommand{\ii}{{\rm i}} 
\newcommand\op{\mathop{\circ}\nolimits}
\renewcommand{\tilde}[1]{\widetilde{#1}}

\DeclarePairedDelimiter\abs{\lvert}{\rvert}

\DeclarePairedDelimiter{\nrm}\lVert\rVert

\newcommand{\cbraceb}[1]{\bigl\{#1\bigr\}}

\newcommand{\has}[1]{\Bigl(#1\Bigr)}

\newcommand{\dd}{\hspace{2pt}\mathrm{d}}

\DeclareMathOperator{\ind}{\mathbf{1}}
\DeclareMathOperator{\UMD}{UMD}

\DeclareMathOperator{\id}{id}

\begin{document}

\title[Complex interpolation of power-weighted Sobolev spaces]{Complex interpolation of power-weighted Sobolev spaces with boundary conditions}

\author[F.B. Roodenburg]{Floris B. Roodenburg}
\address[Floris Roodenburg]{Delft Institute of Applied Mathematics\\
Delft University of Technology \\ P.O. Box 5031\\ 2600 GA Delft\\The
Netherlands} \email{f.b.roodenburg@tudelft.nl}

\makeatletter
\@namedef{subjclassname@2020}{%
  \textup{2020} Mathematics Subject Classification}
\makeatother

\subjclass[2020]{Primary: 46B70, 46E35; Secondary: 46E40}
\keywords{Complex interpolation, weighted function spaces, vector-valued function spaces, Besov spaces, Triebel--Lizorkin spaces, Bessel potential spaces, Sobolev spaces, traces, boundary operators}

\thanks{The author is supported by the VICI grant VI.C.212.027 of the Dutch Research Council (NWO). The author thanks Emiel Lorist, Mark Veraar and the anonymous referee for valuable comments and suggestions leading to improvements of the paper}

\begin{abstract}
We characterise the complex interpolation spaces of weighted vector-valued Sobolev spaces with and without boundary conditions on the half-space and on smooth bounded domains. The weights we consider are power weights that measure the distance to the boundary and do not necessarily belong to the class of Muckenhoupt $A_p$ weights. First, we determine the higher-order trace spaces for weighted vector-valued Besov, Triebel--Lizorkin, Bessel potential and Sobolev spaces. This allows us to derive a trace theorem for boundary operators and to interpolate spaces with boundary conditions. Furthermore, we derive density results for weighted Sobolev spaces with boundary conditions.
\end{abstract}

\maketitle

\setcounter{tocdepth}{1}
\tableofcontents

\section{Introduction}
This paper aims to characterise complex interpolation spaces of weighted Sobolev spaces on a domain $\OO$ with power weights $w^{\d\OO}_{\gam}(x):=\operatorname{dist}(x,\d\OO)^{\gam}$ for $\gam>-1$ measuring the distance to the boundary. The motivation for this stems from the study of partial differential equations (PDEs) and is twofold.

First, it is well known that complex interpolation serves as an essential tool for studying differential operators and evolution equations, see, e.g., the monographs \cite{Am95, Am19, HNVW24, Lu95, PS16, Ya10}. For example, if a sectorial operator $A$ on a Banach space $X$ has bounded imaginary powers, then the domain of $A^\theta$ for $\theta\in(0,1)$ is the complex interpolation space $[X, D(A)]_\theta$, see \cite{Se71}. Typically, $D(A)$ is a Sobolev space on a spatial domain that incorporates the boundary conditions of the boundary value problem. Thus, with the aid of complex interpolation, we can identify domains of fractional powers, which in turn play an important role for perturbation techniques in the theory of maximal regularity, see, e.g., \cite{DDHPV04, KKW06,KW01b}. Complex interpolation of spaces with boundary conditions was studied by Grisvard and Seeley in \cite{Gr63, Gr67, Se71, Se72} and, due to the numerous applications, the topic is now widely addressed in the literature, see, e.g., \cite{Am19, BE19, BL76, GGKR02, Gu91, Tr78, Ya10} and the references therein. 
    
Secondly, solutions to PDEs on a domain $\OO\subseteq \RRd$ may exhibit blow-up behaviour near the boundary of $\OO$. Using weighted spaces with weights of the form $w_{\gam}^{\d\OO}(x):=\operatorname{dist}(x,\d\OO)^{\gam}$ for some suitable value of $\gam$, one can nevertheless study problems which may be ill-posed in unweighted spaces. Moreover, with weighted spaces certain compatibility conditions needed in the unweighted case can be avoided.  Many authors have employed weighted spaces to solve (stochastic) PDEs, see, e.g., \cite{DK18, KK04, KN14, Kr99b, Kr01, KL99, Ma11}. In most of these works, homogeneous weighted Sobolev spaces are used which are known to form a complex interpolation scale, see \cite[Proposition 2.4]{Lo00}. Recently, in \cite{LLRV24, LV18} an alternative approach to solving PDEs on the half-space is presented via the $\Hinf$-calculus for the Laplacian on inhomogeneous Sobolev spaces with weights outside the Muckenhoupt $A_p$ class. To further investigate this $\Hinf$-calculus and its consequences, it is crucial to have access to complex interpolation for these weighted Sobolev spaces. In particular, in the upcoming work \cite{LLRV25}, domains of fractional powers and perturbation techniques will be used to obtain the $\Hinf$-calculus for the Laplacian on bounded $C^{1,\lambda}$-domains for $\lambda\in[0,1]$ depending on $\gam$.\\

For $m\in\NN_0$ we denote by $\BB$ an \emph{$m$-th order normal boundary operator} at $\d\RRdh=\{(0, \tilde{x}):\tilde{x}\in \RR^{d-1}\}$. Consider the power weight $w_\gam(x) := \operatorname{dist}(x,\d\RRdh)^\gam=|x_1|^\gam$  for $x=(x_1,\tilde{x})\in \RR_+\times\RR^{d-1}$ and let $W^{k,p}_{\BB}(\RRdh,w_{\gam})$ be the closed subspace of functions $f\in W^{k,p}(\RRdh,w_{\gam})$ such that $\BB f =0$, whenever the traces exist. We refer to Sections \ref{sec:traceB} and \ref{sec:compl_intp} for the precise definitions. Our main result includes the following characterisation of complex interpolation spaces (see Section \ref{subsec:int_p2}). 
\begin{theorem}\label{thm:intro_intp}
  Let $p\in(1,\infty)$, $k\in\{2,3,\dots\}$ and $\gam\in (-1,\infty)\setminus\{jp-1:j\in\NN_1\}$. Then for $\ell\in \{1,\dots, k-1\}$ we have
  \begin{subequations}
      \begin{align}
  \big[L^p(\RRdh, w_{\gam}), W^{k,p}(\RRdh, w_{\gam})\big]_{\frac{\ell}{k}}&= W^{\ell,p}(\RRdh, w_{\gam}),\label{eq:1.1a}\\
    \big[L^p(\RRdh, w_{\gam}), W^{k,p}_{\BB}(\RRdh, w_{\gam})\big]_{\frac{\ell}{k}}&= W_{\BB}^{\ell,p}(\RRdh, w_{\gam}).\label{eq:1.1b}
  \end{align}
  \end{subequations}
  Moreover, by localisation, the results also hold for smooth bounded domains. 
\end{theorem}
The main novelty of Theorem \ref{thm:intro_intp}, compared to the existing literature, is the range of weights $w_{\gam}$ with $\gam>-1$. Unfortunately, the proof in \cite{Am19} for the unweighted case $\gam=0$ does not carry over immediately to the weighted setting. In the unweighted case, the result \eqref{eq:1.1a} is well known and is used to prove \eqref{eq:1.1b} in \cite{Am19}. In the weighted setting, we can only prove the inclusion ``$\hookrightarrow$" in \eqref{eq:1.1a} directly with the aid of Wolff interpolation \cite{Wolff82} (see Proposition \ref{prop:intp_W}). Nevertheless, this embedding is sufficient to prove \eqref{eq:1.1b} using the characterisation of the trace spaces (see Theorem \ref{thm:intro_Tracem}) and similar arguments as in \cite{Am19}. Then, the converse inclusion ``$\hookleftarrow$" in \eqref{eq:1.1a} also follows.
Furthermore, we make the following remarks concerning Theorem \ref{thm:intro_intp}.
\begin{enumerate}[(i)]
  \item The class of normal boundary operators is considered in, e.g., \cite{Am19, Ge65, Gu91, Se72} and includes the important cases of Dirichlet and Neumann boundary conditions, but also mixed boundary conditions with space-dependent coefficients are allowed. Actually, we prove Theorem \ref{thm:intro_intp} for systems of normal boundary operators. Differential operators with these types of boundary operators are studied in, e.g., \cite{Ag62, ADN59, ADN64, GGS10, Gr73, Gr74, Lu95}. 
  \item If $\gam\in (-1,p-1)$, i.e., $w_{\gam}$ is a Muckenhoupt $A_p$ weight, then the weighted Bessel potential and Sobolev spaces coincide: $H^{k,p}(\RRdh, w_{\gam})=W^{k,p}(\RRdh, w_{\gam})$ for $k\in\NN_0$, and $H^{s,p}(\RRdh,w_{\gam})$ for $s\in\RR$ forms a complex interpolation scale, see \cite[Proposition 5.5 \& 5.6]{LMV17}. In the case $\gam\in(-1,p-1)$ we obtain Theorem \ref{thm:intro_intp} for Bessel potential spaces with boundary conditions and fractional smoothness as well, see Theorem \ref{thm:intp_HB}. This partially extends the result in \cite[Theorem VIII.2.4.8]{Am19}. 
  
      
      The interpolation result in Theorem \ref{thm:intro_intp} without boundary conditions is new for $\gam>p-1$. We refer to \cite{Dr23, SSV14} for results on complex interpolation of Besov and Triebel--Lizorkin spaces with weights outside the class of $A_p$ weights.
  \item We can allow for vector-valued weighted Sobolev spaces $W^{k,p}(\RRdh,w_{\gam};X)$, where $X$ is a $\UMD$ Banach space. 
\end{enumerate}

A key ingredient for proving Theorem \ref{thm:intro_intp} is the characterisation of the trace spaces of weighted vector-valued Besov, Triebel--Lizorkin, Bessel potential and Sobolev spaces. 
The results are summarised in Theorem \ref{thm:intro_Tracem} (see Sections \ref{sec:Trace} and \ref{subsec:tracemSob}). 

\begin{theorem}\label{thm:intro_Tracem}
  Let $p\in(1,\infty)$, $q\in[1,\infty]$, $m\in\NN_0$, $\gam>-1$, $s>m+\frac{\gam+1}{p}$ and let $X$ be a Banach space. Let $\Tr_m=\Tr \op \d_1^m$ be the $m$-th order trace operator at $\{(0,\tilde{x}):\tilde{x}\in \RR^{d-1}\}$. Then
  \begin{enumerate}[(i)]
    \item $\Tr_m\big(B^{s}_{p,q}(\RRd,w_{\gam};X)\big)=B^{s-m-\frac{\gam+1}{p}}_{p,q}(\RR^{d-1};X)$,
    \item\label{it:F} $\Tr_m\big(F^{s}_{p,q}(\RRd,w_{\gam};X)\big)=B^{s-m-\frac{\gam+1}{p}}_{p,p}(\RR^{d-1};X)$,
    \item\label{it:H} $\Tr_m\big(H^{s,p}(\RRdh,w_{\gam};X)\big)=B^{s-m-\frac{\gam+1}{p}}_{p,p}(\RR^{d-1};X)$ if $\gam\in(-1,p-1)$,
    \item\label{it:W} $\Tr_m\big(W^{s,p}(\RRdh,w_{\gam};X)\big)=B^{s-m-\frac{\gam+1}{p}}_{p,p}(\RR^{d-1};X)$ if $s\in\NN_1$ and $\gam\notin \{jp-1:j\in\NN_1\}$,
  \end{enumerate}
 where the notation $\Tr_m(\mc{A}_1)=\mc{A}_2$ means that $\Tr_m : \mc{A}_1\to \mc{A}_2$
is continuous and surjective, and that it has a continuous right inverse (the so-called extension operator $\ext_m$). 
\end{theorem}
The innovative part of Theorem \ref{thm:intro_Tracem} is the extension of the range of weights for Sobolev spaces, which eventually allows us to prove Theorem \ref{thm:intro_intp} for these weights as well. To prove our weighted results, we partially extend the arguments in \cite{Am19, SSS12} for the unweighted case. Moreover, we apply a Sobolev embedding for Triebel--Lizorkin spaces to deduce Theorem \ref{thm:intro_Tracem}\ref{it:F}. This is in contrast to \cite{SSS12}, where an atomic approach was used.

In the unweighted scalar-valued setting, i.e., $\gam=0$ and $X=\CC$, the trace spaces are well known, see \cite{Sc10, Tr78, Tr83}. In the unweighted vector-valued setting, the trace spaces are characterised in \cite{Am19, SSS12}.\\

As a consequence of Theorem \ref{thm:intro_Tracem}, we can derive a trace theorem for the vector $\bTr_m=(\Tr_0,\dots, \Tr_m)$ as well. For weighted scalar-valued Sobolev spaces this is proved in \cite[Section 2.9.2]{Tr78} and \cite{KimD07} via different techniques.  

Moreover, from Theorem \ref{thm:intro_Tracem} we derive some density results in Section \ref{subsec:density}. In particular, we show that the space of test functions 
      \begin{equation*}
        \cbraceb{f \in \Cc^{\infty}(\overline{\RRdh};X): (\d_1^m f)|_{\partial\RRdh}=0}
      \end{equation*}
      is dense in $W^{k,p}_{\Tr_m}(\RRdh, w_{\gam};X)$. This density result will be needed in \cite{LLRV25} to find trace characterisations for weighted Sobolev spaces on bounded $C^{1}$-domains. Application of a standard mollification procedure would result in all traces $\bTr_m$ being zero instead of only one trace. We remark that results for mollification with preservation of boundary values are contained in \cite{Bu98, HPR20}.\\

We comment on some open and related problems. Theorem \ref{thm:intro_intp} only deals with interpolation of the smoothness parameter, while it might be possible to interpolate in other parameters (such as $p$ and $\gam$) as well. Interpolation results for Sobolev spaces with different weights (but without complicated boundary conditions) are contained in \cite{CE19}, \cite[Section 5.3]{KK23}, \cite[Proposition 3.9]{LLRV24} and \cite{Lo82}.

For $\gam\geq p-1$ we have the result of Theorem \ref{thm:intro_intp} only for spaces with integer smoothness. It would be interesting to see if our results carry over to spaces of fractional smoothness, for instance, by defining $H^{s,p}(\RRdh, w_{\gam};X):=[L^p(\RRdh, w_\gam;X), W^{k,p}(\RRdh, w_{\gam};X)]_{s/k}$ for some $s\in \RR_+\setminus\NN_1$ and $k\in \NN_1$ such that $k>s$. It should be noted that for $\gam\geq p-1$ one cannot use the standard definition of Bessel potential spaces used for Muckenhoupt weights, i.e., $\gam\in (-1,p-1)$.

Besides complex interpolation, the real interpolation method plays an important role in the theory of evolution equations as well (see \cite[Section 4.4]{GV17} and \cite[Section 17.2.b]{HNVW24}). Some related results for real interpolation of weighted spaces are contained in \cite{Bu82, Gr63, Lo92,Tr78}. We note that if we have interpolation results for Besov spaces with boundary conditions, then we can obtain Theorem \ref{thm:intro_intp} for real interpolation, see \cite[Theorem VIII.2.4.5]{Am19}. We expect that our results also hold for Besov spaces with boundary conditions if $\gam\in (-1,p-1)$. However, for $\gam\geq p-1$ a more detailed study about interpolation of Besov spaces and their relation to Sobolev spaces is required. A particular problem is that for $\gam\geq p-1$ we only have the embedding
    \begin{equation*}
      \qquad B^{n+\theta}_{p,q}(\RRdh, w_{\gam})\hookrightarrow \big(W^{n,p}(\RRdh, w_{\gam}), W^{n+1,p}(\RRdh,w_{\gam})\big)_{\theta, q},\quad n\in \NN_0,\, \theta\in(0,1),
    \end{equation*}  
while equality holds for $\gam\in(-1,p-1)$, see \cite[Section 7]{LV18}. In particular, for $n=0$ it is proved in \cite[Remark 7.14]{LV18} that the embedding ``$\hookleftarrow$" fails for $\gam \notin (-1,p-1)$.

Finally, we note that the special values $\gam= j p-1$ with $j\in\NN_1$ are excluded, since in these cases the study of the trace operator is more involved, see \cite{KK23}. Moreover, we do not have certain trace characterisations available which play a crucial role in the proof of the main theorem. 
For $\gam\leq -1$ we have that $\Cc^{\infty}(\RRdh)$ is dense in $W^{k,p}(\RRdh,w_{\gam})$ and thus all traces will be zero.

\subsection*{Outline} The outline of this paper is as follows. In Section \ref{sec:prelim} we introduce some concepts and preliminary results needed throughout the paper. In Section \ref{sec:Trace} we determine the trace spaces for weighted Besov and Triebel--Lizorkin spaces. In Section \ref{sec:tracesHW} we extend the results to weighted Bessel potential and Sobolev spaces. As a consequence, we prove some density results. Section \ref{sec:traceB} deals with the characterisation of the trace space for boundary operators. Finally, in Section \ref{sec:compl_intp} we prove the main result about complex interpolation of weighted Sobolev spaces with and without boundary conditions.

\section{Preliminaries}\label{sec:prelim}
\subsection{Notation}\label{subsec:notation}
We denote by $\NN_0$ and $\NN_1$ the sets of natural numbers starting at $0$ and $1$, respectively.
For $d\in\NN_1$ the half-spaces are given by $\RR^d_{\pm}=\RR_{\pm}\times\RR^{d-1}$, where $\RR_+=(0,\infty)$ and $\RR_-=(-\infty,0)$. For $\gam>-1$, $\OO\subseteq \RRd$ open and $x\in \OO$, we define the power weight $w_{\gam}^{\d\OO}(x):=\operatorname{dist}(x,\d\OO)^\gam$. In addition, $w_{\gam}(x):=w_{\gam}^{\smash{\d\RRdh}}(x)=|x_1|^{\gam}$, where $x=(x_1,\tilde{x})\in \RR\times \RR^{d-1}$.

For two topological vector spaces $X$ and $Y$, the space of continuous linear operators is $\mc{L}(X,Y)$ and $\mc{L}(X):=\mc{L}(X,X)$. Unless specified otherwise, $X$ will always denote a Banach space with norm $\|\cdot\|_X$ and the dual space is $X':=\mc{L}(X,\CC)$.\\

For an open and non-empty $\OO\subseteq \RR^d$ and $k\in\NN_0\cup\{\infty\}$, the space $C^k(\OO;X)$ denotes the space of $k$-times continuously differentiable functions from $\OO$ to some Banach space $X$. As usual, this space is equipped with the compact-open topology. In the case $k=0$, we write $C(\OO;X)$ for $C^0(\OO;X)$. Furthermore, we write $C^k_{{\rm b}}(\OO;X)$ for the space of all functions $f\in C^k(\OO;X)$ such that $\d^{\alpha} f$ is bounded on $\OO$ for all multi-indices $\alpha\in \NN_0^d$ with $|\alpha|\leq k$.

Let $\Cc^{\infty}(\OO;X)$ be the space of compactly supported smooth functions on $\OO$ equipped with its usual inductive limit topology. The space of $X$-valued distributions is given by $\mc{D}'(\OO;X):=\mc{L}(\Cc^{\infty}(\OO);X)$. Moreover, $\Cc^{\infty}(\overline{\OO};X)$ is the space of smooth functions with their support in a compact set contained in $\overline{\OO}$.

We denote the Schwartz space by $\SS(\RRd;X)$ and $\SS'(\RRd;X):=\mc{L}(\SS(\RRd);X)$ is the space of $X$-valued tempered distributions. For $f\in \SS(\RRd;X)$ we define the $d$-dimensional Fourier transform and its inverse by
\begin{align*}
  (\FF f)(\xi)&=\hat{f}(\xi) =\int_{\RRd}f(x)e^{-\ii x\cdot \xi}\dd x, \qquad \xi\in \RRd,\\(\FF^{-1} f)(x)& =\frac{1}{(2\pi)^d}\int_{\RRd}f(\xi)e^{\ii x\cdot \xi}\dd \xi,\qquad x\in \RRd,
\end{align*}
which extends to $\SS'(\RRd;X)$ by duality.

\subsection{Weighted function spaces}\label{subsec:spaces} For $\OO\subseteq\RR^{d}$ open, a \emph{weight} is a nonnegative and locally integrable function $w$ on $\OO$ that takes values in $(0,\infty)$ almost everywhere. For $p\in(1,\infty)$ we denote the class of \emph{Muckenhoupt weights} on $\OO\in\{\RRd,\RRdh\}$ by $A_p(\OO)$, i.e., $w\in A_p(\OO)$ if
\begin{equation*}
    [w]_{A_p(\OO)}:=\sup_{B}\Big(\frac{1}{|B|}\int_B w(x)\dd x \Big) \Big(\frac{1}{|B|}\int_{B}w(x)^{-\frac{1}{p-1}}\dd x\Big)^{p-1}<\infty,
\end{equation*}
where the supremum is taken over all balls $B\subseteq \OO$.
Moreover, $A_{\infty}(\OO)=\bigcup_{p>1}A_p(\OO)$. We will mainly focus on the case of power weights $w_{\gam}(x_1,\tilde{x})=\operatorname{dist}(x, \d\RRdh)^\gam=|x_1|^{\gam}$ for $\gam>-1$, where $x=(x_1,\tilde{x})\in \RR^d$. For $\OO\in\{\RR^{d}, \RRdh\}$ it holds that $w_{\gam}\in A_p(\OO)$ if and only if $\gam\in (-1,p-1)$, see \cite[Example 7.1.7]{Gr14_classical_3rd}.\\

For $\OO\subseteq \RRd$ open, $p\in[1,\infty)$, $w$ a weight and $X$ a Banach space, we define the \emph{weighted Lebesgue space $L^p(\OO,w;X)$}, which consists of all equivalence classes of strongly measurable $f\colon \OO\to X$ such that
\begin{equation*}
\nrm{f}_{L^p(\OO,w;X)} := \has{\int_{\OO}\|f(x)\|^p_X\:w(x)\dd x }^{1/p}<\infty.
\end{equation*}
Additionally, for $k\in\NN_0$ and $w$ a weight such that $w^{-\frac{1}{p-1}}\in L^1_{\loc}(\OO)$, we define the \emph{weighted Sobolev space} consisting of all $f\in L^p(\OO,w;X)$ such that
\begin{equation*}
  \|f\|_{W^{k,p}(\OO,w;X)}:=\sum_{|\alpha|\leq k}\|\d^{\alpha}f\|_{L^p(\OO,w;X)}<\infty.
\end{equation*}

\begin{remark}\label{rem:L1loc}
  The local $L^1$ condition for $w^{-\frac{1}{p-1}}$ ensures that all the derivatives $\d^{\alpha}f$ are locally integrable. For $\OO=\RRdh$ this condition holds if $w\in A_p(\RRdh)$ or $w=w_{\gam}$ with $\gam\in\RR$. For $\OO=\RR^d$ the local $L^1$ condition holds if $w\in A_p(\RRd)$ or $w=w_{\gam}$ with $\gam\in(-\infty,p-1)$.  For $\gam\geq p-1$ one has to be careful with defining the weighted Sobolev spaces on the full space because functions might not be locally integrable near $x_1=0$, see \cite{KO1984}. This explains why, for example, we cannot employ classical reflection arguments from $\RRdh$ to $\RR^d$ if $\gam\geq p-1$.
\end{remark}

We continue with the definitions of weighted Besov, Triebel--Lizorkin and Bessel potential spaces. We denote by $\Phi(\RRd)$ the set of \emph{inhomogeneous Littlewood--Paley sequences} consisting of sequences $(\ph_n)_{n\geq 0}\subseteq \mc{S}(\RRd)$ such that $\widehat{\ph}_0:=\widehat{\ph}$ and $\widehat{\ph}_n(\xi):=\widehat{\ph}(2^{-n}\xi)-\widehat{\ph}(2^{-n+1}\xi)$ for $\xi\in \RR^d$ and $n\in\NN_1$, where the generating function $\ph$ satisfies $0\leq \widehat{\ph}(\xi)\leq 1$ for $\xi\in\RRd$, $\widehat{\ph}(\xi)=1$ for $|\xi|\leq 1$ and $\widehat{\ph}(\xi)=0$ for $|\xi|\geq \frac{3}{2}$.
For $f\in \mc{S}' (\RRd;X)$ we use the notation
\begin{equation}\label{eq:notationS_n}
  S_nf:=\ph_n\ast f\quad \text{ for } n\in\NN_0\quad \text{ and }\quad S_{-1} f=0,
\end{equation}
where $(\ph_n)_{n\geq 0}\in \Phi(\RRd)$. \\

We record the following weighted and vector-valued extension of the multiplier result in \cite[Section 1.6.3]{Tr83}, see \cite[Proposition 2.4]{MV12}.

\begin{lemma}\label{lem:est_SnAinfty}
    Let $p\in[1,\infty)$, $q\in [1,\infty]$, $w\in A_\infty(\RR^d)$ and let $X$ be a Banach space. Let $K>0$ be fixed and for $n\in \NN_0$ let $f_n\in L^p(\RR^d,w;X)$ be such that $\supp\hat{f}_n\subseteq \{\xi\in\RR^d:|\xi|\leq  2^n K\}$. Then
\begin{equation*}
  \|(S_n f_n)_{n\geq 0}\|_{L^p(\RR^d,w;\ell^q(X))}\leq C \|(f_n)_{n\geq 0}\|_{L^p(\RR^d,w;\ell^q(X))},
\end{equation*}
where the constant $C>0$ is independent of $(f_n)_{n\geq 0}$.

In particular, for any fixed $n\in \NN_0$ and $f\in L^p(\RRd, w;X)$ with $\supp\hat{f}\subseteq \{\xi\in\RR^d:|\xi|\leq  2^n K\}$, it holds that
\begin{equation*}
  \|S_n f\|_{L^p(\RR^d,w;X)}\leq C \|f\|_{L^p(\RR^d,w;X)},
\end{equation*}
where the constant $C>0$ is independent of $f$ and $n$.
\end{lemma}
\begin{remark}
    The constant $C$ in Lemma \ref{lem:est_SnAinfty} may depend on $K$, but later on $K$ will always be a fixed number associated to the zeroth-order block of Littlewood--Paley decompositions. 
\end{remark}
\begin{proof}
    Since $w\in A_{\infty}(\RR^d)$, we can find an $r\in (0,\min\{p,q\})$ such that $w\in A_{\frac{p}{r}}(\RR^d)$. Then \cite[Proposition 2.4]{MV12} implies
\begin{align*}
   \|(S_n f_n)_{n\geq 0}\|_{L^p(\RR^d,w;\ell^q(X))}&=   \|(\FF^{-1}(\hat{\ph}_n \hat{f}_n\,))_{n\geq 0}\|_{L^p(\RR^d,w;\ell^q(X))} \\
&\leq C \sup_{k\geq 0}\|(1+|\cdot|^{\frac{d}{r}})\FF^{-1}[\hat{\ph}_k(2^kK \cdot)]\|_{L^1(\RR^d)}\|(f_n)_{n\geq 0}\|_{L^p(\RR^d,w;\ell^q(X))},
\end{align*}
where $C>0$ is independent of $(f_n)_{n\geq 0}$. Using the scaling property $\hat{\ph}_k(\cdot)=\hat{\ph}_1(2^{-k+1}\cdot)$ for $k\geq 1$ gives that $\hat{\ph}_k(2^k K \cdot)$ has its support in a ball with a radius that is independent of $k$. Now, \cite[Lemma 14.2.12]{HNVW24} gives the first statement. The second statement follows from applying the first statement with $(f_n)_{n\geq 0}$ replaced by $(0,\dots,0, f, 0,\dots)$ with $f$ at the $n$-th position.
\end{proof}

For $p\in (1,\infty)$, $q\in[1,\infty], s\in\RR, w\in A_{\infty}(\RRd)$ and $X$ a Banach space, we define the \emph{ weighted Besov, Triebel--Lizorkin and Bessel potential spaces} as the space of all $f\in \mc{S}'(\RRd;X)$ for which, respectively, 
\begin{align*}
  \|f\|_{B^s_{p,q}(\RRd,w;X)}&:=\|(2^{n s}\ph_n\ast f)_{n\geq 0}\|_{\ell^q(L^p(\RRd, w;X))}<\infty,\\
  \|f\|_{F^s_{p,q}(\RRd,w;X)}&:=\|(2^{n s}\ph_n\ast f)_{n\geq 0}\|_{L^p(\RRd, w;\ell^q(X))}<\infty,\\
  \|f\|_{H^{s,p}(\RRd,w;X)}&:=\|\FF^{-1}((1+|\cdot|^2)^{s/2}\FF f)\|_{L^p(\RRd, w;X)}<\infty,\quad w\in A_p(\RRd).
\end{align*}
The definitions of the Besov and Triebel--Lizorkin spaces are independent of the chosen Littlewood--Paley sequence up to an equivalent norm, see \cite[Proposition 3.4]{MV12}. We note that all the spaces defined above embed continuously into $\SS'(\RRd;X)$, see \cite[Section 5.2.1e]{Li14}. Furthermore, $\SS(\RR^d;X)$ embeds continuously into all the above spaces and this embedding is dense if $p,q<\infty$, see \cite[Lemma 3.8]{MV12}.

Let $\OO\subseteq \RR^d$  be open. To define the $B, F$ and $H$ spaces on $\OO$, we use \emph{restriction/factor spaces}. Let $\F\in\{B,F,H\}$, then we define
\begin{equation*}
  \F(\OO,w;X) := \big\{f\in \mc{D}'(\OO;X): \exists g\in \F(\RRd,w;X), g|_{\OO}=f \big\}
\end{equation*}
endowed with the norm
\begin{equation*}
  \|f\|_{\F(\OO,w;X)} := \inf\{\|g\|_{\F(\RRd,w;X)}:g|_{\OO}=f \}.
\end{equation*}
Similarly, we can define the factor space for weighted Sobolev spaces. If $\OO=\RRdh$ and $w\in A_p(\RRdh)$, then the weighted Sobolev spaces on $\RRdh$ as defined above coincide with their factor spaces, see \cite[Section 5]{LMV17}. The proof is based on the existence of an extension operator from $\RRdh$ to $\RRd$, see, e.g., \cite[Theorems 5.19 \& 5.22]{AF03}. Throughout the rest of the paper, we will refer to such operators as \emph{reflection operators}. For a discussion of reflection operators on weighted spaces, we refer to \cite[Section 2.1.2]{Tu00}.

\subsection{Properties of Banach spaces} We recall the Fatou and $\UMD$ property for Banach spaces.
If $X, \mc{X}$ are Banach spaces such that $\mc{X}$ embeds continuously into $\SS'(\RR^d;X)$, then we say that $\mc{X}$ has the \emph{Fatou property} (see \cite{BM91, Fr86}) if for all $(f_n)_{n\geq 0}\subseteq \mc{X}$ such that $\lim_{n\to \infty} f_n =f$ in $\SS'(\RRd;X)$ and $\liminf_{n\to\infty}\|f_n\|_{\mc{X}}<\infty$, we have 
\begin{equation*}
 f\in \mc{X}\quad \text{ and }\quad \|f\|_{\mc{X}}\leq \liminf_{n\to\infty}\|f_n\|_{\mc{X}}.
\end{equation*}
The proof of the Fatou property for Besov and Triebel--Lizorkin spaces in \cite{Fr86} carries over to the vector-valued case (see \cite[Proposition 2.18]{SSS12}) and to the weighted case as well. 
\begin{proposition}\label{prop:fatou}
  Let $p\in(1,\infty)$, $q\in [1,\infty]$, $s\in \RR$, $w\in A_{\infty}(\RRd)$ and let $X$ be a Banach space. Then $B^s_{p,q}(\RRd, w;X)$ and $F^s_{p,q}(\RRd, w;X)$ have the Fatou property.
\end{proposition}

A Banach space $X$ satisfies the condition $\UMD$ (unconditional martingale differences) if and only if the Hilbert transform extends to a bounded operator on $L^p(\RR;X)$. We recall the following properties of $\UMD$ spaces, see for instance \cite[Chapter 4 \& 5]{HNVW16}.
\begin{enumerate}[(i)]
\item Hilbert spaces are $\UMD$ spaces. In particular, $\CC$ is a $\UMD$ Banach space.
  \item If $p\in(1,\infty)$, $(S,\Sigma,\mu)$ is a $\sigma$-finite measure space and $X$ is a $\UMD$ Banach space, then $L^p(S;X)$ is a $\UMD$ Banach space.
  \item $\UMD$ Banach spaces are reflexive.
\end{enumerate}
The UMD property is known to be necessary for many results on vector-valued Sobolev spaces (see, e.g., \cite[Section 5.6]{HNVW16} and \cite[Corollary 13.3.9]{HNVW24}).

\subsection{Embeddings}
We start with the well-known Hardy's inequality for spaces with power weights, see, e.g., \cite[Section 5]{Ku85}. First, recall that for $p\in[1,\infty)$, $\gam\in (-\infty,p-1)$ and $X$ a Banach space, we have the Sobolev embedding
\begin{equation*}
  W^{1,p}(\RR_+, w_{\gam}; X)\hookrightarrow C([0,\infty);X),
\end{equation*}
so that the boundary value $u(0)$ is well defined for $u\in W^{1,p}(\RR_+, w_\gam;X)$, see, e.g., \cite[Lemma 3.1]{LV18}.
We state the following version of Hardy's inequality from \cite[Lemma 3.2 \& Corollary 3.3]{LV18}, which will allow us to estimate Sobolev norms with both different weights and smoothness.
\begin{lemma}\label{lem:Hardy}
  Let $p\in[1,\infty)$ and let $X$ be a Banach space. Let $u\in W^{1,p}(\RR_+,w_{\gam};X)$ and assume either
\begin{enumerate}[(i)]
\item \label{it:lem:Hardy1} $\gam<p-1$ and $u(0)=0$, or,
\item $\gam>p-1$.
\end{enumerate}
Then
  \begin{equation*}
    \|u\|_{L^p(\RR_+,w_{\gam-p};X)}\leq C_{p,\gam}\|u'\|_{L^p(\RR_+,w_{\gam};X)}.
  \end{equation*}
  In addition, if $k\in \NN_1$ and $\gam>p-1$, then 
  \begin{equation*}
      W^{k,p}(\RRdh, w_{\gam};X)\hookrightarrow W^{k-1,p}(\RRdh, w_{\gam-p};X).
  \end{equation*}
\end{lemma}

We continue with some standard embedding properties of vector-valued function spaces. 
Let $p\in(1,\infty)$, $s\in \RR$, $k\in\NN_0$, $w\in A_p(\RRd)$ and let $X$ be a Banach space. Then
\begin{subequations}\label{eq:embFHW}
  \begin{align}
  F^s_{p,1}(\RRd,w;X) &\hookrightarrow H^{s,p}(\RRd,w;X)\hookrightarrow F^s_{p,\infty}(\RRd,w;X), \label{eq:embFHF}\\
  F^k_{p,1}(\RRd,w;X) &\hookrightarrow W^{k,p}(\RRd,w;X)\hookrightarrow F^k_{p,\infty}(\RRd,w;X), \label{eq:embFWF}
\end{align}
\end{subequations}
see \cite[Proposition 3.12]{MV12}. It holds that $F^{s}_{p,2}(\RRd;X)=H^{s,p}(\RRd;X)$ if and only if $X$ is isomorphic to a Hilbert space, see \cite{HM96} or \cite[Theorem 14.7.9]{HNVW24} for a more direct proof. For general Banach spaces $X$, the Bessel potential and Sobolev spaces are not equal to a Triebel--Lizorkin space, but instead, we have the embeddings \eqref{eq:embFHW}. In particular, we note that the embedding
\begin{equation}\label{eq:embFL}
    F^0_{p,1}(\RRd,w;X) \hookrightarrow L^p(\RRd,w;X),
\end{equation}
holds for all $w\in A_\infty(\RRd)$. This immediately follows from taking $L^p(\RRd,w)$-norms in the estimate $\|f\|_X\leq \sum_{n=0}^\infty \|S_n f\|_X$ with $f\in \SS(\RRd;X)$ and a density argument, see \cite[Lemma 3.8]{MV12}. 

Moreover, for all $p\in[1,\infty)$, $q\in[1,\infty]$, $s\in\RR$, $w\in A_{\infty}(\RRd)$, it holds that
\begin{equation}\label{eq:embBFB}
  B^{s}_{p,\min\{p,q\}}(\RRd,w;X)\hookrightarrow F^{s}_{p,q}(\RRd,w;X)\hookrightarrow B^{s}_{p,\max\{p,q\}}(\RRd,w;X),
\end{equation}
and if, in addition, $1\leq q_0\leq q_1\leq \infty$, then
\begin{equation}\label{eq:embBBFF}
  B^{s}_{p,q_0}(\RRd,w;X)\hookrightarrow  B^{s}_{p,q_1}(\RRd,w;X)\quad \text{ and }\quad  F^{s}_{p,q_0}(\RRd,w;X)\hookrightarrow  F^{s}_{p,q_1}(\RRd,w;X),
\end{equation}
see \cite[Proposition 3.11]{MV12}. \\

Finally, we state the following theorem on Sobolev embeddings for Besov and Triebel--Lizorkin spaces. Its proof is based on a similar result for radial weights, see \cite[Theorems 1.1 \& 1.2]{MV12}. Moreover, in \cite{MV14b} general $A_\infty$ weights are considered in the scalar-valued setting. The embedding with scalar-valued spaces can be extended to vector-valued spaces only for Muckenhoupt weights, see \cite[Theorem 1.3]{MV14b}. However, this is not sufficient for our purposes.
\begin{theorem}\label{thm:embBBFF}
  Let $1<p_0\leq p_1<\infty$, $q_0,q_1\in[1,\infty]$, $s_0>s_1$, $\gam_0,\gam_1>-1$ and let $X$ be a Banach space. Suppose that
  \begin{equation*}
    \frac{\gam_1}{p_1}\leq \frac{\gam_0}{p_0}\quad \text{ and }\quad s_0-\frac{d+\gam_0}{p_0}=s_1-\frac{d+\gam_1}{p_1}.
  \end{equation*}
  Then 
  \begin{enumerate}[(i)]
    \item\label{it:thm:embBBFF_F} $F^{s_0}_{p_0,q_0}(\RRd,w_{\gam_0};X)\hookrightarrow F^{s_1}_{p_1,q_1}(\RRd,w_{\gam_1};X)$,
    \item\label{it:thm:embBBFF_B} $B^{s_0}_{p_0,q_0}(\RRd,w_{\gam_0};X)\hookrightarrow B^{s_1}_{p_1,q_1}(\RRd,w_{\gam_1};X)$ if, in addition, $q_0\leq q_1$.
  \end{enumerate}
\end{theorem}
\begin{proof}
  \textit{Step 1: proof of \ref{it:thm:embBBFF_B}.} By \eqref{eq:embBBFF} it suffices to consider $q:=q_0=q_1$. Let $f\in B^{s_0}_{p_0,q}(\RRd, w_{\gam_0};X)$, $(\ph_n)_{n\geq 0}\in \Phi(\RR^d)$ and define $S_nf:= \ph_n \ast f$ as in \eqref{eq:notationS_n}. We write $s_1+\delta = s_0$, where
  \begin{equation*}
    \delta = \delta_0+\delta_1 \quad \text{ with }\quad \delta_0:=\frac{\gam_0+1}{p_0}-\frac{\gam_1+1}{p_1}\quad\text{ and }\quad\delta_1:= \frac{d-1}{p_0}-\frac{d-1}{p_1}.
  \end{equation*}
  Let $\FF_{x_1}$ and $\FF_{\tilde{x}}$ be the Fourier transform with respect to $x_1\in\RR$ and $\tilde{x}\in \RR^{d-1}$, respectively. Note that $\supp \FF (S_n f)\subseteq \{\xi\in \RR^d:|\xi|\leq 3\cdot 2^{n-1}\}$, so that $\supp \FF_{x_1} (S_n f(\cdot, \tilde{x}))\subseteq\{\xi_1\in \RR:|\xi_1|\leq 3\cdot 2^{n-1}\}$ for fixed $\tilde{x}\in \RR^{d-1}$ and $\supp \FF_{\tilde{x}} (S_n f(x_1, \cdot))\subseteq\{\tilde{\xi}\in \RR^{d-1}:|\tilde{\xi}|\leq 3\cdot 2^{n-1}\}$ for fixed $x_1\in \RR$ (see \cite[Theorem 4.9, Step 1 \& 2]{SSS12}). Now, the Plancherel--Polya--Nikol'skij type inequality from \cite[Proposition 4.1]{MV12} twice and Minkowski's inequality yield
  \begin{equation}\label{eq:Lp_Sn_est}
      \begin{aligned}
    \|S_n f\|_{L^{p_1}(\RR^d, w_{\gam_1};X)} & \leq C\, 2^{\delta_0 n}\Big(\int_{\RR^{d-1}}\Big(\int_\RR|x_1|^{\gam_0} \|S_n f(x_1, \tilde{x})\|_X^{p_0}\dd x_1\Big)^{\frac{p_1}{p_0}}\dd \tilde{x}\Big)^{\frac{1}{p_1}} \\
     & \leq C\, 2^{\delta_0 n}\Big(\int_{\RR}|x_1|^{\gam_0}\Big(\int_{\RR^{d-1}} \|S_n f(x_1, \tilde{x})\|_X^{p_1}\dd \tilde{x}\Big)^{\frac{p_0}{p_1}}\dd x_1\Big)^{\frac{1}{p_0}} \\
     &\leq C\, 2^{(\delta_0+\delta_1)n}\|S_n f\|_{L^{p_0}(\RR^d, w_{\gam_0};X)},
  \end{aligned}
  \end{equation}
  where $C>0$ is independent of $f$ and $n$. Statement \ref{it:thm:embBBFF_B} now follows from the definition of the Besov norms and the fact that $s_1+\delta = s_0$. 
  
\textit{Step 2: proof of \ref{it:thm:embBBFF_F}.} We follow the proof of \cite[Theorem 1.2]{MV12} for radial weights to deduce \ref{it:thm:embBBFF_F}.  By \eqref{eq:embBBFF} it suffices to prove the required estimate for $q_1=1$. First, let $f\in F^{s_0}_{p_0,q_0}(\RRd, w_{\gam_0};X)\cap  F^{s_1}_{p_1,1}(\RRd, w_{\gam_1};X)$ and without loss of generality we may assume $f\neq 0$. Let $\theta_0\in[0,1)$ be such that $\frac{1}{p_1}-\frac{1-\theta_0}{p_0}=0$ and define the continuous function $g: (\theta_0, 1]\to \RR$ by
\begin{equation*}
  g(\theta)=\frac{\frac{\gam_1}{p_1}-\frac{(1-\theta)\gam_0}{p_0}}{\frac{1}{p_1}-\frac{1-\theta}{p_0}}.
\end{equation*}
Note that $g(1)=\gam_1$ and since $\gam_1>-1$ there exists a $\theta\in (\theta_0,1)$ such that $\gam:=g(\theta)>-1$. Define $r$ and $t$ by
\begin{equation*}
  \frac{1}{p_1}=\frac{1-\theta}{p_0}+\frac{\theta}{r}\quad\text{ and }\quad t-\frac{d+\gam}{r}=s_1-\frac{d+\gam_1}{p_1}.
\end{equation*}
From the conditions on the parameters it follows that $r\in [p_1,\infty)$, and $t<s_0$ satisfies $s_1=\theta t+(1-\theta)s_0$. Moreover, it holds that $\theta\gam p_1 / r +(1-\theta)\gam_0 p_1/p_0 = \gam_1$. Therefore, the Gagliardo--Nirenberg type inequality from \cite[Proposition 5.1]{MV12} gives
\begin{equation}\label{eq:Gag_Nir_F}
  \|f\|_{F^{s_1}_{p_1,1}(\RRd, w_{\gam_1};X)}\leq C\|f\|^{1-\theta}_{F^{s_0}_{p_0,q_0}(\RRd, w_{\gam_0};X)}\|f\|^{\theta}_{F^{t}_{r,r}(\RRd, w_{\gam};X)}.
\end{equation}
Moreover, by \eqref{eq:embBFB}, \ref{it:thm:embBBFF_B} (using that $r\geq p_1$) and \eqref{eq:embBBFF}, we obtain
\begin{equation}\label{eq:besov_est}
  \|f\|_{F^{t}_{r,r}(\RRd, w_{\gam};X)}\eqsim \|f\|_{B^{t}_{r,r}(\RRd, w_{\gam};X)}\leq C \|f\|_{B^{s_1}_{p_1,p_1}(\RRd, w_{\gam_1};X)}\leq C\|f\|_{F^{s_1}_{p_1,1}(\RRd, w_{\gam_1};X)}.
\end{equation}
Substituting \eqref{eq:besov_est} in \eqref{eq:Gag_Nir_F}, yields
\begin{equation}\label{eq:est_intersection}
  \|f\|_{F^{s_1}_{p_1,1}(\RRd,w_{\gam_1};X)}\leq C\|f\|_{F^{s_0}_{p_0,q_0}(\RRd,w_{\gam_0};X)}
\end{equation}
for $f\in F^{s_0}_{p_0,q_0}(\RRd, w_{\gam_0};X)\cap  F^{s_1}_{p_1,1}(\RRd, w_{\gam_1};X) $.

The estimate for $f\in F^{s_0}_{p_0,q_0}(\RRd, w_{\gam_0};X)$ follows from the Fatou property of Triebel--Lizorkin spaces. We provide the argument for the convenience of the reader. For $f\in F^{s_0}_{p_0,q_0}(\RRd, w_{\gam_0};X)$ the series $f_{N}=\sum_{n=0}^N S_n f$ converges to $f$ in $\SS'(\RRd;X)$ as $N\to\infty$. Note that for all $N\in \NN_0$ we have $f_N\in F^{s_0}_{p_0,q_0}(\RRd, w_{\gam_0};X)\cap  F^{s_1}_{p_1,1}(\RRd, w_{\gam_1};X)$ by \eqref{eq:Lp_Sn_est}. From Lemma \ref{lem:est_SnAinfty} we have that
\begin{equation*}
  \|f_N\|_{F^{s_0}_{p_0,q_0}(\RRd, w_{\gam_0};X)}\leq C\|f\|_{F^{s_0}_{p_0,q_0}(\RRd, w_{\gam_0};X)}\quad \text{ for all }N\in\NN_0,
\end{equation*}
and together with \eqref{eq:est_intersection} this implies
\begin{equation*}
  \|f_N\|_{F^{s_1}_{p_1,1}(\RRd, w_{\gam_1};X)}\leq C\|f\|_{F^{s_0}_{p_0,q_0}(\RRd, w_{\gam_0};X)}\quad \text{ for all }N\in\NN_0.
\end{equation*}
Proposition \ref{prop:fatou} yields $f\in F^{s_1}_{p_1,1}(\RRd, w_{\gam_1};X)$ and 
\begin{equation*}
  \|f\|_{F^{s_1}_{p_1,1}(\RRd, w_{\gam_1};X)}\leq \liminf_{N\to \infty}\|f_N\|_{F^{s_1}_{p_1,1}(\RRd, w_{\gam_1};X)}\leq C\|f\|_{F^{s_0}_{p_0,q_0}(\RRd, w_{\gam_0};X)}.
\end{equation*}
This completes the proof.
\end{proof}

\subsection{The complex interpolation method}\label{sec:compl_intp_defs} We collect some basic properties of the complex interpolation method. For a detailed study of interpolation methods, we refer to, e.g., \cite{BL76, Tr78}.\\

Let $A \subseteq \RR$ be an interval. Let $\S_A:=\{z\in \CC: \Re z \in A\}$ be a strip in the complex plane and we write $\S$ if $A=(0,1)$. Let $(X_0,X_1)$ be an interpolation couple of complex Banach spaces, i.e., a pair of Banach spaces which are continuously embedded in a linear Hausdorff space $\mc{X}$. We denote by $\mathscr{H}(X_0,X_1)$ the complex vector space of all continuous functions $f:\overline{\S}\to X_0+X_1$ such that $f$ is holomorphic on $\S$ with values in $X_0+X_1$ and the mapping $t\mapsto f(j+ \ii t)$ belongs to $C_{{\rm b}}(\RR; X_j)$ for $j\in \{0,1\}$. 
The space $\mathscr{H}(X_0,X_1)$ endowed with the norm
\begin{equation*}
  \|f\|_{\mathscr{H}(X_0,X_1)}:= \max_{j\in\{0,1\}}\sup_{t\in \RR}\|f(j+\ii t)\|_{X_j}
\end{equation*}
is a Banach space.

\begin{definition}
  Let $(X_0,X_1)$ be an interpolation couple and $\theta\in(0,1)$, then the \emph{complex interpolation space $[X_0,X_1]_\theta$} is the complex vector space of all $x\in X_0+X_1$ such that $f(\theta)=x$ for some $f\in\mathscr{H}(X_0,X_1)$.
\end{definition}
The space $[X_0,X_1]_\theta$ endowed with the norm
\begin{equation*}
  \|x\|_{[X_0,X_1]_\theta}:=\inf\big\{\|f\|_{\mathscr{H}(X_0,X_1)}: f(\theta)=x \big\}
\end{equation*}
is a Banach space and satisfies
\begin{equation*}
  X_0\cap X_1 \hookrightarrow [X_0,X_1]_\theta\hookrightarrow X_0+X_1.
\end{equation*}
Finally, if $f\in \mathscr{H}(X_0,X_1)$ and $g_t(z): =f(z+\ii t)$ for $z\in \overline{\S}$ and $t\in \RR$, then $g_t\in \mathscr{H}(X_0,X_1)$ and 
\begin{equation}\label{eq:est_compl_shift}
  \|f(\theta+\ii t)\|_{[X_0,X_1]_\theta}\leq C \|f\|_{\mathscr{H}(X_0,X_1)},\qquad \theta\in (0,1),\, t\in \RR.
\end{equation}

\section{Trace spaces of Besov and Triebel--Lizorkin spaces}\label{sec:Trace}
In this section, we characterise the higher-order trace spaces of weighted Besov and Triebel--Lizorkin spaces on $\RR^d$. In Section \ref{subsec:zeroTrace} we first study the zeroth-order trace operator and in Section \ref{subsec:mTrace} we extend the results to higher-order trace operators. 

\subsection{The zeroth-order trace operator}\label{subsec:zeroTrace}
To define the trace operator we pursue as in \cite{SSS12} and use a concept due to Nikol’skij \cite{Ni52}. There exist several other and, in particular cases, equivalent definitions for the trace operator. We refer to \cite[Remark 4.8]{SSS12} for an overview.\\

 Let $X$ be a Banach space and let $f : \RR^d \to X$ be strongly measurable. A function 
$ g : \RR^{d-1} \to X$ is called the trace of $f$ on $\{0\}\times \RR^{d-1}$, if there are $\tilde{f} : \RR^d \to X$,
$p \in [1,\infty]$ and $\delta > 0$ such that
\begin{enumerate}[(i)]
  \item \label{it:defTr1} $\tilde{f}=f$ almost everywhere with respect to the Lebesgue measure on $\RR^d$,
  \item\label{it:defTr2} $\tilde{f}(x_1,\cdot)\in L^p(\RR^{d-1};X)$ for $|x_1|<\delta$,
  \item\label{it:defTr3} $\tilde{f}(0,\cdot)= g$ almost everywhere with respect to the Lebesgue measure on $\RR^{d-1}$,
  \item\label{it:defTr4} $\lim_{x_1\to 0}\|\tilde{f}(x_1,\cdot)-g\|_{L^p(\RR^{d-1};X)}=0$.
\end{enumerate}
By \cite[Remark 4.2 \& 4.3]{SSS12} this definition is independent of $\tilde{f}, p$ and $\delta$ and it coincides 
with the restriction of $f$ to $\{(0,\tilde{x}) :\tilde{x}\in \RR^{d-1}\}$ if $f$ is continuous on $(-\delta, \delta)\times \RR^{d-1}$. If the trace exists according to the above definition, then we write $\Tr f =g$, so that we obtain a linear operator $\Tr$.

The representative $\tilde{f}$ for $f$ can be obtained via the formula $\tilde{f}=\sum_{n=0}^\infty S_n f$, whenever this makes sense. In Proposition \ref{prop:existenceTrace} we will use this representative to show that the trace operator is well defined if $f$ belongs to a certain Besov space.\\

The following results are extensions of the arguments in \cite{SSS12} to the case with weights $w_{\gam}(x)=|x_1|^\gam$, which can also be found in version 1 of the arXiv preprint of \cite{MV14}. We note that radial weights $|x|^{\gam}$ are considered in \cite{AKS08, HS10}. Anisotropic versions of trace theorems have appeared in, e.g., \cite{JS08, Li20, Li21} and see the references therein.

We start with a technical lemma analogous to \cite[Lemma 4.5]{SSS12}.
\begin{lemma}\label{lem:SSS4.5}
  Let $p\in(1,\infty)$, $\gam>-1$ and let $X$ be a Banach space. Let $h\in L^p(\RR^d, w_{\gam};X)$ such that $\supp \hat{h}\subseteq \{\xi\in \RR^{d}:|\xi|\leq R\}$ for some $R>0$. Then there exists a $C>0$, independent of $h$ and $R$, such that for all $x_1\in \RR$, we have
  \begin{equation*}
    \|h(x_1,\cdot)\|_{L^p(\RR^{d-1};X)}\leq C R^{\frac{\gam+1}{p}}\|h\|_{L^p(\RRd,w_{\gam};X)}.
  \end{equation*}
\end{lemma}
\begin{proof}
  First of all, note that $h$ is pointwise defined by the support condition on its Fourier transform, see, e.g., \cite[Lemma 14.2.9]{HNVW24}.
  
  \textit{Step 1: the case $\gam\in (-1,p-1)$.} By scaling and translating, it suffices to consider $x_1=0$ and $R=1$. 
  Let $\hat{\ph}\in \Cc^{\infty}(\RRd)$ such that $\hat{\ph}=1$ on $B(0,1)$. Then for every $\tilde{x}\in \RR^{d-1}$ we have
  \begin{equation*}
    h(0,\tilde{x})  =\big(\FF^{-1}(\hat{\ph}\,\hat{h}\,)\big)(0,\tilde{x})=\int_{\RR^{d-1}}\int_{\RR}h(-y_1, \tilde{x}-\tilde{y})|y_1|^{\frac{\gam}{p}}\ph(y_1, \tilde{y})|y_1|^{-\frac{\gam}{p}}\dd y_1\dd \tilde{y},
  \end{equation*}
  so that H{\"o}lder's inequality implies
  \begin{equation*}
    \|h(0,\tilde{x})\|_X\leq \int_{\RR^{d-1}} \|h(\cdot, \tilde{x}-\tilde{y})\|_{L^p(\RR, w_{\gam};X)} \|\ph(\cdot, \tilde{y})\|_{L^{p'}(\RR, w_{\gam'})}\dd \tilde{y},
  \end{equation*}
  where $p' = \frac{p}{p-1}$ and $\gam'=-\frac{\gam}{p-1}$. By Minkowski's inequality we obtain
  \begin{equation*}
    \|h(0,\cdot)\|_{L^p(\RR^{d-1};X)}\leq \|h\|_{L^p(\RRd, w_{\gam};X)}\int_{\RR^{d-1}}\|\ph(\cdot, \tilde{y})\|_{L^{p'}(\RR, w_{\gam'})}\dd \tilde{y} = C\|h\|_{L^p(\RRd, w_{\gam};X)},
  \end{equation*}
  using that $\gam'>-1$ in the last step.
  
  \textit{Step 2: the case $\gam\geq p-1$.} Let $\FF_{x_1}$ be the Fourier transform with respect to $x_1\in\RR$.
  Since $\supp \hat{h} \subseteq \{\xi\in \RR^{d}:|\xi|\leq R\}$ we have that $\supp \FF_{x_1} (h(\cdot, \tilde{x}))\subseteq \{\xi_1\in \RR: |\xi_1|\leq R\}$ for fixed $\tilde{x}\in \RR^{d-1}$, see \cite[Theorem 4.9, Step 1 \& 2]{SSS12}. By the Plancherel--Polya--Nikol'skij type inequality from \cite[Proposition 4.1]{MV12} we obtain
  \begin{equation}\label{eq:PPNineq}
      \begin{aligned}
    \|h\|_{L^p(\RR^{d};X)} &=\Big(\int_{\RR^{d-1}}\int_{\RR}\|h(x_1,\tilde{x})\|^p_X
    \dd x_1 \dd \tilde{x}\Big)^{\frac{1}{p}}  \\
     & \leq C R^{\frac{\gam}{p}}\Big(\int_{\RR^{d-1}}\int_{\RR}|x_1|^{\gam}\|h(x_1,\tilde{x})\|_X^p\dd x_1\dd \tilde{x}\Big)^{\frac{1}{p}} =  C R^{\frac{\gam}{p}} \|h\|_{L^p(\RRd, w_{\gam};X)}.
  \end{aligned}
  \end{equation}
  From Step 1 applied to $\gam =0$ and \eqref{eq:PPNineq}, we obtain
  \begin{equation*}
    \|h(x_1,\cdot)\|_{L^p(\RR^{d-1};X)}\leq C R^{\frac{1}{p}}\|h\|_{L^p(\RRd;X)}\leq C R^{\frac{\gam+1}{p}} \|h\|_{L^p(\RRd, w_{\gam};X)}.
  \end{equation*}
  Note that this second step actually works for all $\gamma>0$.
\end{proof}

Using Lemma \ref{lem:SSS4.5} we can now prove that $\Tr f$ exists if $f\in  B^{\frac{\gam+1}{p}}_{p,1}(\RRd,w_{\gam};X)$. In particular, provided that $f$ has this regularity, we will in the sequel always work with the identity \eqref{eq:tracef=S_nf}. The following result is a weighted version of \cite[Proposition 4.4]{SSS12}.
\begin{proposition}\label{prop:existenceTrace}
  Let $p\in(1,\infty)$, $\gam>-1$ and let $X$ be a Banach space. Then the trace of a function $f\in B^{\frac{\gam+1}{p}}_{p,1}(\RRd,w_{\gam};X)$ exists and for any $(\ph_n)_{n\geq 0}\in  \Phi(\RR^d)$, we have
  \begin{equation}\label{eq:tracef=S_nf}
    \Tr f=\sum_{n=0}^\infty \ph_n\ast f(0, \cdot)\qquad \text{in }L^p(\RR^{d-1};X).
  \end{equation}
\end{proposition}
\begin{proof}
  Let $f\in B^{\frac{\gam+1}{p}}_{p,1}(\RRd, w_{\gam};X)$, $(\ph_n)_{n\geq 0}\in \Phi(\RRd)$ and set $\tilde{f}=\sum_{n=0}^\infty S_n f$. We verify the conditions \ref{it:defTr1}, \ref{it:defTr2} \ref{it:defTr4} for the definition of the trace operator. Note that \ref{it:defTr3} holds automatically.
  
  \textit{Step 1: proof of \ref{it:defTr1}.} Since $f\in B^{\frac{\gam+1}{p}}_{p,1}(\RRd, w_{\gam};X)$ we have $\sum_{n=0}^\infty \|S_n f\|_{L^p(\RRd, w_{\gam};X)}<\infty$ and therefore the series converges to $f$ in $L^p(\RRd, w_{\gam};X)$ and thus $f=\tilde{f}$ almost everywhere on $\RRd$.

  \textit{Step 2: proof of \ref{it:defTr2}.}
  By Lemma \ref{lem:SSS4.5} (using that $\supp \FF(S_n f)\subseteq \{\xi\in \RR^d: |\xi|\leq 3\cdot 2^{n-1}\}$) we obtain for every $x_1\in\RR$ that
  \begin{align*}
    \|\tilde{f}(x_1, \cdot)\|_{L^p(\RR^{d-1};X)} & \leq \big\|(\|S_n f(x_1, \cdot)\|_{L^p(\RR^{d-1};X)})_{n\geq 0}\big\|_{\ell^1} \\
    &\leq C \big\|(2^{\frac{\gam+1}{p}n}\|S_n f\|_{L^p(\RRd,w_{\gam};X)})_{n\geq 0}\big\|_{\ell^1}= C\|f\|_{B^{\frac{\gam+1}{p}}_{p,1}(\RR^{d}, w_\gam;X)}.
  \end{align*}
  
  \textit{Step 3: proof of \ref{it:defTr4}.} Let $\eps>0$. By applying Lemma \ref{lem:SSS4.5}, we can find an $N\geq 1$ large enough such that
  \begin{equation}\label{eq:largeN}
      \Big\|\sum_{n=N}^\infty \big(S_nf(x_1, \cdot)- S_n f(0, \cdot)\big)\Big\|_{L^p(\RR^{d-1};X)}\leq C\sum_{n=N}^{\infty} 2^{\frac{\gam+1}{p}n}\|S_n f\|_{L^p(\RR^{d}, w_\gam;X)}<\frac{\eps}{2}.
  \end{equation}
  From now on we fix such an $N$. Let $n< N$ and $|x_1|< 2^{-N}$, and we define $h_{x_1}(t,\tilde{x}):=S_n f(t+x_1, \tilde{x})- S_nf(t, \tilde{x})$ for $t\in \RR$ and $\tilde{x}\in \RR^{d-1}$. Note that by the mean value theorem we have 
  \begin{equation}\label{eq:MVT}
      \| h_{x_1}(t,\tilde{x})\|_X \leq |x_1| \sup_{|y|< 2^{-N}}\|\grad S_n f (t+y, \tilde{x})\|_X.
  \end{equation}
  By Lemma \ref{lem:SSS4.5}, \eqref{eq:MVT} and \cite[Lemma 2.3]{SSS12}, we obtain that
      \begin{align*}
      \|S_n f(x_1, \cdot)- S_n f(0, \cdot)\|_{L^p(\RR^{d-1};X)} &= \|h_{x_1}(0, \cdot)\|_{L^p(\RR^{d-1};X)}\\
      &\leq C \,  2^{\frac{\gam+1}{p}n}\|h_{x_1}\|_{L^p(\RR^d, w_\gam;X)}\\
      &\leq C\, 2^{\frac{\gam+1}{p}n}|x_1|\Big\|(t,\tilde{x})\mapsto \sup_{y<2^{-N}}\|\grad S_n f(t+y, \tilde{x})\|_X\Big\|_{L^p(\RR^d, w_\gam)}\\
       &\leq C\, 2^{(\frac{\gam+1}{p}+1)n}|x_1|  \big\|(\mc{M}\|S_nf\|_X^r)^{\frac{1}{r}}\big\|_{L^{p}(\RRd, w_{\gam})},
  \end{align*}
  for some $r\in (0,p)$, where $\mc{M}$ denotes the Hardy--Littlewood maximal operator. Pick $r$ such that $w_{\gam}\in A_{\frac{p}{r}}(\RRd)$ and using the boundedness of $\mc{M}$ on $L^{p/r}(\RRd, w_\gam)$ gives
  \begin{equation}\label{eq:smallN}
    \|S_n f(x_1, \cdot)- S_n f(0, \cdot)\|_{L^p(\RR^{d-1};X)}\leq C \, 2^{(\frac{\gam+1}{p}+1)n}|x_1| \|S_n f\|_{L^p(\RRd, w_{\gam};X)}.
  \end{equation}
  Upon choosing $|x_1|$ small enough, we now obtain by combining \eqref{eq:largeN} and \eqref{eq:smallN} that
  \begin{align*}
      \Big\|\sum_{n=0}^\infty \big(S_nf(x_1, \cdot)- S_n f(0, \cdot)\big)\Big\|_{L^p(\RR^{d-1};X)}\leq \frac{\eps}{2}+ C|x_1|\sum_{n=0}^{N-1} 2^{(\frac{\gam+1}{p}+1)n} \|S_n f\|_{L^p(\RRd, w_{\gam};X)} < \eps.
  \end{align*}
  For instance, it suffices to take $|x_1|< \eps  \,2^{-N}C^{-1}\|f\|^{-1}_{B^{(\gam+1)/p}_{p,1}(\RRd, w_\gam;X)}$, where $C$ is the constant in \eqref{eq:smallN}. This shows that $\lim_{x_1\to 0}\tilde{f}(x_1,\cdot)=\tilde{f}(0,\cdot)$ in $L^p(\RR^{d-1};X)$.
\end{proof}

\begin{remark}\label{rem:gam>0}
  Note that for $\gam>0$ we have $B^{\frac{\gam+1}{p}}_{p,1}(\RRd,w_{\gam};X)\hookrightarrow B^{\frac{1}{p}}_{p,1}(\RRd;X)$ by Theorem \ref{thm:embBBFF}. Thus, Proposition \ref{prop:existenceTrace} for $\gam>0$ can also be deduced from the unweighted result. This argument does not apply for $\gam\in(-1,0)$. 
\end{remark}

With the above results, the trace space of a weighted Besov space can be characterised following the arguments from \cite[Theorem 4.9 \& Proposition 4.12]{SSS12} for the unweighted case. For $p=q\in (1,\infty)$ and $A_p$ weights the result was already shown in \cite[Th\'eor\`eme 7.1]{Gr63}.

\begin{theorem}\label{thm:trace0Besov}
  Let $p\in(1,\infty)$, $q\in[1,\infty]$, $\gam>-1$, $s>\frac{\gam+1}{p}$ and let $X$ be a Banach space. Let $\Tr$ be as in Proposition \ref{prop:existenceTrace}. Then
  \begin{equation*}
    \Tr:B^{s}_{p,q}(\RRd,w_{\gam};X)\to B_{p,q}^{s-\frac{\gam+1}{p}}(\RR^{d-1};X)
  \end{equation*}
   is a continuous and surjective operator. Moreover, there exists a continuous right inverse $\ext$ of $\Tr$ which is independent of $s,p,q,\gam$ and $X$.
\end{theorem}
\begin{proof}
Let $(\ph_n)_{n\geq 0}\in \Phi(\RR^d)$, $(\phi_n)_{n\geq 0}\in \Phi(\RR^{d-1})$, $f\in \SS'(\RR^d;X)$, $g\in \SS'(\RR^{d-1};X)$ and define $S_n f:=\ph_n \ast f$ and $T_ng:=\phi_n\ast g$. Moreover, set $S_{-1}f=T_{-1}g=0$.
  
  \textit{Step 1: trace operator.}  The continuity of the trace operator can be shown similarly as in Step 4 of the proof of \cite[Theorem 4.9]{SSS12} if we apply Lemma \ref{lem:SSS4.5} instead of \cite[Equation (4.5)]{SSS12}. 
  
  \textit{Step 2: extension operator.} The extension operator is defined similarly as in \cite[Section 2.7.2]{Tr83} and \cite{SSS12}. Let $g\in B^{s-\frac{\gam+1}{p}}_{p,q}(\RR^{d-1};X)$ and take $(\rho_n)_{n\geq0}\in \Phi(\RR)$ such that $\rho_1(0)=2$ (and therefore $\rho_n(0)=2^n$ for all $n\geq 0$). We set
\begin{equation}\label{eq:def_ext_tr0}
  \ext g(x_1,\tilde{x})=\sum_{n=0}^\infty 2^{-n}\rho_n(x_1) T_n g(\tilde{x})\quad \text{ in }\SS'(\RRd;X).
\end{equation}
Note that $\ext$ is independent of $s,p,q,\gam$ and $X$. Moreover, since $\rho_n(x_1)=2^{n-1}\rho_1(2^{n-1}x_1)$ it holds that
\begin{equation}\label{eq:shift_rho}
  \|\rho_n\|_{L^p(\RR, w_{\gam})}= 2^{(n-1)(1-\frac{\gam+1}{p})}\|\rho_1\|_{L^p(\RR,w_{\gam})}, \qquad n\in\NN_0.
\end{equation}
Therefore, using \eqref{eq:shift_rho}, Young's inequality, $s-\frac{\gam+1}{p}>0$ and $\gam>-1$, we find
\begin{align*}
\sum_{n=0}^\infty 2^{-n}\|\rho_n T_n g\|_{L^p(\RRd, w_{\gam};X)}&=\sum_{n=0}^\infty 2^{-n }\|\rho_n\|_{L^p(\RR, w_{\gam})}\|T_n g\|_{L^p(\RR^{d-1};X)}\\
&\leq C \|g\|_{L^p(\RR^{d-1};X)} \sum_{n=0}^\infty 2^{-n\frac{\gam+1}{p}} \leq C \|g\|_{B^{s-\frac{\gam+1}{p}}_{p,q}(\RR^{d-1};X)},
\end{align*}
where the constant $C$ is independent of $g$ and $n$. For the last estimate we have used a geometric series ($\gam>-1$) and that $B^{\eps}_{p,q}(\RR^{d-1};X)\hookrightarrow B^{0}_{p,1}(\RR^{d-1};X)\hookrightarrow L^p(\RR^{d-1};X)$ for $\eps>0$, see \cite[Theorem 14.4.19 \& Proposition 14.4.18]{HNVW24}. This shows that $\ext g$ is well defined in $L^p(\RRd,w_{\gam};X)$.
Now, Lemma \ref{lem:est_SnAinfty} and \eqref{eq:shift_rho} imply that
\begin{align*}
  \|S_n \ext g\|_{L^p(\RRd, w_{\gam};X)} & \leq C\sum_{j=-1}^12^{-(n+j)}\|\rho_{n+j}\|_{L^p(\RR, w_{\gam})}\|T_{n+j}g\|_{L^p(\RR^{d-1};X)}  \\
   &\leq C \sum_{j=-1}^1 2^{-(n+j)\frac{\gam+1}{p}}\|T_{n+j}g\|_{L^p(\RR^{d-1};X)},
\end{align*}
where the constant $C$ is independent of $g$ and $n$. This implies that
\begin{equation*}
  \|\ext g\|_{B^s_{p,q}(\RRd,w_{\gam};X)}\leq C \sum_{j=-1}^1 \big\|\big(2^{n(s-\frac{\gam+1}{p})}\|T_{n+j}g\|_{L^p(\RR^{d-1};X)}\big)_{n\geq 0}\big\|_{\ell^q}\leq \|g\|_{B^{s-\frac{\gam+1}{p}}_{p,q}(\RR^{d-1};X)}.
\end{equation*}
This proves the continuity of $\ext$.

Since $\rho_n(0)=2^n$ it follows from the formulas \eqref{eq:tracef=S_nf} and \eqref{eq:def_ext_tr0} that $\Tr( \ext g) = g$ for all $g\in \SS(\RR^{d-1};X)$. By density (see \cite[Lemma 3.8]{MV12}) this extends to all $g\in B^{s-\frac{\gam+1}{p}}_{p,q}(\RR^{d-1};X)$ with $q<\infty$. If $q=\infty$, then we have $B^{s-\frac{\gam+1}{p}}_{p,\infty}(\RR^{d-1};X)\hookrightarrow B^{s-\frac{\gam+1}{p}-\eps}_{p,1}(\RR^{d-1};X)$ for all $\eps>0$ (see \cite[Theorem 14.4.19]{HNVW24}). This shows that $\ext$ is also the right inverse for $\Tr$ in this case. Finally, the existence of a right inverse of $\Tr$ implies that $\Tr$ is surjective.
\end{proof}

Using Theorem \ref{thm:trace0Besov} and the Sobolev embeddings for Triebel--Lizorkin spaces from Theorem \ref{thm:embBBFF}, we characterise the trace spaces for Triebel--Lizorkin spaces. In contrast to \cite[Lemma 4.15]{SSS12} and \cite[Theorem 3.6]{HS10}, the following proof is Fourier analytic and does not use an atomic approach. In the unweighted scalar-valued setting, other methods are used in \cite[Theorem 2.7.2]{Tr83} to obtain the trace space for Triebel--Lizorkin spaces. 
\begin{theorem}\label{thm:trace0TriebelLizorkin}
  Let $p\in(1,\infty)$, $q\in[1,\infty]$, $\gam>-1$, $s> \frac{\gam+1}{p}$ and let $X$ be a Banach space. Then
  \begin{equation*}
    \Tr:F^{s}_{p,q}(\RRd,w_{\gam};X)\to B_{p,p}^{s-\frac{\gam+1}{p}}(\RR^{d-1};X)
  \end{equation*}
   is a continuous and surjective operator. Moreover, there exists a continuous right inverse $\ext$ of $\Tr$ which is independent of $s,p,q,\gam$ and $X$. This extension operator coincides with the extension operator from Theorem \ref{thm:trace0Besov}.
\end{theorem}
\begin{proof}
  \textit{Step 1: trace operator.} 
  Let $f\in F_{p,q}^s(\RRd,w_{\gam};X)$ and for $\eps>0$ small take $s_1:=s-\frac{\eps}{p}$ and $\gam_1:=\gam-\eps$. Then $\gam_1/p< \gam/p$ and $s-(d+\gam)/p=s_1-(d+\gam_1)/p$. So by Theorem \ref{thm:trace0Besov}, \eqref{eq:embBFB} and Theorem \ref{thm:embBBFF}, we obtain
  \begin{align*}
    \|\Tr f\|_{B^{s-\frac{\gam+1}{p}}_{p,p}(\RR^{d-1};X)} & = \|\Tr f\|_{B_{p,p}^{s_1-\frac{\gam_1+1}{p}}(\RR^{d-1};X)}\leq C \|f\|_{B^{s_1}_{p,p}(\RRd,w_{\gam_1};X)}\\
      &\leq C\|f\|_{F^{s_1}_{p,p}(\RRd,w_{\gam_1};X)}\leq C \|f\|_{F^{s}_{p,q}(\RRd,w_{\gam};X)}.
  \end{align*}
  
  \textit{Step 2: extension operator.} 
  Let $g\in B_{p,p}^{s-\frac{\gam+1}{p}}(\RR^{d-1};X)$ and for $\eps>0$ small take $s_0=s+\eps/p$ and $\gam_0=\gam+\eps$. Then
  by Theorem \ref{thm:embBBFF}, \eqref{eq:embBFB} and Theorem \ref{thm:trace0Besov}, we obtain
  \begin{align*}
    \|\ext g\|_{F^s_{p,q}(\RR^{d},w_{\gam};X)} & \leq C \|\ext g\|_{F^{s_0}_{p,p}(\RRd, w_{\gam_0};X)}\leq C \|\ext g\|_{B^{s_0}_{p,p}(\RRd, w_{\gam_0};X)}\\
    &\leq C\|g\|_{B^{s_0-\frac{\gam_0+1}{p}}_{p,p}(\RR^{d-1};X)}=\|g\|_{B^{s-\frac{\gam+1}{p}}_{p,p}(\RR^{d-1};X)}. 
  \end{align*}
  The identity $\Tr\op \ext=\id$ follows from Theorem \ref{thm:trace0Besov}.
\end{proof}

\subsection{The higher-order trace operator}\label{subsec:mTrace}
In this section, we study the $m$-th order trace operator $\Tr_m=\Tr\op \d_1^m$ for $m\in\NN_1$. Furthermore, we will write $\Tr_0=\Tr$.

We start with higher-order trace spaces for Besov spaces. To this end, we first define the $m$-th order extension operator.

\begin{definition}\label{def:eta}
  Let $\eta_0\in \Cc^{\infty}((-1,1))$ and $\eta\in \Cc^{\infty}((1,\frac{3}{2}))$ such that 
  \begin{equation}\label{eq:eta_0}
    \FF^{-1}\eta_0(0)=\FF^{-1}\eta(0)=1,
  \end{equation}
  where $\FF$ is the one-dimensional Fourier transform. For $m\in\NN_0$ and $D:=-\ii \d_1$ we define
  \begin{equation*}
    \eta_0^m:=(-D)^m\eta_0/m!\quad \text{ and }\quad \eta_j^m:=(-D)^m\eta(2^{-j}\cdot)/m!,\quad j\in \NN_1.
  \end{equation*} 
  Let $(\phi_j)_{j\geq 0}\in \Phi(\RR^{d-1})$. For any $g\in \SS(\RR^{d-1};X)$ we define
  \begin{equation*}
    \ext_mg(x_1,\tilde{x}):=\sum_{j=0}^\infty2^{-j}(\FF^{-1} \eta_j^m)(x_1)(\phi_j\ast g)(\tilde{x}),
  \end{equation*}
  whenever this converges in $\SS'(\RR^d;X)$.
\end{definition}

The motivation behind the definition of this extension operator $\ext_m$ is as follows. By properties of the Fourier transform, we have 
 \begin{equation*}
    (\FF^{-1}\eta_j^m)(x_1)=\frac{2^jx_1^m}{m!}(\FF^{-1}\zeta)(2^jx_1),\qquad j\in \NN_0,\, x_1\in\RR,
  \end{equation*}
  where $\zeta:=\eta $ if $j\geq 1$ and $\zeta:=\eta_0$ if $j=0$. Using the product rule, one can calculate the derivatives $\d_1^m$ of $\ext_m$ to show that $\Tr_m\op \ext_m = \Tr\op\d_1^m \op\ext_m =\id$, i.e., $\ext_m$ is a right inverse to the $m$-th order trace operator. We make this rigorous in Theorems \ref{thm:tracemBesov} and \ref{thm:tracemTriebelLizorkin} for Besov and Triebel--Lizorkin spaces, respectively.\\

We start by showing the continuity of the extension operator $\ext_m$ on Besov spaces. The following lemma is analogous to \cite[Lemma VIII.1.2.7]{Am19}.

\begin{lemma}\label{lem:exis_extm}
  Let $p\in(1,\infty)$, $q\in [1,\infty]$, $m\in \NN_0$, $\gam>-1$, $s>m+\frac{\gam+1}{p}$ and let $X$ be a Banach space. Then for $g\in B^{s-m-\frac{\gam+1}{p}}_{p,q}(\RR^{d-1};X)$ it holds that
  \begin{enumerate}[(i)]
      \item $\ext_m g$ exists in $B^s_{p,q}(\RRd,w_\gam;X)$ if $q\in [1,\infty)$,
      \item $\ext_m g$ exists in $B^t_{p,1}(\RRd, w_\gam;X)$ for any $t<s$ if $q\in [1,\infty]$.
  \end{enumerate}
Moreover, for any $q\in [1,\infty]$, we have that
  $$\ext_m: B^{s-m-\frac{\gam+1}{p}}_{p,q}(\RR^{d-1};X)\to B^s_{p,q}(\RRd,w_{\gam};X)$$
 is continuous.
\end{lemma}

\begin{proof}
\textit{Step 1: preparations.}
  For $m,j\in\NN_0$ and $(\phi_j)_{j\geq 0}\in \Phi(\RR^{d-1})$ let $\eta_j^m$ and $\ext_m$ be as defined in Definition \ref{def:eta}. Let $\FF_{d}$ be the $d$-dimensional Fourier transform. Moreover, let $(\rho_j)_{j\geq 0}\in \Phi(\RR)$ and set $\FF_d\ph_j(x):=(\FF_{1}\rho_j)(x_1)(\FF_{d-1}\phi_j)(\tilde{x})$. For notational convenience we write $\hat{\ph}_j(x)=\hat{\rho}_j(x_1)\hat{\phi}_j(\tilde{x})$. 
  
  Let $g\in B^{s-m-\frac{\gam+1}{p}}_{p,q}(\RR^{d-1};X)$ and for notational convenience we introduce
  \begin{equation*}
    \hat{g}:=\FF_{d-1}g\quad \text{ and }\quad g_j:=(\FF^{-1}_{1}\eta_j^m)(\phi_j\ast g)\quad\text{ for }j\in\NN_0.
  \end{equation*} 
  Note that for $k,j\in\NN_1$ we have
  \begin{align*}
    \FF_d(\ph_k \ast g_j)= \hat{\ph}_k\FF_d g_j=\hat{\rho}_k\hat{\phi}_k \FF_d\big((\FF_{1}^{-1}\eta_j^m)(\FF^{-1}_{d-1}\hat{\phi}_j\hat{g})\big) 
     =  (\hat{\rho}_k \eta_j^m)(\hat{\phi}_k\hat{\phi}_j\hat{g}),
  \end{align*}
  and thus $\ph_k \ast g_j=\FF^{-1}_d\big[(\hat{\rho}_k \eta_j^m)(\hat{\phi}_k\hat{\phi}_j\hat{g})\big]$. Since
 $\supp \eta_j^m\subseteq \{\xi_1\in \RR: 2^j< |\xi_1|<3\cdot2^{j-1}\}\subseteq \supp \hat{\rho}_j$, we have
  \begin{equation}\label{eq:|k-j|>2}
    \ph_k\ast g_j = 0 \quad \text{if }\,|k-j|\geq 2.
  \end{equation}
  Moreover, by properties of the Fourier transform, we obtain for $j\geq 1$
  \begin{align*}
    \FF_{1}^{-1}\eta^m_j  = \frac{2^{-jm}}{m!} \FF^{-1}_{1}\big( ((-D)^m\eta)(2^j\cdot)\big)
    = \frac{2^{-j(m-1)}}{m!}\big(\FF_{1}^{-1}((-D)^m\eta)\big)(2^j\cdot),
  \end{align*}
  and with a substitution $x\mapsto 2^{-j}x$ we obtain that
  \begin{equation}\label{eq:eta_norm}
    \|\FF_{1}^{-1}\eta_j^m\|_{L^p(\RR,w_{\gam})}=2^{j(1-m-\frac{\gam+1}{p})}\|\FF_{1}^{-1}\eta^m\|_{L^p(\RR,w_{\gam})},
  \end{equation}
  where $\eta^m:=(-D)^m\eta/m!$. Define $s_ng:= \sum_{j=0}^n2^{-j}g_j$ for $n\in \NN_0$. Therefore, for $0\leq k\leq n+1$ we obtain from \eqref{eq:|k-j|>2}, Lemma \ref{lem:est_SnAinfty} and \eqref{eq:eta_norm} that
  \begin{equation}\label{eq:est_Sn}
      \begin{aligned}
    2^{ks}\|\ph_k\ast& s_ng \|_{L^p(\RRd,w_{\gam};X)}\\
     & =2^{ks}\Big\|\sum_{j=k-1}^{k+1}\ph_k\ast 2^{-j}g_j\Big\|_{L^p(\RRd,w_{\gam};X)} \\
     & \leq C \sum_{j=-1}^1 2^{(k+j)s-(k+j)}\|g_{k+j}\|_{L^p(\RRd,w_{\gam};X)} \\
     & = C \sum_{j=-1}^1 2^{(k+j)s-(k+j)}\|\FF^{-1}_1\eta_{k+j}^m\|_{L^p(\RR,w_{\gam})}\|\phi_{k+j}\ast g\|_{L^p(\RR^{d-1};X)} \\
     &\leq C \sum_{j=-1}^1 2^{(k+j)(s-m-\frac{\gam+1}{p})}\|\phi_{k+j}\ast g\|_{L^p(\RR^{d-1};X)},
  \end{aligned} 
  \end{equation}
and if $k\geq n+2$ we find
\begin{equation}\label{eq:est_Sn2}
  2^{ks}\|\ph_k\ast s_ng \|_{L^p(\RRdh,w_{\gam};X)}=0.
\end{equation}
  
  \textit{Step 2: existence.} For $g\in B^{s-m-\frac{\gam+1}{p}}_{p,q}(\RR^{d-1};X)$ we prove that
  \begin{enumerate}[(i)]
    \item\label{it:exis1} $\ext_m g$ exists in $B^s_{p,q}(\RR^d,w_{\gam};X)$ if $q\in[1,\infty)$,
    \item\label{it:exis2} $\ext_m g$ exists in $B^{t}_{p,1}(\RRd,w_{\gam};X)$ for any $t<s$ if $q\in[1,\infty]$.
  \end{enumerate}
  If $q\in[1,\infty)$ and $0<\ell<n$, then by \eqref{eq:|k-j|>2}, \eqref{eq:est_Sn} and \eqref{eq:est_Sn2}, we have
  \begin{equation}\label{eq:q<infty}
    \begin{aligned}
    \|s_ng -s_\ell g\|_{B^s_{p,q}(\RRd,w_{\gam};X)} &= \Big(\sum_{k=\ell}^{n+1}\big(2^{ks}  \|\ph_k\ast(s_n g -s_\ell g)\|_{L^p(\RRd,w_{\gam};X)} \big)^q\Big)^{\frac{1}{q}}\\
    &\leq C \Big(\sum_{k=\ell}^{n+1}\big(\sum_{j=-1}^1 2^{(k+j)(s-m-\frac{\gam+1}{p})}\|\phi_{k+j}\ast g\|_{L^p(\RR^{d-1};X)}\big)^q\Big)^{\frac{1}{q}}\\
    &\leq C \Big(\sum_{k=\ell-1}^{n+2}\big(2^{k(s-m-\frac{\gam+1}{p})}\|\phi_{k}\ast g\|_{L^p(\RR^{d-1};X)}\big)^q\Big)^{\frac{1}{q}},
    \end{aligned}
    \end{equation}
    which converges to zero as $\ell,n\to \infty$ since $g\in B^{s-m-\frac{\gam+1}{p}}_{p,q}(\RR^{d-1};X)$. Thus, $(s_ng)_{n\geq 0}$ is a Cauchy sequence in $B^s_{p,q}(\RRd,w_{\gam};X)$. Hence, $\ext_m g=\sum_{j=0}^\infty 2^{-j}g_j=\lim_{n\to\infty} s_n g$ exists in this space and this proves \ref{it:exis1}.
    
    If $q\in[1,\infty]$ and $t<s$, then by \cite[Theorem 14.4.19]{HNVW24} it follows that $$g\in B^{s-m-\frac{\gam+1}{p}}_{p,q}(\RR^{d-1};X)\hookrightarrow B^{t-m-\frac{\gam+1}{p}}_{p,1}(\RR^{d-1};X).$$ Therefore, \ref{it:exis1} implies that $\ext_m g$ exists in $B^t_{p,1}(\RR^d,w_{\gam};X)$ and this proves \ref{it:exis2}. 
  
  \textit{Step 3: continuity.} Let $g\in B^{s-m-\frac{\gam+1}{p}}_{p,q}(\RR^{d-1};X)$ with $q\in[1,\infty]$. Take $k\in\NN_0$ and $0<\ell< n$, then the estimate
    \begin{equation*}
      2^{kt}\|\ph_k\ast (s_ng-s_\ell g)\|_{L^p(\RRd,w_{\gam};X)}\leq C\|s_n g- s_\ell g\|_{B^{t}_{p,1}(\RRd,w_{\gam};X)},
    \end{equation*}
    together with \eqref{eq:q<infty} implies that $(\ph_k\ast s_ng)_{n\geq 0}$ is a Cauchy sequence in $L^p(\RRd, w_{\gam};X)$ for all $k\in\NN_0$. Therefore, its limit $w_k=\lim_{n\to\infty} \ph_k\ast s_ng$ exists in $L^p(\RRd, w_{\gam};X)$ and thus in $\SS'(\RRd;X)$ as well. Since $B^t_{p,1}(\RRd,w_{\gam};X)\hookrightarrow \SS'(\RRd;X)$, Step 2\ref{it:exis2} implies that $s_n g\to \ext_m g$, and thus also $\ph_k\ast s_ng\to \ph_k\ast \ext_m g$, both in $\SS'(\RR;X)$ as $n\to\infty$.   
     We conclude that $w_k=\ph_k\ast \ext_m g$ and 
     \begin{equation}\label{eq:convphk}
       \ph_k\ast s_ng \to \ph_k\ast \ext_m g\quad \text{ in }L^p(\RRd,w_{\gam};X) \,\text{ as }\,n\to\infty.
     \end{equation}
  The convergence in \eqref{eq:convphk} shows that we can let $n\to \infty$ in \eqref{eq:est_Sn} to obtain for any $k\in\NN_0$
  \begin{equation*}
    2^{ks}\|\ph_k\ast \ext_m g\|_{L^p(\RRd, w_{\gam};X)}\leq C \sum_{j=-1}^1 2^{(k+j)(s-m-\frac{\gam+1}{p})}\|\phi_{k+j}\ast g\|_{L^p(\RR^{d-1};X)},
  \end{equation*}
   and therefore
   \begin{align*}
     \|\ext_m g\|_{B^{s}_{p,q}(\RRd,w_{\gam};X)} & \leq C\Big\|\big(2^{k(s-m-\frac{\gam+1}{p})}\|\phi_k\ast g\|_{L^p(\RR^{d-1};X)}\big)_{k\geq 0}\Big\|_{\ell^q} \\
     & \leq C\|g\|_{B^{s-m-\frac{\gam+1}{p}}_{p,q}(\RR^{d-1};X)}.
   \end{align*}
  To conclude, $\ext_m g\in B^{s}_{p,q}(\RRd,w_{\gam};X)$ and the lemma is proved. 
\end{proof}

Using Lemma \ref{lem:exis_extm} we can now determine the higher-order trace spaces of weighted Besov spaces and show that the corresponding extension operator is indeed given as in Definition \ref{def:eta}. We follow the arguments from \cite[Theorem VIII.1.2.1]{Am19}.

\begin{theorem}\label{thm:tracemBesov}
  Let $p\in(1,\infty)$, $q\in [1,\infty]$, $m\in \NN_0$, $\gam>-1$, $s>m+\frac{\gam+1}{p}$ and let $X$ be a Banach space. Then
  \begin{equation*}
    \Tr_m : B^s_{p,q}(\RRd, w_{\gam};X)\to B^{s-m-\frac{\gam+1}{p}}_{p,q}(\RR^{d-1};X)
  \end{equation*}
  is a continuous and surjective operator.  Moreover, the extension operator $\ext_m$ from Definition \ref{def:eta} defines a continuous right inverse of $\Tr_m$ which is independent of $s,p,q,\gam$ and $X$. For any $0\leq j<m$, we have $\Tr_j\op \ext_m=0$.
\end{theorem}
\begin{proof}
  \textit{Step 1: trace operator.} Note that $\d_1^m: B^s_{p,q}(\RRd, w_{\gam};X)\to B^{s-m}_{p,q}(\RR^d, w_{\gam};X)$ is continuous, see \cite[Proposition 3.10]{MV12}.
  Therefore, since $s-m>\frac{\gam+1}{p}$ it follows from Theorem \ref{thm:trace0Besov} that
  \begin{equation*}
    \Tr_m=\Tr\op\d_1^m: B^{s}_{p,q}(\RR^d,w_{\gam};X)\to B^{s-m-\frac{\gam+1}{p}}_{p,q}(\RR^{d-1};X)
  \end{equation*}
  is continuous. 
  
  \textit{Step 2: extension operator. } Let $\ext_m$ be the extension operator from Definition \ref{def:eta}. The continuity follows from Lemma \ref{lem:exis_extm}. It remains to prove that $\ext_m$ is the right inverse of $\Tr_m$ and that
  \begin{equation}\label{eq:trext=delta}
    \Tr_j(\ext_m g)=\delta_{jm}g,\qquad 0\leq j\leq m, \,\, g\in B^{s-m-\frac{\gam+1}{p}}_{p,q}(\RR^{d-1};X),
  \end{equation}
 where $\delta_{jm}$ is the Kronecker delta.
 Note that from Definition \ref{def:eta} and properties of the Fourier transform it follows that for $n\geq 0$
  \begin{equation}\label{eq:Finveta}
    (\FF^{-1}\eta_n^m)(x_1)=\frac{2^nx_1^m}{m!}(\FF^{-1}\zeta)(2^nx_1),\qquad x_1\in\RR,
  \end{equation}
  where $\FF$ is the one-dimensional Fourier transform and $\zeta:=\eta $ if $n\geq 1$ and $\zeta:=\eta_0$ if $n=0$.  
  Therefore, if $g\in \SS(\RR^{d-1};X)$, then using $\Tr_j =\Tr\op \d_1^j$, the product rule and \eqref{eq:eta_0}, gives
  \begin{align}\label{eq:Trm=0}
    \Tr_j \big(2^{-n}(\FF^{-1}\eta_n^m)(\phi_n\ast g)\big) & = \delta_{jm}(\phi_n\ast g), \qquad 0\leq j\leq m,\, n\geq 0.
  \end{align}
  Therefore, by continuity of the trace operator (Step 1) and \eqref{eq:Trm=0}, we find for $0\leq j\leq m$
  \begin{equation*}
    \Tr_j(\ext_mg)=\sum_{n=0}^\infty\Tr_j\big(2^{-n}(\FF^{-1}\eta_n^m)(\phi_n\ast g)\big)=\delta_{jm}\sum_{n=0}^\infty (\phi_n\ast g)=\delta_{jm}g,
  \end{equation*}
  using properties of the Littlewood--Paley sequence in the last identity. By density (see \cite[Lemma 3.8]{MV12}) this extends to all $g\in B^{s-m-\frac{\gam+1}{p}}_{p,q}(\RR^{d-1};X)$ with $q<\infty$. If $q=\infty$, then we have $B^{s-m-\frac{\gam+1}{p}}_{p,\infty}(\RR^{d-1};X)\hookrightarrow B^{s-m-\frac{\gam+1}{p}-\eps}_{p,1}(\RR^{d-1};X)$ for all $\eps>0$ (see \cite[Theorem 14.4.19]{HNVW24}). This proves \eqref{eq:trext=delta}.
  \end{proof}

We continue with the characterisation of higher-order trace spaces for Triebel--Lizorkin spaces.  
\begin{theorem}\label{thm:tracemTriebelLizorkin}
  Let $p\in(1,\infty)$, $q\in [1,\infty]$, $m\in \NN_0$, $\gam>-1$, $s>m+\frac{\gam+1}{p}$ and let $X$ be a Banach space. Then
  \begin{equation*}
    \Tr_m : F^s_{p,q}(\RRd, w_{\gam};X)\to B^{s-m-\frac{\gam+1}{p}}_{p,p}(\RR^{d-1};X)
  \end{equation*}
  is a continuous and surjective operator.  Moreover, the extension operator $\ext_m$ from Definition \ref{def:eta} defines a continuous right inverse of $\Tr_m$ which is independent of $s,p,q,\gam$ and $X$. For any $0\leq j<m$, we have $\Tr_j\op \ext_m=0$. 
\end{theorem}
\begin{proof}
  The proof is similar to the proof of Theorem \ref{thm:trace0TriebelLizorkin} using Theorem \ref{thm:tracemBesov} instead of Theorem \ref{thm:trace0Besov}.
\end{proof}
%

\section{Trace spaces of Bessel potential and Sobolev spaces}\label{sec:tracesHW}
Let $p\in(1,\infty)$, $X$ a Banach space and let $\OO\subseteq\RRd$ be open. Recall that the power weight  $w_\gam^{\d\OO}(x):=\operatorname{dist}(x, \d\OO)^\gam$ for $x\in \OO$ and $\gam>-1$. In the rest of this paper, we mainly consider the following two weighted spaces:
\begin{enumerate}[(i)]
  \item Bessel potential spaces $H^{s,p}(\OO, w^{\d\OO}_{\gam};X)$ with $s\in \RR$ and $\gam\in(-1,p-1)$,
  \item Sobolev spaces $W^{k,p}(\OO, w^{\d\OO}_{\gam};X)$ with $k\in\NN_0$ and $\gam\in(-1,\infty)\setminus\{jp-1:j\in\NN_1\}$.
\end{enumerate}
To finally prove our main results in Section \ref{sec:compl_intp} about complex interpolation for these weighted spaces on bounded domains, it suffices to consider the half-space $\OO=\RRdh$ by localisation arguments. In this case, we write $w_{\gam}(x)=w_{\gam}^{\smash{\d\RRdh}}(x)=|x_1|^\gam$ for $x=(x_1, \tilde{x})\in \RR^d$ as before.\\

The outline of this section is as follows. We start with the trace spaces for the above-mentioned weighted Bessel potential and Sobolev spaces on $\RRd$ and $\RRdh$ in Section \ref{subsec:tracemSob}. Afterwards, we apply the trace theorems to prove some density results in Section \ref{subsec:density}.

\subsection{Trace spaces of weighted spaces on the half-space}\label{subsec:tracemSob}
Using the results from Section \ref{subsec:mTrace} about weighted Besov and Triebel--Lizorkin spaces, we will now characterise higher-order trace spaces of weighted Bessel potential and Sobolev spaces. 
The methods we will employ depend on the weight exponent $\gam$. For $\gam\in (-1, p-1)$, i.e., $w_\gam$ is a Muckenhoupt weight, we can use standard techniques relying on reflection operators to transfer results from $\RRd$ to $\RRdh$. For $\gam>p-1$ such reflection techniques do not work since weighted Bessel potential and Sobolev spaces on $\RR^d$ cannot be defined as usual if $\gam>p-1$ (cf. Remark \ref{rem:L1loc}). Nonetheless, the definition of a Sobolev space on $\RRdh$ does make sense for $\gam>p-1$ and the main difficulty is to establish the results for this regime of weight exponents. \\

We first consider the case $\gam\in(-1,p-1)$, where we can use the sandwiching embeddings \eqref{eq:embFHW} to obtain that, e.g., 
\begin{equation*}
    \Tr_m:H^{s,p}(\RRd,w_{\gam};X)\to B_{p,p}^{s-m-\frac{\gam+1}{p}}(\RR^{d-1};X), \qquad s> m+\tfrac{\gam+1}{p},
  \end{equation*}
is a continuous and surjective operator, see Theorem \ref{thm:tracemHW_RRdh} below. If $f_1,f_2\in H^{s,p}(\RRd,w_{\gam};X)$ satisfy $f_1|_{\RRdh}=f_2|_{\RRdh}$, then we have $\Tr_jf_1=\Tr_j f_2$ for all $j\in \{0,\dots, m\}$. This can be proved similarly as in \cite[Proposition 6.3(2)]{LMV17}. This implies that the trace operator $\Tr_m$ on $\RRd$ gives rise to a well-defined trace operator on $\RRdh$, which we denote by $\Tr_m$ as well. Moreover, if $f$ is continuous on $[0,\delta)\times \RR^{d-1}$ for some $\delta>0$, then the trace operator coincides with the restriction of $f$ to $\{(0,\tilde{x}):\tilde{x}\in \RR^{d-1}\}$.

For weighted Bessel potential and Sobolev spaces on $\RRd$ and $\RRdh$ with $\gam\in (-1,p-1)$, we obtain the following trace results.

\begin{theorem}\label{thm:tracemHW_RRdh}
  Let $p\in(1,\infty)$, $m\in \NN_0$, $\gam\in(-1,p-1)$, $\OO\in\{\RR^d, \RRdh\}$ and let $X$ be a Banach space. 
  \begin{enumerate}[(i)]
    \item\label{it:cor:tracemHW_RRdh_H} If $s>0$ satisfies $s>m+\frac{\gam+1}{p}$, then 
      \begin{equation*}
    \Tr_m:H^{s,p}(\OO,w_{\gam};X)\to B_{p,p}^{s-m-\frac{\gam+1}{p}}(\RR^{d-1};X)
  \end{equation*}
   is a continuous and surjective operator. 
    \item \label{it:cor:tracemHW_RRdh_W} If $k\in \NN_1$ satisfies $k>m+\frac{\gam+1}{p}$, then 
      \begin{equation*}
    \Tr_m:W^{k,p}(\OO,w_{\gam};X)\to B_{p,p}^{k-m-\frac{\gam+1}{p}}(\RR^{d-1};X)
  \end{equation*}
   is a continuous and surjective operator.
  \end{enumerate}
  In both cases, there exists a continuous right inverse $\ext_m$ of $\Tr_m$ which is independent of $k,s,p,\gam$ and $X$. For any $0\leq j<m$, we have $\Tr_j\op \ext_m=0$.
\end{theorem}

Before we can prove Theorem \ref{thm:tracemHW_RRdh}, we need the following trace result for Besov spaces on the half-space. This proposition is used to prove Theorem \ref{thm:tracemHW_RRdh} without a $\UMD$ condition for the Banach space $X$, see also Remark \ref{rem:UMDcondrefl} below.
\begin{proposition}\label{prop:tracemBesovRRdh}
    Let $p\in(1,\infty)$, $q\in [1,\infty]$, $m\in \NN_0$, $\gam\in (-1,p-1)$, $s>m+\frac{\gam+1}{p}$ and let $X$ be a Banach space. Then
  \begin{equation*}
    \Tr_m : B^s_{p,q}(\RRdh, w_{\gam};X)\to B^{s-m-\frac{\gam+1}{p}}_{p,q}(\RR^{d-1};X)
  \end{equation*}
  is a continuous and surjective operator. Moreover, there exists a continuous right inverse $\ext_m$ of $\Tr_m$ which is independent of $k,s,p,\gam$ and $X$. For any $0\leq j<m$, we have $\Tr_j\op \ext_m=0$.
\end{proposition}
\begin{proof}
     The result follows from Theorem \ref{thm:tracemBesov} if we have a reflection operator between Besov spaces on $\RRdh$ and $\RRd$. To obtain this reflection operator, note that for any $\tilde{s}>0$, $\theta\in (0,1)$, $s_0,s_1\in \NN_0$ such that $s_0<s_1$ and $\tilde{s}=(1-\theta)s_0+\theta s_1$, we have
     \begin{equation}\label{eq:B=(WW)}
         (W^{s_0,p}(\RRdh, w_{\gam};X), W^{s_1,p}(\RRdh, w_{\gam};X))_{\theta,q} = B^{\tilde{s}}_{p,q}(\RRdh, w_{\gam};X).
     \end{equation}
This identity follows from \cite[Proposition 6.1]{MV12} (using $\gam\in (-1,p-1)$) and \cite[Lemma 5.4]{LMV17} (which holds for real interpolation as well, see \cite[Section 1.2.4]{Tr78}). Now, the higher-order reflection operator on Sobolev spaces from \cite[Lemma 5.1]{LMV17}, is also bounded on Besov spaces by \eqref{eq:B=(WW)}.
\end{proof}
\begin{remark}
    Using Theorem \ref{thm:tracemTriebelLizorkin} a similar result as in Proposition \ref{prop:tracemBesovRRdh} holds for the trace on Triebel--Lizorkin spaces on $\RRdh$. For Triebel--Lizorkin spaces the construction of a reflection operator is more involved, since Triebel--Lizorkin spaces arise as interpolation spaces of Sobolev spaces using the $\ell^q$-interpolation method, see \cite{Ku15, LL24}. Alternatively, one can also redo \cite[Theorem 2.9.2]{Tr83} using the weighted multiplier result \cite[Theorem 1.3]{MV14}. 
\end{remark}

We will now provide the proof of Theorem \ref{thm:tracemHW_RRdh}.

\begin{proof}[Proof of Theorem \ref{thm:tracemHW_RRdh}]
  \textit{Step 1: $\OO=\RR^d$.} This follows immediately from Theorem \ref{thm:tracemTriebelLizorkin} and the embeddings \eqref{eq:embFHF} and \eqref{eq:embFWF}. Note that these embeddings are only valid for $w\in A_p(\RRd)$, i.e., $\gam\in (-1, p-1)$.
  
  \textit{Step 2: $\OO=\RRdh$.} Statement \ref{it:cor:tracemHW_RRdh_W} with $\OO=\RRdh$ follows from Step 1 and a higher-order reflection argument, see \cite[Lemma 5.1]{LMV17}.

It remains to prove statement \ref{it:cor:tracemHW_RRdh_H} with $\OO=\RRdh$.
 Let $\eps>0$ be small and set $\gam_0:=\gam+\eps$, $\gam_1:=\gam-\eps$, $s_0:=s+\eps/p$ and $s_1:=s-\eps/p$. From Theorem \ref{thm:embBBFF}, \eqref{eq:embFHF}  and \eqref{eq:embBFB} (which also hold on $\RRdh$ since the spaces on $\RRdh$ are defined as factor spaces), we have the chain of embeddings
 \begin{equation}\label{eq:embRRdh}
    \begin{aligned}
   B^{s_0}_{p,p}(\RRdh, w_{\gam_0};X) & =F^{s_0}_{p,p}(\RRdh, w_{\gam_0};X)\hookrightarrow F^{s}_{p,1}(\RRdh,w_{\gam};X)\hookrightarrow H^{s,p}(\RRdh,w_{\gam};X) \\
  &\hookrightarrow F^{s}_{p,\infty}(\RRdh, w_{\gam};X)\hookrightarrow F^{s_1}_{p,p}(\RRdh, w_{\gam_1};X)=B^{s_1}_{p,p}(\RRdh, w_{\gam_1};X).
 \end{aligned}
 \end{equation}
From Proposition \ref{prop:tracemBesovRRdh} and \eqref{eq:embRRdh} it follows that for $f\in H^{s,p}(\RRdh, w_{\gam};X)$
\begin{align*}
  \|\Tr_m f\|_{B^{s-m-\frac{\gam+1}{p}}_{p,p}(\RR^{d-1};X)}& =\|\Tr_m f\|_{B^{s_1-m-\frac{\gam_1+1}{p}}_{p,p}(\RR^{d-1};X)}\\
  &\leq C\|f\|_{B^{s_1}_{p,p}(\RRdh, w_{\gam_1};X)}\leq C\|f\|_{H^{s,p}(\RRdh, w_{\gam};X)}
\end{align*}
and for $g\in B^{s-m-\frac{\gam+1}{p}}_{p,p}(\RR^{d-1};X)$
\begin{align*}
  \|\ext_m g\|_{H^{s,p}(\RRdh, w_{\gam};X)} & \leq C \|\ext_m g\|_{B^{s_0}_{p,p}(\RRdh, w_{\gam_0};X)} \\
   & \leq C \|g\|_{B^{s_0-m-\frac{\gam_0+1}{p}}_{p,p}(\RR^{d-1};X)}= \|g\|_{B^{s-m-\frac{\gam+1}{p}}_{p,p}(\RR^{d-1};X)}.
\end{align*}
This finishes the proof.
\end{proof}

\begin{remark}\label{rem:UMDcondrefl}
Alternatively, in the proof of Theorem \ref{thm:tracemHW_RRdh}\ref{it:cor:tracemHW_RRdh_H} with $\OO=\RRdh$ above, we could have used a reflection operator from $ H^{s,p}(\RRdh,w_{\gam};X)$ into $H^{s,p}(\RRd,w_{\gam};X)$. Such a reflection operator can be obtained using \cite[Lemma 5.1]{LMV17} and complex interpolation. However, this requires $X$ to be $\UMD$, see \cite[Proposition 5.6]{LMV17}. Instead, we did the reflection argument at the level of Besov spaces in Proposition \ref{prop:tracemBesovRRdh}, where no $\UMD$ condition was required. 
\end{remark}

\begin{remark}\label{rem:map_prop}
  For $m\in\NN_0$ and $\OO\in\{\RR^d,\RRdh\}$ it is clear that the trace operator $\Tr_m$ maps $\Cc^{\infty}(\overline{\OO};X)$ into $\Cc^{\infty}(\RR^{d-1};X)$. Moreover, the extension operator $\ext_m$ maps $\Cc^{\infty}(\RR^{d-1};X)$ into $C^{\infty}(\overline{\OO};X)$. Indeed, let $g \in \Cc^{\infty}(\RR^{d-1} ;X)$, then $g\in B_{p,p}^{s-m-\frac{1}{p}}(\RR^{d-1};X)$ and $\ext_m g\in H^{s,p}(\OO;X)$ for all $s>m+\frac{1}{p}$. The Sobolev embedding implies that $\ext_m g\in C^{\infty}(\overline{\OO};X)$.
\end{remark}

We extend the result for Sobolev spaces in Theorem \ref{thm:tracemHW_RRdh} from $\gam\in(-1,p-1)$ to  $\gam\in(-1,\infty)\setminus\{jp-1:j\in\NN_1\}$.

\begin{theorem}\label{thm:tracemW_RRdh_allweights}
    Let $p\in(1,\infty)$, $m\in \NN_0$, $k\in\NN_1$, $\gam\in(-1,\infty)\setminus\{jp-1:j\in\NN_1\}$ such that $k>m+\frac{\gam+1}{p}$ and let $X$ be a Banach space. Then 
   \begin{equation*}
    \Tr_m:W^{k,p}(\RRdh,w_{\gam};X)\to B_{p,p}^{k-m-\frac{\gam+1}{p}}(\RR^{d-1};X)
  \end{equation*}
   is a continuous and surjective operator. Moreover, there exists a continuous right inverse $\ext_m$ of $\Tr_m$ which is independent of $k,p,\gam$ and $X$. For any $0\leq j<m$, we have $\Tr_j\op \ext_m=0$.
\end{theorem}
\begin{proof}
  Let $j\in\NN_0$ be such that $\gam\in (jp-1,(j+1)p-1)$. The case $j=0$ follows from Theorem \ref{thm:tracemHW_RRdh}\ref{it:cor:tracemHW_RRdh_W}, so we can assume $j\geq 1$. 
  
  \textit{Step 1: trace operator.} Let $f\in W^{k,p}(\RRdh,w_{\gam};X)$, then
  \begin{align*}
    \|\Tr_m f\|_{B^{k-m-\frac{\gam+1}{p}}_{p,p}(\RR^{d-1};X)} &= \|\Tr_m f\|_{B^{k-j-m-\frac{\gam-jp+1}{p}}_{p,p}(\RR^{d-1};X)}\\
    &\leq C\|f\|_{W^{k-j,p}(\RRdh, w_{\gam-jp};X)}\leq C\|f\|_{W^{k,p}(\RRdh,w_{\gam};X)},
  \end{align*}
  since $\gam-jp\in(-1,p-1)$ and $k-j>m\geq 0$ so that Theorem \ref{thm:tracemHW_RRdh}\ref{it:cor:tracemHW_RRdh_W} applies. The last estimate follows from Hardy's inequality, see Lemma \ref{lem:Hardy}.
  
  \textit{Step 2: extension operator.} Let $\alpha=(\alpha_1,\tilde{\alpha})\in\NN_0\times \NN_0^{d-1}$ with $|\alpha|\leq k$ and let $g\in B^{k-m-\frac{\gam+1}{p}}_{p,p}(\RR^{d-1};X)$. Then estimating the norm on $\RRdh$ by the norm on $\RRd$, using \eqref{eq:embFL}, Theorem \ref{thm:embBBFF} and \eqref{eq:embBFB}, gives
  \begin{equation}\label{eq:Sob_est1}
        \begin{aligned}
  \|\d^{\alpha}\ext_m g\|_{L^p(\RRdh, w_{\gam};X)} &\leq \|\d^{\alpha}\ext_m g\|_{L^p(\RRd, w_{\gam};X)}\\
  &\leq C \|\d^{\alpha}\ext_m g\|_{F^{0}_{p,1}(\RRd, w_{\gam};X)}\\
 & \leq C \|\d^{\alpha}\ext_m g\|_{F^{s_0}_{p,p}(\RRd, w_{\gam_0};X)}\\
 & \leq C \|\d^{\alpha}\ext_m g\|_{B^{s_0}_{p,p}(\RRd, w_{\gam_0};X)},
\end{aligned}
  \end{equation}
where $\gam_0>\gam$ and $s_0>0$ are such that $s_0-\frac{\gam_0}{p}=-\frac{\gam}{p}$. 
With a similar computation as in the proof of Lemma \ref{lem:exis_extm} (see \eqref{eq:est_Sn}), it follows that for $\ell\in \NN_0$, we have
\begin{equation}\label{eq:Sob_est2}
\begin{aligned}
       & \; 2^{\ell s_0}\|\ph_\ell \ast \d^\alpha \ext_m g\|_{L^p(\RRd, w_{\gam_0};X)}\\
       \leq&\; C\sum_{i=-1}^1 2^{\ell(s_0-1)}\| \d_1^{\alpha_1}(\FF^{-1}\eta^m_{\ell+i})\|_{L^p(\RR, w_{\gam_0})}\|\d^{\tilde{\alpha}}(\phi_{\ell+i}\ast g)\|_{L^p(\RR^{d-1};X)},
\end{aligned}
\end{equation}
where $\FF$ is the one-dimensional Fourier transform and $(\phi_\ell)_{\ell\geq 0}\in \Phi(\RR^{d-1})$. 
Taking the $\ell^p$-norm of \eqref{eq:Sob_est2} and combining this with \eqref{eq:Sob_est1} gives
\begin{equation}\label{eq:proofextW_est1}
    \begin{aligned}
 & \;\|\d^{\alpha}\ext_m g\|_{L^p(\RRdh, w_{\gam};X)} \\
  \leq &\;C\Big(\sum_{\ell=0}^{\infty}\sum_{i=-1}^1 2^{\ell p(s_0-1) }\| \d_1^{\alpha_1}(\FF^{-1}\eta^m_{\ell+i})\|^p_{L^p(\RR, w_{\gam_0})}\|\d^{\tilde{\alpha}}(\phi_{\ell+i}\ast g)\|^p_{L^p(\RR^{d-1};X)}\Big)^{\frac{1}{p}}.
\end{aligned}
\end{equation}
We estimate the two $L^p$-norms in the summation above separately. We start with estimating the $L^p(\RR, w_{\gam_0})$-norm. Let $\ell\in \NN_0$ and $i\in \{-1,0,1\}$. By \eqref{eq:Finveta} and the product rule, we compute
\begin{equation*}
\begin{aligned}
    \d_{x_1}^{\alpha_1}\big(\FF^{-1}\eta_{\ell+i}^m\big)(x_1)&= \frac{2^{\ell+i}}{m!}\,\d_{x_1}^{\alpha_1} \big(x_1^m (\FF^{-1}\zeta )(2^{\ell+i}x_1)\big)\\
    &=  \frac{2^{\ell+i}}{m!} \sum_{n=0}^{\alpha_1} c_n\, x_1^{m-n}\, 2^{(\ell+i)(\alpha_1-n)}(\FF^{-1}\zeta)^{(\alpha_1-n)}(2^{\ell+i}x_1),
\end{aligned}
\end{equation*}
where $c_n=\binom{\alpha_1}{n}\frac{m!}{(m-n)!}$ if $n\leq m$ and zero otherwise.
Taking norms and applying Hardy's inequality (Lemma \ref{lem:Hardy}), yields
\begin{equation}\label{eq:proofextW_est2}
    \begin{aligned}
     \| \d_1^{\alpha_1}\big(\FF^{-1}&\eta^m_{\ell+i}\big)\|^p_{L^p(\RR, w_{\gam_0})}\\
     &  
  \leq C\, 2^{\ell p} \sum_{n=0}^{\alpha_1}2^{(\ell+i)(\alpha_1-n)p} \big\| x_1\mapsto x_1^{m-n} (\FF^{-1}\zeta)^{(\alpha_1-n)}(2^{\ell+i}x_1)\big\|^p_{L^p(\RR, w_{\gam_0})}\\
  &\leq C\, 2^{\ell p(1+\alpha_1)}\big\|x_1\mapsto x_1^{m} (\FF^{-1}\zeta)^{(\alpha_1)}(2^{\ell+i}x_1)\big\|^p_{L^p(\RR, w_{\gam_0})}\\
  &\leq C\, 2^{\ell p (1+ \alpha_1-m -\frac{\gam_0+1}{p})}.
\end{aligned}
\end{equation}
For the $L^p(\RR^{d-1};X)$-norm we estimate with the notation $T_n g:=\phi_n\ast g$
\begin{equation}\label{eq:proofextW_est3}
     \begin{aligned}
   \|\d^{\tilde{\alpha}}(\phi_{\ell+i}\ast g)\|^p_{L^p(\RR^{d-1};X)}& \leq \sum_{n=-1}^1 \|\d^{\tilde{\alpha}}(\phi_{\ell+i}\ast T_{\ell+i+n}g)\|^p_{L^p(\RR^{d-1};X)}\\
  &\leq C\, 2^{\ell p|\tilde{\alpha}|}\sum_{n=-1}^1\|T_{\ell+i+n}g\|^p_{L^p(\RR^{d-1};X)}.
\end{aligned}
\end{equation}
Combining the estimates \eqref{eq:proofextW_est1}, \eqref{eq:proofextW_est2} and \eqref{eq:proofextW_est3} yields
\begin{align*}
  \|\d^{\alpha} \ext_m g\|_{L^p(\RRd, w_{\gam};X)} &\leq C \Big(\sum_{\ell=0}^{\infty}2^{\ell p (s_0+|\alpha|-m-\frac{\gam_0+1}{p})}\| T_{\ell} g \|^p_{L^p(\RR^{d-1};X)}\Big)^{\frac{1}{p}}\\
  &\leq C\Big(\sum_{\ell=0}^{\infty}2^{\ell p(k-m-\frac{\gam+1}{p})}\|\phi_\ell\ast g\|^p_{L^p(\RR^{d-1};X)}\Big)^{\frac{1}{p}}\\
  &=C\|g\|_{B^{k-m-\frac{\gam+1}{p}}_{p,p}(\RR^{d-1};X)},
\end{align*}
where the constants are independent of $g$ and $\ell$ and we have used that
\begin{equation*}
  s_0+|\alpha|-m-\frac{\gam_0+1}{p}=|\alpha|-m-\frac{\gam+1}{p}\leq k-m-\frac{\gam+1}{p}.
\end{equation*}
This proves the continuity for the extension operator.
\end{proof}

To close this section we state a related result for the vector of traces given by
\begin{equation*}
  \bTr_m:=(\Tr_0,\dots, \Tr_m),\qquad m\in\NN_0.
\end{equation*}
We follow the proof of \cite[Theorem VIII.1.3.2]{Am19} to determine the trace space for $\bTr_m$ from the trace space for $\Tr_m$. A similar result is contained in \cite[Section 2.9.2]{Tr78} for weighted scalar-valued Sobolev spaces.
\begin{proposition}\label{prop:bTr_RRdh}
 Let $p\in(1,\infty)$, $m\in \NN_0$ and let $X$ be a Banach space. 
  \begin{enumerate}[(i)]
    \item\label{it:prop:bTr_RRdh_H} If $s>0$, $\gam\in (-1,p-1)$, $s>m+\frac{\gam+1}{p}$ and $\OO\in\{\RRd,\RRdh\}$, then 
      \begin{equation*}
    \bTr_m:H^{s,p}(\OO,w_{\gam};X)\to \prod_{j=0}^m B_{p,p}^{s-j-\frac{\gam+1}{p}}(\RR^{d-1};X)
  \end{equation*}
   is a continuous and surjective operator. 
    \item\label{it:prop:bTr_RRdh_W} If $k\in \NN_1$, $\gam\in(-1,\infty)\setminus\{jp-1:j\in\NN_1\}$ and $k>m+\frac{\gam+1}{p}$, then 
      \begin{equation*}
    \bTr_m:W^{k,p}(\RRdh,w_{\gam};X)\to \prod_{j=0}^m B_{p,p}^{k-j-\frac{\gam+1}{p}}(\RR^{d-1};X)
  \end{equation*}
   is a continuous and surjective operator.
  \end{enumerate}
  In both cases, there exists a continuous right inverse $\bext_m$ of $\bTr_m$ which is independent of $k,s,p,\gam$ and $X$. 
\end{proposition}
\begin{proof} 
The proof is similar to \cite[Theorem VIII.1.3.2]{Am19} and for the convenience of the reader, we sketch the proof.

\textit{Step 1: proof of \ref{it:prop:bTr_RRdh_H}.} For $j\in\{0,\dots, m\}$ let $\Tr_j$ and $\ext_j$ be as in Theorem \ref{thm:tracemHW_RRdh}\ref{it:cor:tracemHW_RRdh_H}. The continuity of $\bTr_m$ follows from Theorem \ref{thm:tracemHW_RRdh}\ref{it:cor:tracemHW_RRdh_H}. We continue with constructing an extension operator $\bext_m$. Let 
  \begin{equation*}
    (g_0,\dots,g_m)\in \prod_{j=0}^m B_{p,p}^{s-j-\frac{\gam+1}{p}}(\RR^{d-1};X)
  \end{equation*}
  and define $f_0:=\ext_0 g_0$. For $j\in\{1,\dots, m\}$ we recursively define 
  \begin{equation*}\label{eq:def_fj}
    f_j:=f_{j-1}+\ext_j(g_j-\Tr_j f_{j-1}) = H_jf_{j-1}+ \ext_j g_{j},
  \end{equation*}
   where $H_j:=1-\ext_j\op \Tr_j$ is a bounded operator on $H^{s,p}(\OO,w_{\gam};X)$ by Theorem \ref{thm:tracemHW_RRdh}\ref{it:cor:tracemHW_RRdh_H}. It is straightforward to verify that
   \begin{equation}\label{eq:f_m}
     f_m=\ext_m g_m+ \sum_{j=0}^{m-1}\Big(\prod_{n=j+1}^m H_n\Big)\ext_jg_j.
   \end{equation} 
   Moreover, note that from Theorem \ref{thm:tracemHW_RRdh}\ref{it:cor:tracemHW_RRdh_H} we have for $j\in\{0,\dots, m\}$
   \begin{equation}\label{eq:traceprop}
        \begin{aligned}
     \Tr_j\op H_j  & = \Tr_j\op\,(1-\ext_j\op \Tr_j)=\Tr_j-\Tr_j=0\quad \text{ and } \\
     \Tr_{j_1}\op H_j& =  \Tr_{j_1}\op\, (1-\ext_j\op \Tr_j)=\Tr_{j_1}\qquad\text{ for }j_1\in\{0,\dots, j-1\}. 
   \end{aligned}
   \end{equation}
From \eqref{eq:f_m} and \eqref{eq:traceprop} it follows that $\Tr_j f_m =g_j$ for $j\in\{0,\dots,m\}$. Hence, we can define $\bext_m(g_0,\dots, g_m):= f_m$ and it is straightforward to verify that this operator satisfies the desired properties.

\textit{Step 2: proof of \ref{it:prop:bTr_RRdh_W}.}
This case is similar to Step 1 using Theorem \ref{thm:tracemW_RRdh_allweights} instead of Theorem \ref{thm:tracemHW_RRdh}\ref{it:cor:tracemHW_RRdh_H}.
\end{proof}

\subsection{Density results}\label{subsec:density}
As an application of the trace theorems in the previous section, we prove density of certain classes of test functions in weighted spaces with zero boundary conditions. 

\subsubsection{Bessel potential spaces}\label{subsec:densityBessel} Let $p\in(1,\infty)$, $s\in \RR$, $\gam\in(-1,p-1)$, $\OO\in \{\RR^d, \RRdh\}$ and let $X$ be a Banach space. Then we define
\begin{equation*}
  H^{s,p}_0(\OO,w_{\gam};X):=\Big\{f\in H^{s,p}(\OO,w_{\gam};X): \Tr(\d^{\alpha}f)=0\text{ if }s-|\alpha|>\tfrac{\gam+1}{p}\Big\}.
\end{equation*}
All the traces in the above definition are well defined by Theorem \ref{thm:tracemHW_RRdh}. We note that if $s\leq \frac{\gam+1}{p}$, then
\begin{equation}\label{eq:H=H0}
    H^{s,p}_0(\OO,w_{\gam} ;X) = H^{s,p}(\OO,w_{\gam};X).
\end{equation}
Moreover, we set $\RR_\bullet^d:=\RR^d\setminus\{(0,\tilde{x}): \tilde{x}\in \RR^{d-1}\}$.\\

To prove a density result for $H^{s,p}_0$, we first record the following proposition about the boundedness of the pointwise multiplication operator $\ind_{\RRdh}$. The following proposition is an extension of \cite[Proposition 6.2]{LMV17}.
\begin{proposition}\label{prop:LMV6.2_ext}
  Let $p\in(1,\infty)$, $\gam\in(-1,p-1)$ and let $X$ be a Banach space. Let $s>-1+\frac{\gam+1}{p}$ and $m\in \NN_0\cup\{-1\}$ be such that $m+\frac{\gam+1}{p}< s< m+1+\frac{\gam+1}{p}$. Then for all $f\in H^{s,p}(\RRd, w_\gam;X)\cap \big\{g\in W^{m+1,1}_{\loc}(\RRd;X): g(0,\cdot)=\dots = (\d_1^m g)(0, \cdot)=0\big\}$ we have
  \begin{equation*}
      \big\|\ind_{\RRdh}f \big\|_{H^{s,p}(\RRd, w_\gam;X)}\leq C \|f\|_{H^{s,p}(\RRd, w_\gam;X)},
  \end{equation*}
  where the constant $C>0$ only depends on $p, \gam, s$ and $X$. 
  
As a consequence, $\ind_{\RRdh}$ is a pointwise multiplier on the closure of $\Cc^{\infty}(\RR_\bullet^{d};X)$ with respect to the norm of $H^{s,p}(\RRd,w_{\gam};X)$. Moreover, it holds that
  \begin{equation*}
    \d^{\alpha}(\ind_{\RRdh} f) = \ind_{\RRdh}(\d^{\alpha} f)\qquad \text{ for }|\alpha|\leq m\,\text{ and }\, f\in \overline{\Cc^{\infty}(\RR_\bullet^{d};X)}^{H^{s,p}(\RRd,w_{\gam};X)}.
  \end{equation*}
\end{proposition}

The proof of Proposition \ref{prop:LMV6.2_ext} relies on an improvement of \cite[Lemma 4.2]{LMV17}, where the $\UMD$ condition on the Banach space can be omitted. 

\begin{lemma}\label{lem:noUMD}
Let $p\in (1,\infty)$, $s\in \RR$, $\sigma\geq 0$, $w\in A_p(\RRd)$ and let $X$ be a Banach space. Then $f\in H^{s+\sigma,p}(\RRd, w;X)$ if and only if $f, (-\del)^{\sigma/2}\in H^{s,p}(\RRd, w;X)$. In this case, it holds that
\begin{equation*}
    \|f\|_{H^{s+\sigma,p}(\RRd, w;X)} \eqsim \|f\|_{H^{s,p}(\RRd, w;X)}+\|(-\del)^{\frac{\sigma}{2}} f\|_{H^{s,p}(\RRd, w;X)},
\end{equation*}
where the constant only depends on $p,s, \sigma, w$ and $X$.
\end{lemma}
\begin{proof}
    The Laplacian $-\del$ is a sectorial operator on $L^p(\RRd, w_{\gam};X)$ with domain $D(\del)= H^{2,p}(\RRd, w;X)$. By lifting we obtain that $-\del$ is a sectorial operator on $H^{s,p}(\RRd, w_{\gam};X)$ with $D(\del)= H^{s+2,p}(\RRd, w;X)$ for any $s\in \RR$. Now, \cite[Proposition 15.2.12]{HNVW24} yields for any $s\in\RR$ and $\sigma\geq 0$ that
\begin{align*}
  \|f\|_{H^{s+\sigma,p}(\RRd,w_{\gam};X)} &= \big\|(1-\del)^{\frac{\sigma}{2}}f\big\|_{H^{s,p}(\RRd,w_{\gam};X)}\\& \eqsim \|f\|_{H^{s,p}(\RRd,w_{\gam};X)} + \big\|(-\del)^{\frac{\sigma}{2}}f\big\|_{H^{s,p}(\RRd,w_{\gam};X)}.
\end{align*} 
This finishes the proof.
\end{proof}

As a consequence of Lemma \ref{lem:noUMD}, we can answer the question for which Banach spaces $X$ the pointwise multiplication operator $\ind_{\RRdh}$ is a bounded operator on $H^{s,p}(\RRd,w_{\gam};X)$. In particular, this solves \cite[Problem Q.12]{HNVW24}.

\begin{theorem}\label{thm:multRRdh_noUMD}
    Let $p\in (1,\infty)$, $\gam\in (-1,p-1)$, $-1+\frac{\gam+1}{p}<s<\frac{\gam+1}{p}$ and let $X$ be a Banach space. Then the pointwise multiplication operator $\ind_{\RRdh}$ extends to a bounded operator on $H^{s,p}(\RRd,w_{\gam};X)$.
\end{theorem}
\begin{proof}
    The proof is precisely the same as \cite[Theorem 4.1]{LMV17}, where this theorem is proved for $\UMD$ Banach spaces $X$. The only place where their proof uses the $\UMD$ condition is through the norm equivalence in \cite[Lemma 4.2]{LMV17}. Lemma \ref{lem:noUMD} gives precisely this norm equivalence without the $\UMD$ condition.
\end{proof}

We now turn to the proof of Proposition \ref{prop:LMV6.2_ext}. Comparing the statements of Proposition \ref{prop:LMV6.2_ext} and Theorem \ref{thm:multRRdh_noUMD}, we note that in Theorem \ref{thm:multRRdh_noUMD} the additional condition $s< (\gam+1)/p$ is required. This is precisely the setting where no traces at $\d\RRdh$ exist. In the setting of Proposition \ref{prop:LMV6.2_ext}, certain traces exist if $s\geq (\gam+1)/p$, but all these traces are zero. This implies that derivatives and $\ind_{\RRdh}$ commute, which allows us to remove the condition $s< (\gam+1)/p$.
\begin{proof}[Proof of Proposition \ref{prop:LMV6.2_ext}]
    The proof is completely analogous to \cite[Lemma 6.1 \& Proposition 6.2]{LMV17}. First, \cite[Lemma 6.1]{LMV17} can be extended from $d=1$ to $d\geq 1$. Afterwards, redoing the proof of \cite[Proposition 6.2]{LMV17} using Theorem \ref{thm:multRRdh_noUMD} instead of \cite[Theorem 4.1]{LMV17}, gives the desired result.
\end{proof}
 
We continue with a density result for $H^{s,p}_0(\OO,w_{\gam};X)$, which will be needed in Section \ref{sec:compl_intp} to characterise complex interpolation spaces of Bessel potential spaces with vanishing boundary conditions. The proposition below is an extension of the results in \cite[Propositions 6.4 \& 6.6]{LMV17}.
\begin{proposition}\label{prop:H_0_Trchar}
  Let $p\in(1,\infty)$, $s>-1+\frac{\gam+1}{p}$ such that $s\notin \NN_0+\frac{\gam+1}{p}$, $\gam\in(-1,p-1)$ and let $X$ be a Banach space. Then
  \begin{enumerate}[(i)]
    \item\label{it:densH1} $H^{s,p}_0(\RR^d,w_{\gam};X)=\overline{\Cc^{\infty}(\RR_\bullet^{d};X)}^{H^{s,p}(\RRd,w_{\gam};X)}$,
    \item\label{it:densH2}  $H^{s,p}_0(\RRdh,w_{\gam};X)=\overline{\Cc^{\infty}(\RRdh;X)}^{H^{s,p}(\RRdh,w_{\gam};X)}$.
  \end{enumerate}
\end{proposition}
\begin{proof} Let $m\in\NN_0\cup\{-1\}$ be such that $m+\frac{\gam+1}{p}<s<m+1+\frac{\gam+1}{p}$. For $m=-1$ no traces exist (see also \eqref{eq:H=H0}) and the arguments below can easily be adapted to this case by ignoring any trace conditions. Moreover, we note that derivatives in the directions $\tilde{x}=(x_2,\dots, x_d)$ commute with the trace operator (which follows with a density argument), and therefore for $m\in\NN_0$ and $\OO\in \{\RR^d,\RRdh\}$ it holds that
\begin{equation*}
  H^{s,p}_0(\OO,w_{\gam};X)=\{f\in H^{s,p}(\OO, w_{\gam};X): \bTr_m f=0\}.
\end{equation*}

  \textit{Step 1: proof of \ref{it:densH1}. } For $f\in \Cc^{\infty}(\RR_\bullet^{d};X)$ it is clear that $\overline{\Tr}_m f=0$ and by continuity (see Proposition \ref{prop:bTr_RRdh}\ref{it:prop:bTr_RRdh_H}) this extends to the closure of $\Cc^{\infty}(\RR_\bullet^{d};X)$. Conversely, let $f\in H^{s,p}(\RRd,w_{\gam};X)$ such that $\overline{\Tr}_m f =0$. By \cite[Lemma 3.4]{LMV17} there exists a sequence $(f^{(1)}_n)_{n\geq 1}\subseteq \Cc^{\infty}(\RRd;X)$ such that 
  \begin{equation}\label{eq:convf_1}
    f^{(1)}_n\to f\quad \text{ in }H^{s,p}(\RRd,w_{\gam};X)\,\text{ as }\, n\to\infty.
  \end{equation}
  Let $\bext_m$ be as in Proposition \ref{prop:bTr_RRdh}\ref{it:prop:bTr_RRdh_H} and define $f^{(2)}_n:=f^{(1)}_n-\bext_m(\bTr_m f^{(1)}_n)$ for $n\in\NN_1$. Then by Remark \ref{rem:map_prop} we have $f^{(2)}_n\in C^{\infty}(\RRd;X)$ and from Proposition \ref{prop:bTr_RRdh}\ref{it:prop:bTr_RRdh_H} it follows that $\bTr_m f^{(2)}_n=0$ for all $n\in\NN_1$. Moreover, using $\bTr_m f=0$, Proposition \ref{prop:bTr_RRdh}\ref{it:prop:bTr_RRdh_H} twice and \eqref{eq:convf_1}, it follows
  \begin{align*}
    \|\bext_m(\bTr_mf_n^{(1)})\|_{H^{s,p}(\RRd,w_{\gam};X)}& \leq C \sum_{j=0}^m\|\bTr_m f- \bTr_m f^{(1)}_n\|_{B_{p,p}^{s-j-\frac{\gam+1}{p}}(\RR^{d-1};X)} \\
     &\leq C \|f-f_n^{(1)}\|_{H^{s,p}(\RRd,w_{\gam};X)}\to 0\quad \text{ as }n\to\infty
  \end{align*}
and therefore
  \begin{align*}
    \|f-f_n^{(2)}\|_{H^{s,p}(\RRd,w_{\gam};X)} & \leq \|f-f_n^{(1)}\|_{H^{s,p}(\RRd,w_{\gam};X)}+\|\bext_m(\bTr_mf_n^{(1)})\|_{H^{s,p}(\RRd,w_{\gam};X)}  \to 0
  \end{align*}
  as $n\to \infty$.  By a standard cut-off argument we find a sequence $(f^{(3)}_n)_{n\geq 1}\subseteq \Cc^{\infty}(\RRd;X)$ such that $\bTr_m f_n^{(3)} =0$ for all $n\in \NN_1$ and $f_n^{(3)}\to f$ in $H^{s,p}(\RR^d,w_{\gam};X)$ as $n\to\infty$.
  
  It remains to show that any $g\in \Cc^{\infty}(\RRd;X)$ with $\bTr_m g =0$ can be approximated by functions in $\Cc^{\infty}(\RR_\bullet^d;X)$. By writing $g=\ind_{\RRdh} g + \ind_{\RR^d_{-}} g =: g_++g_-$, we see that $g_{\pm}\in H^{s,p}(\RRd,w_{\gam};X)$ by Proposition \ref{prop:LMV6.2_ext} and thus it suffices to approximate $g_+$ and $g_-$ separately. Let $\phi^{\pm}\in \Cc^{\infty}(\RR^d)$ such that $\int_{\RR^d} \phi^{\pm} \dd x =1$ and $\supp \phi^+ \subseteq [1,\infty)\times \RR^{d-1}$ and  $\supp \phi^- \subseteq (-\infty, -1]\times \RR^{d-1}$. Define $\phi^{\pm}_n(x):=n^d\phi^{\pm}(n x)$ for $n\in \NN_1$ and $x\in \RRd$. Then $\phi^{\pm}_n\ast  g_{\pm}\in \Cc^{\infty}(\RR_\bullet^d;X)$ and $\phi_n\ast g_{\pm}\to g_{\pm}$ in $H^{s,p}(\RRd,w_{\gam};X)$ as $n\to\infty$ by \cite[Lemma 3.6]{LMV17}. 
  
  \textit{Step 2: proof of \ref{it:densH2}. }Again, the embedding ``$\hookleftarrow$" is straightforward to prove. For the other embedding, let $f\in H^{s,p}(\RRdh,w_{\gam};X)$ such that $\overline{\Tr}_m f =0$. Then there exists an $\tilde{f}\in H^{s,p}(\RRd,w_{\gam};X)$ such that $\tilde{f}|_{\RRdh}=f$ and $\bTr_m \tilde{f}=\bTr_m f =0$ (see the discussion about traces on the half-space at the beginning of Section \ref{subsec:tracemSob}). By Step 1 there exists a sequence $\tilde{f}_n\in \Cc^{\infty}(\RR_\bullet^d;X)$ such that $\tilde{f}_n\to \tilde{f}$ in $H^{s,p}(\RRd,w_{\gam};X)$ as $n\to\infty$. It follows that $f_n:=\tilde{f}_n|_{\RRdh}\in \Cc^{\infty}(\RRdh;X)$ and $f_n\to f$ in $H^{s,p}(\RRdh,w_{\gam};X)$ as $n\to\infty$.
\end{proof}

\subsubsection{Sobolev spaces}\label{subsec:dens_W} Let $p\in(1,\infty)$, $k\in\NN_0$, $\gam\in (-1,\infty)\setminus\{jp - 1:j\in\NN_1\}$ and $X$ a Banach space. Then we define
\begin{equation*}
  W^{k,p}_{0}(\RR_+^d, w_{\gam}; X):=\left\{f\in W^{k,p}(\RR_+^d,w_{\gam};X): \Tr(\d^{\alpha}f)=0 \text{ if }k-|\alpha|>\tfrac{\gam+1}{p}\right\}.
\end{equation*}
Moreover, if $m\in\NN_0$, then we define
  \begin{align*}
  W^{k,p}_{\Tr_m}(\RR_+^d, w_{\gam}; X)&:=\left\{f\in W^{k,p}(\RR_+^d,w_{\gam};X): \operatorname{Tr}_mf=0 \text{ if }k-m>\tfrac{\gam+1}{p}\right\}.
\end{align*}
All the traces in the above definitions are well defined by Theorem \ref{thm:tracemW_RRdh_allweights}. Note that for $m=0$ and $m=1$ the above Sobolev spaces with zero boundary conditions coincide with the Dirichlet and Neumann spaces as defined in \cite{LLRV24}. Similar to the Bessel potential spaces, it holds that $\Cc^\infty(\RRdh;X)$ is dense in $W_0^{k,p}(\RRdh,w_{\gam};X)$, see, e.g., \cite[Proposition 3.8]{LV18}.
To deal with higher-order boundary conditions we need spaces of smooth functions where only higher-order derivatives are compactly supported. For $m\in\NN_0$ we define
\begin{equation*}\label{eq:setDense}
  C^{\infty}_{{\rm c},m}(\overline{\RRdh};X):=\{f\in \Cc^{\infty}(\overline{\RRdh};X): \d_1^m f\in \Cc^{\infty}(\RRdh;X)\}.
\end{equation*}
In particular, note that $f\in C^{\infty}_{{\rm c},m}(\overline{\RRdh};X)$ satisfies $\d^\alpha f\in \Cc^\infty(\RRdh;X)$ for all $\alpha=(\alpha_1,\tilde{\alpha})\in \NN_0\times\NN^{d-1}_0$ with $\alpha_1\geq m$.
\\

The following two density results will be required in \cite{LLRV25} to find trace characterisations for weighted Sobolev spaces on domains with, e.g., $C^{1}$-boundary. We first prove a density result for homogeneous boundary conditions of order $m$. 
\begin{proposition}\label{prop:tracechar_RRdh}
  Let $p\in(1,\infty)$, $k,m\in\NN_0$ such that $k\geq m$, $\gam\in (-1,\infty)\setminus\{jp - 1:j\in\{1,\dots, k-m\}\}$ and let $X$ be a Banach space. Then
  \begin{equation*}
    W^{k,p}_{\Tr_m}(\RR_+^d, w_{\gam}; X)=\overline{\cbraceb{f \in \Cc^{\infty}(\overline{\RRdh};X): (\d_1^m f)|_{\partial\RRdh}=0}}^{W^{k,p}(\RRdh,w_{\gam};X)}.
  \end{equation*}
\end{proposition}
\begin{proof}
  \textit{Step 1: the case $\gam>(k-m)p -1$.} 
Note that for $\gam>(k-m)p -1$  it holds that  $W^{k,p}_{\Tr_m}(\RR_+^d, w_{\gam}; X)=W^{k,p}(\RR_+^d, w_{\gam}; X)$, so that the result follows from
\begin{align*}
  C_{\mathrm{c}, m}^{\infty}(\overline{\RRdh};X)&\subseteq \cbraceb{f \in \Cc^{\infty}(\overline{\RRdh};X): (\d_1^m f)|_{\partial\RRdh}=0}\quad  \text{and} \\
   C_{\mathrm{c},m}^{\infty}(\overline{\RRdh};X)&\stackrel{\text{dense}}{\hookrightarrow} W^{k,p}(\RRdh,w_{\gam};X),
\end{align*}
see \cite[Lemma 3.4]{LLRV24}.

  \textit{Step 2: the case $\gam<(k-m)p -1$.} Let $\eps>0$ and take $f\in W^{k,p}(\RRdh, w_{\gam};X)$ such that $\Tr_m f =0$. By  \cite[Theorem 7.2 \& Remark 11.12(iii)]{Ku85}, which also holds in the vector-valued case, and a standard cut-off argument, there exists an $f^{(1)}\in  \Cc^{\infty}(\overline{\RRdh};X)$ such that
\begin{equation}\label{eq:conv_kufner}
  \|f-f^{(1)}\|_{W^{k,p}(\RRdh,w_{\gam};X)}<\eps.
\end{equation}
 Let $\ext_m$ be as in Theorem \ref{thm:tracemW_RRdh_allweights} and define $f^{(2)}:=f^{(1)}-\ext_m(\Tr_m f^{(1)})$. Then by Remark \ref{rem:map_prop} we have $f^{(2)}\in C^\infty(\overline{\RRdh};X)$ and from Theorem \ref{thm:tracemW_RRdh_allweights} it follows that $(\d_1^m f)|_{\RRdh}=0$ and  
\begin{align*}
  \|f-f^{(2)}\|_{W^{k,p}(\RRdh,w_{\gam};X)} \leq & \; \|f-f^{(1)}\|_{W^{k,p}(\RRdh,w_{\gam};X)}+\|\ext_m(\Tr_m f^{(1)})\|_{W^{k,p}(\RRdh, w_{\gam};X)}.
\end{align*}
From $\Tr_m f =0$, Theorem \ref{thm:tracemW_RRdh_allweights} twice and \eqref{eq:conv_kufner}, it follows
\begin{equation*}\label{eq:Est_extTr}
  \begin{aligned}
  \|\ext_m (\Tr_m f^{(1)})\|_{W^{k,p}(\RRdh, w_{\gam};X)}  
  &\leq C\|\Tr_m f -\Tr_m f^{(1)}\|_{B^{k-m-\frac{\gam+1}{p}}_{p,p}(\RR^{d-1};X)}\\
  &\leq C \|f-f^{(1)}\|_{W^{k,p}(\RRdh, w_{\gam};X)} < C\eps.
\end{aligned}
\end{equation*}
This proves that $\|f-f^{(2)}\|_{W^{k,p}(\RRdh,w_{\gam};X)} < (C+1)\eps$. Finally, a standard cut-off argument gives an approximating sequence in $\Cc^{\infty}(\overline{\RRdh};X)$ satisfying the boundary condition.
\end{proof}

Under suitable restrictions on $\gam$, such that traces of order $\geq m+1$ do not exist, we can prove density of $ C^{\infty}_{{\rm c},m}(\overline{\RRdh};X)$ in $W^{k,p}_{\Tr_m}(\RRdh, w_{\gam};X)$.
The proof combines the construction of an approximating sequence from \cite[Lemma 3.4]{LLRV24} with the result of Proposition \ref{prop:tracechar_RRdh}.
\begin{proposition}
  Let $p\in(1,\infty)$, $k,m\in\NN_0$ such that $k\geq m+1$, $(k-m-1)p-1 <\gam< (k-m)p-1$ and let $X$ be a Banach space. Then
  \begin{equation*}
    W^{k,p}_{\Tr_m}(\RRdh, w_{\gam};X)=\overline{C^{\infty}_{{\rm c},m}(\overline{\RRdh};X)}^{W^{k,p}(\RRdh, w_{\gam};X)}.
  \end{equation*}
\end{proposition}
\begin{proof}
Let $\eps>0$ and take $f \in W^{k,p}(\RRdh,w_{\gamma};X)$ such that $\Tr_m f=0$. The conditions on $\gam$ guarantee that $\Tr_m f$ exists and traces of order $\geq m+1$ do not exist. By Proposition \ref{prop:tracechar_RRdh} there exists a $g\in \Cc^\infty(\overline{\RRdh};X)$ with its support in $[0,R]\times [-R,R]^{d-1}$ for some $R>0$ such that $(\d_1^m g)|_{\d\RRdh}=0$ and 
\begin{equation}\label{eq:conv_f_g}
  \|f-g\|_{W^{k,p}(\RRdh,w_{\gamma};X)}< \eps.
\end{equation}
Let $\phi \in C^\infty(\R_+)$ be such that $\phi =0$ on $[0,\tfrac12]$ and $\phi=1$ on $[1,\infty)$ and set $\phi_n(x_1) := \phi(nx_1)$. We construct a sequence $(g_n)_{n\geq1}$ as follows:
\begin{itemize}
  \item If $m=0$, define $g_n(x):=\phi_n(x_1)g(x)$.
  \item If $m\geq 1$, define
  $$
  g_n(x):= \sum_{j=0}^{m-1} g(1,\tilde{x})\frac{(x_1-1)^j}{j!} + \frac{1}{(m-1)!}\int_1^{x_1} (x_1-t)^{m-1}\phi_n(t)\partial^m_1g(t,\tilde{x})\dd t.
  $$
\end{itemize}
Note that by integration by parts $g_n(x) = g(x)$ for all $x \in \R^d_+$ with $x_1\geq\frac1n$. For $\abs{\alpha}\leq k$ with $\alpha_1\leq m$ and $x \in \R^d_+$ with $x_1<1$ we have
$$
\|\partial^\alpha g_n(x)\|_X\leq C_m \nrm{g}_{C_\b^{k+m}(\R^d_+;X)}\nrm{\phi}_{L^\infty(\R_+)}.
$$ Moreover, we have 
\begin{equation}\label{eq:prodrule}
  \partial_1^mg_n(x) = \phi_n(x_1)\partial^m_1g(x),\qquad x \in \R^d_+,
\end{equation}
so that in particular $\partial_1^mg_n \in \Cc^\infty({\R^d_+};X)$. We write $\alpha=(\alpha_1,\tilde{\alpha})\in\NN_0\times \NN_0^{d-1}$. If $\alpha_1\in \{m+1,\dots, k\}$, then it follows from \eqref{eq:prodrule} and the product rule that
\begin{equation}\label{eq:prodrule2}
  \begin{aligned}
  \|\d_1^{\alpha_1}g_n(x)\|_X&= \big\|\d_1^{\alpha_1-m}\big(\phi_n(x_1)\partial^m_1g(x)\big) \big\|_X \\
 & \leq n^{\alpha_1-m }\big\|\phi^{(\alpha_1-m)}(nx_1)\d_1^m g(x)\big\|_X + C_{k,m}\sum_{j=0}^{\alpha_1-m-1} n^j \big\|\phi^{(j)}(nx_1)\d_1^{\alpha_1-j}g(x)\big\|_X\\
 &\leq n^{k-m }\|\phi\|_{C_{\b}^{k-m}(\RR_+)}\big\|\d_1^m g(x)\big\|_X+ C_{k,m}\,n^{k-m-1}\|\phi\|_{C_{\b}^{k-m-1}(\RR_+)}\|g\|_{C_{\b}^k(\RRdh;X)}.
\end{aligned}
\end{equation}
Let $K_R:=[-R,R]^{d-1}$. The properties of $g_n$ and \eqref{eq:prodrule2} imply that
\begin{align*}
  \|g-g_n\|_{W^{k,p}(\RRdh,w_{\gamma};X)}
   \leq &\; \sum_{|\alpha|\leq k} \Big(\int_{K_R}\int_0^{\frac{1}{n}}x_1^\gam\|\d^{\alpha} g(x)\|_X^p\dd x_1\dd \tilde{x}\Big)^{\frac{1}{p}}  \\
   & +\sum_{\substack{\abs{\alpha}\leq k\\ \alpha_1\leq m}} \Big(\int_{K_R}\int_0^{\frac{1}{n}}x_1^\gam\|\d^{\alpha} g_n(x)\|_X^p\dd x_1\dd \tilde{x}\Big)^{\frac{1}{p}}\\
    & + n^{k-m }\|\phi\|_{C_{\b}^{k-m} (\RR_+)} \Big(\int_{K_R}\int_0^{\frac{1}{n}}x_1^\gam\|\d^{\tilde{\alpha}} \d_1^{m}g(x)\|_X^p\dd x_1\dd \tilde{x}\Big)^{\frac{1}{p}}   \\ &+ C_{k,m}\,n^{k-m-1}\|\phi\|_{C_{\b}^{k-m-1}(\RR_+)}\|g\|_{C_{\b}^k(\RRdh;X)} \Big(\int_{K_R}\int_0^{\frac{1}{n}}x_1^\gam\dd x_1\dd \tilde{x}\Big)^{\frac{1}{p}}\\
    \leq &\; C_{k,m}\,n^{\frac{1}{p}((k-m-1)p -1-\gam )}\cdot\tfrac{(2R)^{\frac{d-1}{p}}}{\gam+1}\|g\|_{C_\b^{k+m}(\R^d_+;X)}\|\phi\|_{C_{\b}^{k-m-1}(\RR_+)}\\
    &+  n^{k-m -1}\|\phi\|_{C_{\b}^{k-m} (\RR_+)} \Big(\int_{K_R}\int_0^{\frac{1}{n}}x_1^{\gam-p}\|\d^{\tilde{\alpha}} \d_1^{m}g(x)\|_X^p\dd x_1\dd \tilde{x}\Big)^{\frac{1}{p}}  \\
    \leq &\; C_{k,m,p}\,n^{\frac{1}{p}((k-m-1)p -1-\gam )}\cdot\tfrac{(2R)^{\frac{d-1}{p}}}{\gam+1}\|g\|_{C_\b^{k+m}(\R^d_+;X)}\|\phi\|_{C_{\b}^{k-m}(\RR_+)},
\end{align*}
where in the last step we first applied Hardy's inequality (using condition \ref{it:lem:Hardy1} in Lemma \ref{lem:Hardy} since $\gam-p>-p-1$ and $(\d_1^m g)|_{\d\RRdh}=0$), then estimated $g$ and finally calculated the remaining integrals over $x_1$ and $\tilde{x}$. Hence, as $\gam>(k-m-1)p-1$, taking $n$ large enough gives with \eqref{eq:conv_f_g}
 that
 \begin{equation*}
   \|f- g_n\|_{W^{k,p}(\RRdh,w_{\gamma};X)}\leq \|f-g\|_{W^{k,p}(\RRdh,w_{\gamma};X)} + \|g-g_n\|_{W^{k,p}(\RRdh,w_{\gamma};X)}< 2\eps.
 \end{equation*}
 This completes the proof.
 \end{proof}

\section{Trace theorem for boundary operators}\label{sec:traceB}
In this section, we study differential operators on the boundary of the half-space and we prove a trace theorem for systems of normal boundary operators on weighted Bessel potential and Sobolev spaces.\\

We start with the definition of a boundary operator which is in the same spirit as \cite[Section VIII.2]{Am19}.
\begin{definition}[Boundary operators]\label{def:bound_op}  Let $p\in(1,\infty)$, $s>0$, $\gam\in(-1,\infty)\setminus\{jp-1:j\in\NN_1\}$, $m\in\NN_0$ and let $X,Y$ be Banach spaces. Then the \emph{boundary operator of order $m$}, given by
\begin{equation*}
  \mc{B}:=\sum_{|\alpha|\leq m}b_{\alpha}\Tr \d^{\alpha}\qquad \text{with }b_{\alpha}\neq 0\text{ for some }|\alpha|=m,
\end{equation*} 
is of \emph{type $(p,s,\gam,m,Y)$} if
  \begin{enumerate}[(i)]
    \item $s>m+\frac{\gam+1}{p}$, 
    \item\label{it:def:bound_op2} for all $|\alpha|\leq m$ we have $b_{\alpha}\in C^{\ell}_{{\rm b}}(\RR^{d-1};\mc{L}(X,Y))$ for some $\ell\in\NN_0$ such that $\ell>s-m-\frac{\gam+1}{p}$.
  \end{enumerate}
\end{definition}
Note that for $\alpha=(\alpha_1,\tilde{\alpha})\in\NN_0\times \NN_0^{d-1}$, we have $\Tr\d^{\alpha}=\d^{\tilde{\alpha}}\Tr_{\alpha_1}$ and therefore the boundary operator $\mc{B}$ can be rewritten as
\begin{equation}\label{eq:BTr}
  \mc{B}=\sum_{j=0}^m b_j\Tr_j,\quad\text{ where } \quad b_j(\cdot, \grad_{\tilde{x}}):=\sum_{|\tilde{\alpha}|\leq m-j}b_{(j,\tilde{\alpha})}\d^{\tilde{\alpha}}.
\end{equation}
Note that the leading order coefficient $b_m$ is independent of $\grad_{\tilde{x}}$. In the following lemma, we prove that $\mc{B}$ is a bounded operator from a weighted Bessel potential or Sobolev space to its trace space.
\begin{lemma}\label{lem:boundB}
 Let $p\in(1,\infty)$, $m\in\NN_0$ and let $X,Y$ be Banach spaces. 
 \begin{enumerate}[(i)]
   \item\label{it:lem:boundB1} If $s>0$, $\gam\in(-1,p-1)$ and $\mc{B}$ is of type $(p,s,\gam,m,Y)$, then
     \begin{equation*}
    \mc{B}: H^{s,p}(\RRdh,w_{\gam};X)\to B^{s-m-\frac{\gam+1}{p}}_{p,p}(\RR^{d-1};Y)\quad \text{ is bounded. }
  \end{equation*}
   \item\label{it:lem:boundB2} If $k\in \NN_1$, $\gam\in(-1,\infty)\setminus\{jp-1:j\in \NN_1\}$ and $\mc{B}$ is of type $(p,k,\gam,m,Y)$, then
         \begin{equation*}
    \mc{B}: W^{k,p}(\RRdh,w_{\gam};X)\to B^{k-m-\frac{\gam+1}{p}}_{p,p}(\RR^{d-1};Y)\quad \text{ is bounded. }
  \end{equation*}
 \end{enumerate}
\end{lemma}
\begin{proof} \textit{Step 1: proof of \ref{it:lem:boundB1}.} By Theorem \ref{thm:tracemHW_RRdh}\ref{it:cor:tracemHW_RRdh_H} we have that
  \begin{equation}\label{eq:proofest_1}
    \Tr_j:H^{s,p}(\RRdh,w_{\gam};X)\to B_{p,p}^{k-j-\frac{\gam+1}{p}}(\RR^{d-1};X)\quad\text{ for all }j\in\{0,\dots, m\},
  \end{equation}
  is bounded and by \cite[Proposition 3.10]{MV12} it follows that
  \begin{equation}\label{eq:proofest_2}
    \d^{\tilde{\alpha}}: B_{p,p}^{s-j-\frac{\gam+1}{p}}(\RR^{d-1};X)\to B_{p,p}^{s-j-|\tilde{\alpha}|-\frac{\gam+1}{p}}(\RR^{d-1};X) \quad\text{ for all }|\tilde{\alpha}|\leq m-j,
  \end{equation}
  is bounded. Note that \eqref{eq:proofest_2} and the embedding $B^{s_1}_{p,p}(\RR^{d-1};X)\hookrightarrow B^{s_2}_{p,p}(\RR^{d-1};X)$ for $s_1\geq s_2$, imply that
  \begin{equation}\label{eq:proofest_3}
    \d^{\tilde{\alpha}}: B_{p,p}^{s-j-\frac{\gam+1}{p}}(\RR^{d-1};X)\to B_{p,p}^{s-m-\frac{\gam+1}{p}}(\RR^{d-1};X) \quad\text{ for all }|\tilde{\alpha}|\leq m-j,
  \end{equation}
  is bounded as well. Moreover, by Definition \ref{def:bound_op}\ref{it:def:bound_op2} and \cite[Example 14.4.33]{HNVW24}, the pointwise multiplication operator 
  \begin{equation}\label{eq:proofest_4}
    M_{b_{\alpha}}: B_{p,p}^{s-m-\frac{\gam+1}{p}}(\RR^{d-1};X)\to B_{p,p}^{s-m-\frac{\gam+1}{p}}(\RR^{d-1};Y)
  \end{equation}
  is bounded. Combining \eqref{eq:BTr}, \eqref{eq:proofest_1}, \eqref{eq:proofest_3} and \eqref{eq:proofest_4}, yields the result.
  
  \textit{Step 2: proof of \ref{it:lem:boundB2}.}
  This case is analogous to Step 1 using Theorem \ref{thm:tracemW_RRdh_allweights} to obtain \eqref{eq:proofest_1} for weighted Sobolev spaces with $\gam\in(-1,\infty)\setminus\{jp-1:j\in\NN_1\}$.
\end{proof}

We make the following definitions about (systems of) normal boundary operators.  
\begin{definition}[Normal boundary operators]\label{def:nor_bound_op}
An operator $\mc{B}$ of type $(p,s,\gam,m,Y)$ is a \emph{normal boundary operator of type $(p,s,\gam,m,Y)$} if the leading order coefficient $b_m$ in \eqref{eq:BTr} admits a continuous right inverse $b_m^{{\rm c}}\in C_{{\rm b}}^{\ell}(\RR^{d-1};\mc{L}(Y,X))$, where $\ell$ is as in Definition \ref{def:bound_op}.
\end{definition}
In the above definition, $b_m(\tilde{x})$ is a retraction for any $\tilde{x}\in\RR^{d-1}$ with coretraction $b_m^{{\rm c}}(\tilde{x})$, see \cite[Section 1.2.4]{Tr78}.

\begin{definition}[System of normal boundary operators]\label{def:syst}
  Let $n\in\NN_0$ and
  \begin{enumerate}[(i)]
    \item\label{it:def:syst1} $0\leq m_0<m_1<\cdots <m_n$ be integers,
    \item\label{it:def:syst2}$Y_0,\dots, Y_n$ be Banach spaces,
    \item\label{it:def:syst3} $\mc{B}^{m_i}$ be a normal boundary operator of type $(p,s,\gam,m_i,Y_i)$.
  \end{enumerate}
  Let $\overline{m}:=(m_0,\dots, m_n)$ and $\overline{Y}:=(Y_0,\dots, Y_n)$. Then the operator 
  \begin{equation*}
    \mc{B}:=(\mc{B}^{m_0},\dots, \mc{B}^{m_n})
  \end{equation*}
  is a \emph{normal boundary operator of type $(p,s, \gam,\overline{m}, \overline{Y})$}.
\end{definition}
Note that a normal boundary operator of type $(p,s,\gam, \overline{m}, \overline{Y})$ satisfies $s>m_n+\frac{\gam+1}{p}$. In particular, $(\Tr_{m_i})_{i=0}^n$ is a normal boundary operator of type  $(p,s,\gam, \overline{m}, \overline{X})$ for $s>m_n+\frac{\gam+1}{p}$. This includes the important cases for Dirichlet ($n=0$ and $m_0=0$) and Neumann ($n=0$ and $m_0=1$) boundary conditions.
 
\begin{remark}
In the case that $X=\CC^q$ for some $q\in \NN_1$, the definition of a normal boundary system as given in \cite[Definition 3.1]{Se72} coincides with our definition above. The definition in \cite{Se72} requires that the coefficient matrix for the highest-order derivative in $\BB^{m_i}$ is surjective, which implies that the boundary operator can be replaced by a $q\times q$ matrix operator (leaving the kernel of $\BB^{m_i}$ invariant) such that the coefficient for the highest order derivative is a projection-valued matrix, see \cite[Section 3]{Se72}. This is a key ingredient for characterising the complex interpolation spaces in \cite{Se72}. For general Banach spaces $X$ we replace these conditions by the retraction and coretraction in Definition \ref{def:nor_bound_op}. In the case that $X$ and $Y_i$ are Hilbert spaces we refer to \cite[Lemma VIII.2.1.3]{Am19} for a sufficient condition for $\BB$ to be normal. 

We note that \cite[Definition 3.1]{Se72} arises from Agmon's condition \cite{Ag62} about minimal growth of the resolvent of an elliptic operator.
\end{remark}

We can now prove the following trace theorem for normal boundary operators analogous to \cite[Theorem VIII.2.2.1]{Am19}. 

\begin{theorem}\label{thm:retract_W}
 Let $p\in(1,\infty)$ and let $X$ be a Banach space. Assume that the conditions \ref{it:def:syst1}-\ref{it:def:syst2} of Definition \ref{def:syst} hold.
 \begin{enumerate}[(i)]
   \item\label{it:thm:retract_W1} If $s>0$, $\gam\in(-1,p-1)$ and $\mc{B}$ is of type $(p,s,\gam,\overline{m},\overline{Y})$, then
     \begin{equation*}
    \mc{B}:H^{s,p}(\RRdh,w_{\gam};X)\to \prod_{j=0}^n B^{s-m_j-\frac{\gam+1}{p}}_{p,p}(\RR^{d-1};Y_j) 
  \end{equation*}
  is a continuous and surjective operator.
   \item\label{it:thm:retract_W2} If $k\in \NN_1$, $\gam\in(-1,\infty)\setminus\{jp-1:j\in \NN_1\}$ and $\mc{B}$ is of type $(p,k,\gam,\overline{m},\overline{Y})$, then
         \begin{equation*}
    \mc{B}:W^{k,p}(\RRdh,w_{\gam};X)\to \prod_{j=0}^n B^{k-m_j-\frac{\gam+1}{p}}_{p,p}(\RR^{d-1};Y_j) 
  \end{equation*}
  is a continuous and surjective operator.
 \end{enumerate}
 In both cases, there exists a continuous right inverse $\ext_{\BB}$ of $\mc{B}$ which is independent of $k,s,p,\gam, \overline{Y}$ and $X$. For any $0\leq j<m_n$ such that $j\notin \{m_0,\dots, m_{n-1}\}$, we have $\Tr_j\op \ext_{\BB}=0$.
\end{theorem}
\begin{proof} The proof is similar to \cite[Theorem VIII.2.2.1]{Am19} and for the convenience of the reader, we provide the proof.
 
The continuity of $\mc{B}$ follows from Lemma \ref{lem:boundB}. We construct the right inverse $\ext_\BB$ in the case of Bessel potential spaces. The case of Sobolev spaces is the same if one invokes Theorem \ref{thm:tracemW_RRdh_allweights} instead of Theorem \ref{thm:tracemHW_RRdh}\ref{it:cor:tracemHW_RRdh_H} in the construction below. 
Let
\begin{equation*}
  \mc{B}^{m_i}=\sum_{j=0}^{m_i}b_{i,j}\Tr_j,\quad b_{i,j}=\sum_{|\tilde{\alpha}|\leq m_i-j}b^{m_i}_{(j,\tilde{\alpha})}\d^{\tilde{\alpha}} \quad \text{ for }i\in\{0,\dots, n\},
\end{equation*}
where $b^{m_i}_{\alpha}$ are the coefficients of the operator $\mc{B}^{m_i}$. According to Definition \ref{def:nor_bound_op} there exists a right inverse $b^{{\rm c}}_{i,m_i}\in C_{{\rm b}}^\ell(\RR^{d-1}; \mc{L}(Y_i,X))$ of $b_{i,m_i}$. Define the operator
\begin{equation}\label{eq:deftildeC}
  \tilde{\mc{C}}^{m_i}=-\sum_{j=0}^{m_i-1} b^{{\rm c}}_{i,m_i} b_{i,j} \Tr_j,\quad \tilde{\mc{C}}^{0}:=0.
\end{equation}
Reasoning as in the proof of Lemma \ref{lem:boundB} gives that
\begin{equation*}
   \tilde{\mc{C}}^{m_i}: H^{s,p}(\RRdh,w_{\gam};X)\to B^{s-m_i-\frac{\gam+1}{p}}_{p,p}(\RR^{d-1};X)
\end{equation*}
is a bounded linear operator. Now, define the operator $\tilde{\mc{C}}:=(\tilde{\mc{C}}^{j})_{j=0}^{m_n}$
by setting $\tilde{\mc{C}}^{j}=\tilde{\mc{C}}^{m_i}$ if $j=m_i$ for some $i\in\{0,\dots, n\}$, and $\tilde{\mc{C}}^{j}=0$ if $j\notin\{m_0,\dots, m_n\}$. Then
\begin{equation*}
  \tilde{\mc{C}}:  H^{s,p}(\RRdh,w_{\gam};X)\to \prod_{j=0}^{m_n} B^{s-j-\frac{\gam+1}{p}}_{p,p}(\RR^{d-1};X)
\end{equation*}
is linear and bounded as well. Let $g=(g_0,\dots, g_{m_n})\in \prod_{i=0}^{n}B^{s-m_i-\frac{\gam+1}{p}}_{p,p}(\RR^{d-1};Y_i)$ and define $h=(h_j)_{j=0}^{m_n} \in \prod_{j=0}^{m_n}B^{s-j-\frac{\gam+1}{p}}_{p,p}(\RR^{d-1};X)$ by setting $h_j=h_{m_i}:=b^{{\rm c}}_{i,m_i}g_{m_i}$ if $j=m_i$ for some $i\in\{0,\dots, n\}$, and $h_j=0$ if $j\notin\{m_0,\dots, m_n\}$. Furthermore, let $f_0:=\ext_0 h_0$ and for $j\in\{1,\dots, m_n\}$ we recursively define
\begin{equation*}
  f_j := f_{j-1}+\ext_j\big(h_j+ \tilde{\mc{C}}^jf_{j-1}-\Tr_jf_{j-1}\big).
\end{equation*}
Note that $f_j\in H^{s,p}(\RRdh,w_{\gam};X)$ for all $j\in\{0,\dots, m_n\}$ by Theorem \ref{thm:tracemHW_RRdh}\ref{it:cor:tracemHW_RRdh_H}. We will prove that $\ext_\BB g:=f_{m_n}$ defines the right inverse we are looking for. 

By Theorem \ref{thm:tracemHW_RRdh}\ref{it:cor:tracemHW_RRdh_H} we have that $\Tr_\ell\op \ext_j=\delta_{j\ell}$ for $\ell\leq j$ and therefore
\begin{equation}\label{eq:propmcC}
  \begin{aligned}
  \Tr_jf_j&= h_j+ \tilde{\mc{C}}^jf_{j-1}\qquad &&\text{ for }0\leq j\leq m_n,\\
  \Tr_\ell f_j&= \Tr_\ell f_{j-1}\qquad &&\text{ for }0\leq \ell\leq j\leq m_n.
  \end{aligned}
\end{equation}
Suppose that $0\leq \ell\leq j\leq m_n$, then by \eqref{eq:propmcC} we have
\begin{equation}\label{eq:hlC}
  \Tr_\ell f_j = \Tr_\ell f_{j-1}=\cdots=\Tr_\ell f_{\ell}= h_{\ell}+ \tilde{\mc{C}}^\ell f_{\ell-1}.
\end{equation}
Moreover, if $\ell=m_i\in \{m_0,\dots, m_n\}$, then by \eqref{eq:deftildeC} and \eqref{eq:propmcC} 
\begin{equation}\label{eq:Cl=Cj}
  \tilde{\mc{C}}^\ell f_{\ell-1}=\tilde{\mc{C}}^\ell f_{\ell}=\dots=\tilde{\mc{C}}^\ell f_{j}.
\end{equation}
Hence, \eqref{eq:deftildeC} and \eqref{eq:Cl=Cj} imply
\begin{equation*}
  \Tr_{m_i}f_j = h_{m_i}+\tilde{\mc{C}}^{m_i}f_j \quad \text{ for }m_i\leq j\leq m_n,\quad 1\leq i\leq n.
\end{equation*}
Multiplying from the left with $b_{i,m_i}$ gives
\begin{equation}\label{eq:Bmi=gmi}
  \mc{B}^{m_i}f_j = g_{m_i}\quad \text{ for }m_i\leq m_n.
\end{equation}
Define
\begin{equation*}
  \mc{C}^{m_i}: \prod_{j=0}^i B^{s-m_j-\frac{\gam+1}{p}}_{p,p}(\RR^{d-1};Y_j)\to H^{s,p}(\RRdh,w_{\gam};X)
\end{equation*}
by $\mc{C}^{m_i}(g_{m_0},\dots,g_{m_n}):=f_{m_i}$ for $i\in\{1,\dots, n\}$. 
Then from \eqref{eq:Bmi=gmi} we obtain for $0\leq \ell\leq i\leq j\leq n$ that 
\begin{equation*}
  \mc{B}^{m_{\ell}}\mc{C}^{m_i}(g_0,\dots,g_{m_i}) = \mc{B}^{m_{\ell}}f_{m_i}=g_{m_\ell}.
\end{equation*}
So indeed, $\ext_{\BB}:=\mc{C}^{m_n}$ is the right inverse for $\mc{B}$.

Finally, if $\ell\notin\{m_0,\dots, m_n\}$, then $h_\ell=0$ and $\tilde{\mc{C}}^\ell=0$, so that \eqref{eq:hlC} implies 
\begin{equation*}
  \Tr_j(\ext_{\BB} g)=\Tr_jf_{m_n}=0.
\end{equation*}
This completes the proof.
\end{proof}

\section{Complex interpolation of Bessel potential and Sobolev spaces}\label{sec:compl_intp}
In this final section, we prove the characterisations for the complex interpolation spaces of weighted Bessel potential and Sobolev spaces. In Section \ref{subsec:int_p1} we start with some preliminary interpolation results on the half-space required for our main results which are stated in Section \ref{subsec:int_p2}. The proofs of the main results are given in Section \ref{sec:proof_intp}.

\subsection{Interpolation results for weighted spaces} \label{subsec:int_p1}
We start with an interpolation result for the weighted Bessel potential spaces $H_0^{s,p}(\OO, w_\gam;X)$, see Section \ref{subsec:densityBessel}. Interpolation theory for vector-valued Bessel potential (and Sobolev) spaces relies on the vector-valued Mihlin multiplier theorem, which requires the $\UMD$ condition, see \cite[Section 3]{MV15}. Consequently, unlike in previous sections where this condition was unnecessary, we will require the $\UMD$ condition from now on.

\begin{proposition}\label{prop:H_0_int}
  Let $p\in(1,\infty)$, $\gam\in(-1,p-1)$, $\OO\in\{\RR^d, \RRdh\}$ and let $X$ be a $\UMD$ Banach space. Assume that $\theta\in(0,1)$, $-1+\frac{\gam+1}{p}<s_0< s_\theta< s_1$ with $s_\theta=(1-\theta)s_0+\theta s_1$ satisfy $s_0,s_\theta,s_1\notin \NN_0 + \frac{\gam+1}{p}$. Then
  \begin{equation*}
    H^{s_\theta,p}_0(\OO,w_{\gam};X)=\big[H^{s_0,p}_0(\OO,w_{\gam};X), H^{s_1,p}_0(\OO,w_{\gam};X)\big]_\theta.
  \end{equation*}
\end{proposition}
\begin{proof}
\textit{Step 1: the case $\OO=\RRd$. }We define $H^{s,p}_{\RR^d_{\pm}}(\RR^d,w_{\gam};X)$ as the closed subspace of $H^{s,p}(\RRd,w_{\gam};X)$ consisting of all functions with its support in $\RR^d_{\pm}$. These spaces form a complex interpolation scale (see \cite[Proposition 5.7]{LMV17}) and below we construct an isomorphism to reduce to this known interpolation result.

Let $s>-1+\frac{\gam+1}{p}$ be such that $s\notin \NN_0+\frac{\gam+1}{p}$. Then we claim that $\bTr_m f=0 $ for integers $m\in [0,s-\frac{\gam+1}{p})$ and $f\in H^{s,p}_{\RR^d_{\pm}}(\RRd,w_{\gam};X)$. Indeed, this follows from the fact that if $f\in H^{s,p}(\RRd,w_{\gam};X)$ satisfies $f|_{(0,\delta)\times \RR^{d-1}}=0$ or $f|_{(-\delta,0)\times \RR^{d-1}}=0$ for some $\delta>0$, then $\bTr_m f=0$. This observation can be proved similarly as in \cite[Proposition 6.3(2)]{LMV17} using the continuity of $\bTr_m$ from Proposition \ref{prop:bTr_RRdh}\ref{it:prop:bTr_RRdh_H}.

The claim proves that the mapping 
\begin{equation*}
  R:H^{s,p}_{\RRdh}(\RRd,w_{\gam};X)\times H^{s,p}_{\RR^d_-}(\RRd,w_{\gam};X)\to H^{s,p}_0(\RRd,w_{\gam};X),\qquad R(g,h):=g+h,
\end{equation*}
is well defined and continuous. Furthermore, by Propositions \ref{prop:LMV6.2_ext} and \ref{prop:H_0_Trchar} the mapping
\begin{equation*}
  S:H^{s,p}_0(\RRd,w_{\gam};X)\to H^{s,p}_{\RRdh}(\RRd,w_{\gam};X)\times H^{s,p}_{\RR^d_-}(\RRd,w_{\gam};X), \qquad Sf:=(\ind_{\RRdh}f, \ind_{\RR^d_-}f),
\end{equation*}
is well defined and continuous. Since $R^{-1}=S$ the result now follows from \cite[Proposition 5.7]{LMV17}. 

\textit{Step 2: the case $\OO=\RRdh$. }Let $m\in\NN_0$ be the smallest integer such that $m\geq\max\{|s_0|, |s_1|\}$ and let  $s>-1+\frac{\gam+1}{p}$ such that $|s|\leq m$ and $s\notin \NN_0+\frac{\gam+1}{p}$. Moreover, let $$\mc{E}_+^m\in \mc{L}(H^{s,p}(\RRdh,w_{\gam};X), H^{s,p}(\RRd,w_{\gam};X))$$ be the reflection operator from \cite[Proposition 5.6]{LMV17}. Then by \cite[Proposition 5.6]{LMV17} the mapping
\begin{equation*}
  S: H_0^{s,p}(\RRdh, w_{\gam};X)\to H^{s,p}_0(\RRd, w_{\gam};X),\qquad Sf:=\mc{E}_+^m f,
\end{equation*}
is well defined and continuous. Furthermore, let $R: H_0^{s,p}(\RRd, w_{\gam};X)\to H^{s,p}_0(\RRdh, w_{\gam};X)$ be the restriction operator. The result now follows from Step 1 and \cite[Lemma 5.3]{LMV17}.
\end{proof}

We continue with interpolation results for Sobolev spaces.
First, we introduce the following multiplication operator.\\

For $\kappa\in \RR$ we define the pointwise multiplication operator
\begin{equation*}
  M^\kappa:\Cc^{\infty}(\RRdh;X)\to \Cc^{\infty}(\RRdh;X)\quad \text{ by }\quad M^{\kappa}u(x)=x^\kappa_1 u(x), \qquad x\in \RRdh,
\end{equation*}
which extends to an operator on $\mc{D}'(\RRdh;X)$ by duality and $M^{-\kappa}$ acts as inverse for $M^\kappa$. Moreover, we write $M:=M^1$.\\

The first interpolation result is for the Sobolev spaces $W^{k,p}_0(\RRdh, w_{\gam};X)$ with $\gam>-1$, see Section \ref{subsec:dens_W}. In this case, due to the zero boundary conditions, the operator $M^j: W^{k,p}_0(\RRdh, w_{\gam};X)\to W^{k,p}_0(\RRdh, w_{\gam-jp};X)$ is an isomorphism for all $j\in \NN_1$. This allows us to lift the interpolation result for $\gam\in (-1,p-1)$ to $\gam\in (jp-1, (j+1)p-1)$ for any $j\in \NN_1$.

\begin{proposition}\label{prop:intp_W_0}
  Let $p\in(1,\infty)$, $k_0\in\NN_0$, $k_1\in\NN_1\setminus\{1\}$, $\ell\in\{1,\dots,k_1-1\}$, $\gam\in (-1,\infty)\setminus\{jp-1:j\in\NN_1\}$ and let $X$ be a $\UMD$ Banach space.
Then
    \begin{equation*}
    \big[W^{k_0,p}_0(\RR^d_+,w_{\gam};X), W^{k_0+k_1,p}_0(\RR^d_+, w_{\gam};X)\big]_{\frac{\ell}{k_1}}=W^{k_0+\ell,p}_0(\RR^d_+,w_{\gam};X).
  \end{equation*}
\end{proposition}
\begin{proof}
   We first consider the case $k_0=0$. If $\gam\in(-1,p-1)$, then the statement follows from \cite[Proposition 3.15]{LV18}. For $\gam\in(jp-1,(j+1)p-1)$ with $j\in\NN_1$ we obtain by applying \cite[Lemma 3.6]{LLRV24}, properties of the complex interpolation method and the case $j=0$, that
  \begin{align*}
    \big[L^p(\RR^d_+&,w_{\gam};X), W^{k_1,p}_0(\RR^d_+,w_{\gam};X)\big]_{\frac{\ell}{k_1}}\\
     &= \big[M^{-j}L^p(\RR^d_+,w_{\gam-jp};X), M^{-j}W^{k_1,p}_0(\RR^d_+,w_{\gam-jp};X)\big]_{\frac{\ell}{k_1}} \\
     & = M^{-j}\big[L^p(\RR^d_+,w_{\gam-jp};X), W^{k_1,p}_0(\RR^d_+,w_{\gam-jp};X)\big]_{\frac{\ell}{k_1}}\\
     & =M^{-j}W_0^{\ell,p}(\RR^d_+,w_{\gam-jp};X)
      = W_0^{\ell,p}(\RR^d_+,w_{\gam};X).
  \end{align*}
The case $k_0\in\NN_1$ follows from the case $k_0=0$ and reiteration for the complex interpolation method (see \cite[Theorem 4.6.1]{BL76}):
\begin{align*}
  \big[W^{k_0,p}_0(\RR^d_+&,w_{\gam};X), W^{k_0+k_1,p}_0(\RR^d_+, w_{\gam};X)\big]_{\frac{\ell}{k_1}} \\ = &\; \big[[L^p(\RRdh, w_{\gam};X), W_0^{k_0+k_1}(\RRdh,w_{\gam};X)]_{\frac{k_0}{k_0+k_1}}, W^{k_0+k_1,p}_0(\RRdh, w_{\gam};X)\big]_{\frac{\ell}{k_1}} \\
  =&\; \big[L^p(\RRdh,w_{\gam};X), W_0^{k_0+k_1,p}(\RRdh,w_{\gam};X)\big]_{\frac{k_0+\ell}{k_0+k_1}}= W^{k_0+\ell,p}_0(\RRdh, w_{\gam};X). 
\end{align*}
This completes the proof.
\end{proof}

We proceed with the interpolation of weighted Sobolev spaces without boundary conditions. As a consequence of Proposition \ref{prop:intp_W_0}, we can derive such a result if $\gam$ is large enough so that certain traces do not exist. On the other hand, for arbitrary $\gam\in(-1,\infty)\setminus\{jp-1:j\in\NN_1\}$ we can only prove one of the embeddings for the complex interpolation space using Wolff interpolation. The converse embedding will be proved in Proposition \ref{prop:intp_Wfull}.

\begin{proposition}\label{prop:intp_W}
  Let $p\in(1,\infty)$, $k_0\in\NN_0$, $k_1\in\NN_1\setminus\{1\}$, $\ell\in\{1,\dots,k_1-1\}$ and let $X$ be a $\UMD$ Banach space. 
  
  \begin{enumerate}[(i)]
      \item \label{it:prop:intp_Wlargegamma} If $\gam\in ((k_0+\ell)p-1, \infty)\setminus\{jp-1:j\in\NN_1\}$, then
\begin{equation*}
    \big[W^{k_0,p}(\RR^d_+,w_{\gam};X), W^{k_0+k_1,p}(\RR^d_+, w_{\gam};X)\big]_{\frac{\ell}{k_1}}=W^{k_0+\ell,p}(\RR^d_+,w_{\gam};X).
  \end{equation*} 
      \item\label{it:prop:intp_W} If $\gam\in(-1,\infty)\setminus\{jp-1:j\in\NN_1\}$, then 
      \begin{equation*}
    \big[W^{k_0,p}(\RR^d_+,w_{\gam};X), W^{k_0+k_1,p}(\RR^d_+, w_{\gam};X)\big]_{\frac{\ell}{k_1}}\hookrightarrow W^{k_0+\ell,p}(\RR^d_+,w_{\gam};X).
  \end{equation*}
  \end{enumerate}
\end{proposition}

\begin{proof} 
\textit{Step 1: the proof of \ref{it:prop:intp_Wlargegamma}.}
    We first consider the case $k_0=0$. It follows from Proposition \ref{prop:intp_W_0} that
\begin{align*}
  W^{\ell,p}(\RR_+^d,w_{\gam};X)  &= W^{\ell,p}_0(\RR_+^d,w_{\gam};X)
   = \big[L^p(\RR^d_+,w_{\gam};X),W^{k_1,p}_0(\RR^d_+,w_{\gam};X)\big]_{\frac{\ell}{k_1}}\\
   & \hookrightarrow \big[L^p(\RR^d_+,w_{\gam};X),W^{k_1,p}(\RR^d_+,w_{\gam};X)\big]_{\frac{\ell}{k_1}}.
\end{align*}
To prove the other embedding, it suffices to show that
\begin{equation*}
  \|u\|_{W^{\ell,p}(\RRdh,w_{\gam};X)}\leq C\|u\|_{[L^p(\RRdh,w_{\gam};X), W^{k_1,p}(\RRdh,w_{\gam};X)]_{\frac{\ell}{k_1}}},\qquad u\in \Cc^\infty(\RRdh;X),
\end{equation*}
by \cite[Lemma 3.9]{LV18}. Let $j\in \NN_0$ be such that $\gam\in(jp-1,(j+1)p-1)$. The case $j=0$ follows from \cite[Proposition 5.5 \& 5.6]{LMV17}. For $j\in\NN_1$ we obtain by applying \cite[Lemma 3.6]{LLRV24}, the case $j=0$ and properties of the complex interpolation method, that
\begin{align*}
  \|u\|_{W^{\ell,p}(\RR^d_+,w_{\gam};X)} & =\|M^{-j} M^j u\|_{W^{\ell,p}(\RR^d_+,w_{\gam};X)}  \\
   & \leq C \|M^j u\|_{W^{\ell,p}(\RR^d_+,w_{\gam-jp};X)}  \\
   & \leq C \|M^j u\|_{[L^p(\RR^d_+,w_{\gam-jp};X), W^{k_1,p}(\RR^d_+,w_{\gam-jp};X)]_{\frac{\ell}{k_1}}}\\
   &\leq C \|u\|_{[L^p(\RR^d_+,w_{\gam};X), W^{k_1,p}(\RR^d_+,w_{\gam};X)]_{\frac{\ell}{k_1}}}.
\end{align*}
Similar to the proof of Proposition \ref{prop:intp_W_0}, the case $k_0\in\NN_1$ follows from the case $k_0=0$ and reiteration. This completes the proof of \ref{it:prop:intp_Wlargegamma}.

\textit{Step 2: the proof of \ref{it:prop:intp_W}.}
 From now on, we abbreviate $W^{k,p}(w_{\gam}):=W^{k,p}(\RRdh,w_{\gam};X)$ for notational convenience.
  Again, let $j\in\NN_0$ be such that $\gam\in (jp-1,(j+1)p-1)$. As also noted in Step 1, for  $j=0$ the result \ref{it:prop:intp_W} follows from \cite[Proposition 5.5 \& 5.6]{LMV17}, so we may assume that $j\geq 1$.
  
  \textit{Step 2a: the case $j\in\{1,\dots, k_0\}$.} We first prove the embedding in \ref{it:prop:intp_W} for $j\in \{1,\dots, k_0\}$. Note that $k_0-j\geq 0$ and $\gam-jp\in (-1,p-1)$, so \cite[Proposition 5.5 \& 5.6]{LMV17} imply that 
  \begin{equation*}\label{eq:intpW_step1}
    \big[W^{k_0-j,p}(w_{\gam-jp}), W^{k_0+k_1-j,p}(w_{\gam-jp})\big]_{\frac{\ell}{k_1}}\hookrightarrow W^{k_0+\ell-j,p}(w_{\gam-jp}).
  \end{equation*}
  We prove by induction on $j_1\geq 0$ that 
    \begin{equation}\label{eq:intpW_IH}
    \big[W^{k_0+j_1-j,p}(w_{\gam+(j_1-j)p}), W^{k_0+k_1+j_1-j,p}(w_{\gam+(j_1-j)p})\big]_{\frac{\ell}{k_1}}\hookrightarrow W^{k_0+\ell+j_1-j,p}(w_{\gam+(j_1-j)p}),
  \end{equation}
  so that \ref{it:prop:intp_W} for $\gam<(k_0+1)p-1$ follows from \eqref{eq:intpW_IH} with $j_1=j$. Assume that \eqref{eq:intpW_IH} holds for some $j_1\geq 0$. Let 
  \begin{equation*}
    \tilde{k}:=k_0+j_1-j\quad \text{ and }\quad \tilde{\gam}:=\gam+(j_1-j)p.
  \end{equation*}
  Let $u\in W^{\tilde{k}+k_1+1,p}(w_{\tilde{\gam}+p})$, then 
  \begin{align*}
\|u\|_{W^{\tilde{k}+\ell+1,p}(w_{\tilde{\gam}+p})}=\| u\|_{W^{\tilde{k}+\ell, p}(w_{\tilde{\gam}+p})}+\sum_{|\alpha|=1}\sum_{|\beta|=\tilde{k}+ \ell}\|M \d^{\alpha+\beta}u\|_{L^p(w_{\tilde{\gam}})}.
  \end{align*}
  Using $\d_1^n M = M\d_1^n +n\d_1^{n-1}$ for $n\in \NN_1$ and the induction hypothesis \eqref{eq:intpW_IH}, we obtain
  \begin{align*}
\sum_{|\alpha|=1}\sum_{|\beta|=\tilde{k}+ \ell}&\|M \d^{\alpha+\beta}u\|_{L^p(w_{\tilde{\gam}})}\\
\leq &\; \sum_{|\alpha|=1} \Big(
\|M\d^{\alpha} u \|_{W^{\tilde{k}+\ell,p}(w_{\tilde{\gam}})}
+ \sum_{\substack{|\beta|=\tilde{k}+ \ell\\\beta_1>0}}  \|\d^{\beta}M\d^{\alpha}u -\beta_1\d_1^{\beta_1-1}\d^{\tilde{\beta}+\alpha}u\|_{L^p(w_{\tilde{\gam}})}\Big)\\
\leq &\; C \Big(\sum_{|\alpha|=1}\|M\d^{\alpha} u \|_{W^{\tilde{k}+\ell,p}(w_{\tilde{\gam}})}+ \|u\|_{W^{\tilde{k}+\ell,p}(w_{\tilde{\gam}})}\Big)\\
\leq &\; C \Big(\sum_{|\alpha|=1}\|M\d^{\alpha} u \|_{[W^{\tilde{k},p}(w_{\tilde{\gam}}),W^{\tilde{k}+k_1,p}(w_{\tilde{\gam}})]_{\frac{\ell}{k_1}} }+ \|u\|_{[W^{\tilde{k},p}(w_{\tilde{\gam}}),W^{\tilde{k}+k_1,p}(w_{\tilde{\gam}})]_{\frac{\ell}{k_1}}}\Big)\\
\leq &\; C \|u\|_{[W^{\tilde{k}+1,p}(w_{\tilde{\gam}+p}),W^{\tilde{k}+k_1+1,p}(w_{\tilde{\gam}+p})]_{\frac{\ell}{k_1}}},
  \end{align*}
  where the last estimate follows from the fact that for $|\alpha|=1$ the operators
  \begin{equation}\label{eq:Hardy_bdd}
  \begin{aligned}
    M\d^{\alpha}, \id  &:W^{\tilde{k}+1,p}(w_{\tilde{\gam}+p})\to W^{\tilde{k},p}(w_{\tilde{\gam}})\quad \text{ and }  \\
    M\d^{\alpha}, \id  &:W^{\tilde{k}+k_1+1,p}(w_{\tilde{\gam}+p})\to W^{\tilde{k}+k_1,p}(w_{\tilde{\gam}})
  \end{aligned}
  \end{equation}
  are bounded. Indeed, from Hardy's inequality (see Lemma \ref{lem:Hardy} using  $\tilde{\gam}> j_1 p -1\geq -1$) it follows that the operator $\id$ in \eqref{eq:Hardy_bdd} is bounded. Furthermore, by \cite[Lemma 3.6]{LLRV24} it follows that the operator $M\d^{\alpha}$ with $|\alpha|=1$ in \eqref{eq:Hardy_bdd} is bounded.
  
  Thus, we have proved
  \begin{equation*}
    \|u\|_{W^{\tilde{k}+\ell+1,p}(w_{\tilde{\gam}+p})}\leq C \|u\|_{[W^{\tilde{k}+1,p}(w_{\tilde{\gam}+p}),W^{\tilde{k}+k_1+1,p}(w_{\tilde{\gam}+p})]_{\frac{\ell}{k_1}}},\quad u\in W^{\tilde{k}+k_1+1,p}(w_{\tilde{\gam}+p}),
  \end{equation*}
  and the induction is completed by noting that $W^{\tilde{k}+k_1+1,p}(w_{\tilde{\gam}+p})$ is dense in the interpolation space (see \cite[Theorem 1.9.3]{Tr78}).
  
  \textit{Step 2b: the case $j\in \{k_0+1,\dots, k_0+\ell-1\}$.} Assume that $\ell\in \{2,\dots, k_1-1\}$ (the case $\ell=1$ is covered by Step 1 and 3) and $j\in \{k_0+1,\dots, k_0+\ell-1\}$. We apply Step 1 and 2a together with Wolff interpolation to prove \ref{it:prop:intp_W}. Since $\gam>jp-1$ and $1\leq j-k_0\leq \ell-1$, Step 1 implies
  \begin{equation*}
    \big[W^{k_0,p}(w_{\gam}), W^{k_0+\ell,p}(w_{\gam})\big]_\lambda = W^{j,p}(w_{\gam}),\quad \text{ with }\lambda:=\frac{j-k_0}{\ell}\in[0,1],
  \end{equation*}
  and since $\gam<(j+1)p-1$, Step 2a implies
  \begin{equation*}
    \big[W^{j,p}(w_{\gam}), W^{k_0+k_1,p}(w_{\gam})\big]_\mu \hookrightarrow W^{k_0+\ell,p}(w_{\gam}),\quad \text{ with }\mu:=\frac{k_0+\ell-j}{k_0+k_1-j}\in[0,1].
  \end{equation*} 
  Wolff interpolation \cite[Corollary 3]{JNP84} gives that
  \begin{equation*}
    \big[W^{k_0,p}(w_{\gam}), W^{k_0+k_1,p}(w_{\gam})\big]_\eta \hookrightarrow W^{k_0+\ell,p}(w_{\gam}),\quad \text{ with }\eta:=\frac{\mu}{1+\lambda(\mu-1)}.
  \end{equation*}
  A straightforward calculation shows that $\eta=\frac{\ell}{k_1}$.
  
   \textit{Step 2c: the case $j\geq k_0+\ell$.} If $j\geq k_0+\ell$, then $\gam>jp-1\geq (k_0+\ell)p-1$, so \ref{it:prop:intp_W} follows from Step 1. This completes the proof.
\end{proof}

\subsection{Main results about complex interpolation of weighted spaces} \label{subsec:int_p2}
We start with the definition of weighted spaces with vanishing boundary conditions on the half-space.
Let $p\in(1,\infty)$ and let $X$ be a Banach space. Assume that the conditions \ref{it:def:syst1}-\ref{it:def:syst2} of Definition \ref{def:syst} hold.
 \begin{enumerate}[(i)]
   \item If $s>0$, $\gam\in(-1,p-1)$ and $\mc{B}$ is of type $(p,s,\gam,\overline{m},\overline{Y})$, then we define for $t\leq s$
\begin{align*}
  H^{t,p}_{\mc{B}}(\RRdh,w_\gam;X)  =\Big\{f\in H^{t,p}(\RRdh,w_\gam;X): \BB^{m_i}f=0 \text{ if }m_i+\frac{\gam+1}{p}<t\Big\}.
\end{align*}
   \item If $k\in \NN_1$, $\gam\in(-1,\infty)\setminus\{jp-1:j\in \NN_1\}$ and $\mc{B}$ is of type $(p,k,\gam,\overline{m},\overline{Y})$, then we define for $\ell\leq k$
\begin{align*}
  W^{\ell,p}_{\mc{B}}(\RRdh,w_\gam;X)  =\Big\{f\in W^{\ell,p}(\RRdh,w_\gam;X): \BB^{m_i}f=0 \text{ if }m_i+\frac{\gam+1}{p}<\ell\Big\}.
\end{align*}
 \end{enumerate}
All the traces in the above definition are well defined by Theorem \ref{thm:retract_W}. We note that if $t,\ell<m_0+\frac{\gam+1}{p}$, then the condition $\BB^{m_i}f=0$ is empty for all $i\in\{0,\dots, n\}$, so that the spaces with and without boundary conditions coincide.\\
 
The following two theorems characterise the complex interpolation spaces for weighted spaces with boundary conditions on the half-space. The proofs of these theorems are given in Section \ref{sec:proof_intp}.
\begin{theorem}[Complex interpolation of weighted Bessel potential spaces]\label{thm:intp_HB}
   Let $p\in(1,\infty)$, $\gam\in(-1,p-1)$ and let $X$ be a $\UMD$ Banach space. Let $\BB$ be a normal boundary operator of type $(p,s,\gam,\overline{m},\overline{Y})$ as in Definition \ref{def:syst}.
    Assume that $\theta\in(0,1)$, $s_0< s_\theta< s_1\leq s$ with $s_\theta=(1-\theta)s_0+\theta s_1$ satisfy $s_0,s_\theta,s_1\notin\{m_i+\frac{\gam+1}{p}: 0\leq i\leq n\}$ and $s_0>-1+\frac{\gam+1}{p}$. 
  Then
      \begin{align*}
  H_{\mc{B}}^{s_\theta,p}(\RRdh,w_{\gam};X) & =\big[H^{s_0,p}(\RRdh,w_{\gam};X), H_{\mc{B}}^{s_1,p}(\RRdh,w_{\gam};X)\big]_\theta\\
   & =\big[H^{s_0,p}_{\mc{B}}(\RRdh,w_{\gam};X), H_{\mc{B}}^{s_1,p}(\RRdh,w_{\gam};X)\big]_\theta.
  \end{align*}
\end{theorem}

\begin{theorem}[Complex interpolation of weighted Sobolev spaces]\label{thm:intp_WB}
   Let $p\in(1,\infty)$, $\gam\in(-1,\infty)\setminus\{jp-1:j\in\NN_1\}$ and let $X$ be a $\UMD$ Banach space. Let $\BB$ be a normal boundary operator of type $(p,k,\gam,\overline{m},\overline{Y})$ as in Definition \ref{def:syst}. Assume that $k_0\in\NN_0$ and $k_1\in \NN_1\setminus\{1\}$ are such that $k_0+k_1\leq k$. Then for $\ell\in \{1,\dots, k_1-1\}$ we have
       \begin{align*}
  W_{\mc{B}}^{k_0+\ell,p}(\RRdh,w_{\gam};X) & =\big[W^{k_0,p}(\RRdh,w_{\gam};X), W_{\mc{B}}^{k_0+k_1,p}(\RRdh,w_{\gam};X)\big]_{\frac{\ell}{k_1}}\\
   & =\big[W^{k_0,p}_{\mc{B}}(\RRdh,w_{\gam};X), W_{\mc{B}}^{k_0+k_1,p}(\RRdh,w_{\gam};X)\big]_{\frac{\ell}{k_1}}.
  \end{align*}
\end{theorem}

As a consequence of Theorem \ref{thm:intp_WB} we can characterise the complex interpolation space of weighted Sobolev spaces without boundary conditions. One of the embeddings, which is needed in the proof of Theorem \ref{thm:intp_WB}, was already proved in Proposition \ref{prop:intp_W}.
\begin{proposition}\label{prop:intp_Wfull}
  Let $p\in(1,\infty)$, $k_0\in\NN_0$, $k_1\in\NN_1\setminus\{1\}$, $\gam\in (-1,\infty)\setminus\{jp-1:j\in\NN_1\}$, $\ell\in\{1,\dots,k_1-1\}$ and let $X$ be a $\UMD$ Banach space. Then
      \begin{equation*}
    W^{k_0+\ell,p}(\RR^d_+,w_{\gam};X)=\big[W^{k_0,p}(\RR^d_+,w_{\gam};X), W^{k_0+k_1,p}(\RR^d_+, w_{\gam};X)\big]_{\frac{\ell}{k_1}}.
  \end{equation*}
\end{proposition}
\begin{proof}
  Note that the embedding ``$\hookleftarrow$" has been proved in Proposition \ref{prop:intp_W}\ref{it:prop:intp_W}. For the other embedding, let $j\in\NN_0$ be the smallest integer such that $k_0+\ell<j+\frac{\gam+1}{p}$. Then by Theorem \ref{thm:intp_WB} we obtain
  \begin{align*}
    W^{k_0+\ell,p}(\RRdh, w_{\gam};X)&=W^{k_0+\ell,p}_{\Tr_j}(\RRdh, w_{\gam};X)\\
    &= \big[W^{k_0,p}(\RRdh, w_{\gam};X), W^{k_0+k_1,p}_{\Tr_j}(\RRdh, w_{\gam};X)\big]_{\frac{\ell}{k_1}}\\
    &\hookrightarrow \big[W^{k_0,p}(\RRdh, w_{\gam};X), W^{k_0+k_1,p}(\RRdh, w_{\gam};X)\big]_{\frac{\ell}{k_1}}.
  \end{align*}
  This completes the proof.
\end{proof}

By standard localisation techniques, the complex interpolation results above can be extended to smooth bounded domains. Using localisation the definition of the traces on the half-space and the function space $W^{\ell,p}_\BB(\RRdh, w_{\gam};X)$ can be extended to smooth bounded domains $\OO\subseteq \RRd$, we refer to \cite[Section 5.2]{LV18} for details. Recall that $w_{\gam}^{\d\OO}(x):=\operatorname{dist}(x,\d\OO)^\gam$ for $x\in\OO$. From Theorem \ref{thm:intp_WB} and Proposition \ref{prop:intp_Wfull} we obtain the following result. Similar results for Bessel potential spaces on domains with $\gam\in(-1,p-1)$ can be obtained using Theorem \ref{thm:intp_HB}.
\begin{proposition}
  Assume that the conditions of Theorem \ref{thm:intp_WB} hold and let $\OO$ be a bounded $C^{\infty}$-domain. Then for $\ell\in \{1,\dots, k_1-1\}$ we have
       \begin{align*}
  W_{\mc{B}}^{k_0+\ell,p}(\OO,w^{\d\OO}_{\gam};X) & =\big[W^{k_0,p}(\OO,w^{\d\OO}_{\gam};X), W_{\mc{B}}^{k_0+k_1,p}(\OO,w^{\d\OO}_{\gam};X)\big]_{\frac{\ell}{k_1}}\\
   & =\big[W^{k_0,p}_{\mc{B}}(\OO,w^{\d\OO}_{\gam};X), W_{\mc{B}}^{k_0+k_1,p}(\OO,w^{\d\OO}_{\gam};X)\big]_{\frac{\ell}{k_1}},
  \end{align*}
  and
  \begin{equation*}
    W^{k_0+\ell,p}(\OO,w^{\d\OO}_{\gam};X)=\big[W^{k_0,p}(\OO,w^{\d\OO}_{\gam};X), W^{k_0+k_1,p}(\OO, w^{\d\OO}_{\gam};X)\big]_{\frac{\ell}{k_1}}.
  \end{equation*}
\end{proposition}
For simplicity, we only deal with smooth domains to avoid problems with the space-dependent coefficients in the boundary operators, although this assumption can be weakened. In \cite{LLRV25}, we prove an interpolation result for Dirichlet and Neumann boundary conditions, but with minimal smoothness assumptions on the domain. \\

We close this section with an application of the complex interpolation result in Proposition \ref{prop:intp_Wfull} to derive Fubini's property (mixed derivative theorem) for weighted Sobolev spaces on the half-space. It extends the range of weights obtained in \cite[Corollary 3.19]{LV18}. The Fubini property for unweighted spaces is studied in \cite[Section I.4]{Tr01} and some results for weighted spaces are contained in \cite[Section 5]{Li21}.
\begin{proposition}\label{prop:Fubprop}
  Let $p\in(1,\infty)$, $d\geq 2$, $k\in\NN_0$, $\gam\in (-1,\infty)\setminus\{jp-1:j\in\NN_1\}$ and let $X$ be a $\UMD$ Banach space. Then
  \begin{equation*}
    L^p(\RR^{d-1};W^{k,p}(\RR_+,w_{\gam};X))\cap W^{k,p}(\RR^{d-1};L^p(\RR_+,w_{\gam};X))= W^{k,p}(\RRdh, w_{\gam};X).
  \end{equation*}
\end{proposition}
\begin{proof}
  The embedding ``$\hookleftarrow$" is clear. For the other embedding, let $k_1,k_2\in\{0,\dots, k\}$ such that $k_1+k_2=k$. By \cite[Theorem 3.18]{LV18} and Proposition \ref{prop:intp_Wfull}, we have
    \begin{align*}
    L^p(\RR^{d-1}; W^{k,p}&(\RR_+,w_{\gam};X))\cap W^{k,p}(\RR^{d-1}; L^p(\RR_+,w_{\gam};X)) \\
     =& \;H^{0,p}(\RR^{d-1}; W^{k,p}(\RR_+,w_{\gam};X))\cap H^{k,p}(\RR^{d-1}; L^p(\RR_+,w_{\gam};X)) \\
      \hookrightarrow &\; \big[H^{0,p}(\RR^{d-1}; W^{k,p}(\RR_+,w_{\gam};X)), H^{k,p}(\RR^{d-1}; L^p(\RR_+,w_{\gam};X))\big]_{\frac{k_1}{k}} \\
     =& \;H^{k_1,p}(\RR^{d-1};[W^{k,p}(\RR_+,w_{\gam};X),L^p(\RR_+,w_{\gam};X)]_{\frac{k_1}{k}})\\
     =&\; W^{k_1,p}(\RR^{d-1}; W^{k_2,p}(\RR_+,w_{\gam};X)),
  \end{align*}
  where it was used that $1-\frac{k_1}{k}=\frac{k_2}{k}$. This proves the result.
\end{proof}

\subsection{The proofs of Theorems \ref{thm:intp_HB} and \ref{thm:intp_WB}}\label{sec:proof_intp}
Using the trace theorem for $\BB$ from Section \ref{sec:traceB}, we will prove Theorems \ref{thm:intp_HB} and \ref{thm:intp_WB}. The proofs follow the arguments in \cite[Theorems VIII.2.4.3, VIII.2.4.4 \& VIII.2.4.8]{Am19} for unweighted Bessel potential spaces. We will provide the proof of Theorem \ref{thm:intp_WB} in full detail and afterwards indicate how the arguments should be adapted to prove Theorem \ref{thm:intp_HB}.

\begin{proof}[Proof of Theorem \ref{thm:intp_WB}]
First, we prove in Steps 1 and 2 that for $\theta:=\frac{\ell}{k_1}$ we have
     \begin{equation}\label{it:thm:intp_WB_1}
  W_{\mc{B}}^{k_0+\ell,p}(\RRdh,w_{\gam};X)  =\big[W^{k_0,p}(\RRdh,w_{\gam};X), W_{\mc{B}}^{k_0+k_1,p}(\RRdh,w_{\gam};X)\big]_{\theta}.
  \end{equation} 
  For notational convenience we write $W^{k,p}:=W^{k,p}(\RRdh,w_{\gam};X)$ for $k\in \NN_0$.

\textit{Step 1: the embedding ``$\hookleftarrow$". }Note that by Proposition \ref{prop:intp_W}\ref{it:prop:intp_W} it holds that
\begin{equation}\label{eq:proofWB_est1}
  \big[W^{k_0,p},W_\BB^{k_0+k_1,p}\big]_{\theta} \hookrightarrow\big[W^{k_0,p},W^{k_0+k_1,p}\big]_{\theta}\hookrightarrow W^{k_0+\ell,p}.
\end{equation}
Now, suppose that $u\in  [W^{k_0,p},W_\BB^{k_0+k_1,p}]_{\theta}$. We prove that $u\in W^{k_0+\ell,p}_\BB$ in Step 1a and 1b, i.e., $u$ also satisfies the required boundary conditions. We distinguish between the case where the boundary traces do not exist (Step 1a), and the case where they exist and must be shown to vanish (Step 1b).

\textit{Step 1a. }Assume that $k_0+\ell<m_0+\frac{\gam+1}{p}$. Then it holds that $W^{k_0+\ell, p}_\BB=W^{k_0+\ell, p}$. Therefore, the embedding $$\big[W^{k_0,p}, W_\BB^{k_0+k_1,p}\big]_\theta\hookrightarrow W^{k_0+\ell, p}_\BB$$ follows from \eqref{eq:proofWB_est1}.

\textit{Step 1b. }Assume that $k_0+\ell>m_0+\frac{\gam+1}{p}$ and let $i\leq n$ be the largest integer such that $m_i+\frac{\gam+1}{p}<k_0+\ell$. By definition of the complex interpolation method (see Section \ref{sec:compl_intp_defs}), there exists an $f\in \mathscr{H}(W^{k_0,p}, W_\BB^{k_0+k_1,p})$ such that $f(\theta)=u$. By \eqref{eq:est_compl_shift} and \eqref{eq:proofWB_est1} we have that the restriction $f|_{\mathbb{S}_{[\theta,1]}}$ is bounded and continuous with values in $W^{k_0+\ell,p}$ and is holomorphic on $\S_{(\theta,1)}$. Therefore, if $0\leq j\leq i$, then $\BB^{m_j}f$ is bounded and continuous on $\S_{[\theta,1]}$ with values in $B_{p,p}^{k_0+\ell-m_j-\frac{\gam+1}{p}}(\RR^{d-1};Y_j)$ (see Theorem \ref{thm:retract_W}) and is holomorphic on $\S_{(\theta,1)}$. Moreover, for $\Re(z)=1$ we have $\BB^{m_j}f(z)=0$. By the three lines lemma (see \cite[Lemma 1.1.2]{BL76}), $\BB^{m_j} f(z)$ vanishes identically and thus
\begin{equation*}
  \BB^{m_j}u=\BB^{m_j}f(\theta)=0\qquad \text{ for }0\leq j\leq i.
\end{equation*}
This proves $u\in W^{k_0+\ell, p}_\BB$ and thus $[W^{k_0,p}, W_\BB^{k_0+k_1,p}]_{\theta}\hookrightarrow W^{k_0+\ell, p}_\BB$.

\textit{Step 2: the embedding ``$\hookrightarrow$".} We continue with the proof of the embedding ``$\hookrightarrow$" in \eqref{it:thm:intp_WB_1}. Again, we consider two cases depending on which traces exist. In Step 2a, we consider the case where all traces exist. For $u\in W^{k_0+\ell,p}_{\mc{B}}$ we explicitly construct a decomposition $u=v+(u-v)$, where $v$ satisfies the boundary conditions, and $u-v$ has vanishing traces. In this setting, we construct a right inverse for an extended system of boundary operators $\mc{C}$ to find such a $v$. In Step 2b, we consider the case where certain traces do not exist and the result can be deduced by reiteration.

Before proceeding with the main argument, we construct a modified boundary operator such that its leading order coefficient is a projection. This is a necessary preliminary step for the construction of the extended system of boundary operators $\mc{C}$ in Step 2a.

Let $0\leq i\leq n$ and let $b_{i,m_i}$ be the leading order coefficient of $\BB^{m_i}$ for which $b^{{\rm c}}_{i,m_i}$ is the right inverse from Definition \ref{def:nor_bound_op}. Define the boundary operator $\tilde{\BB}^{m_i}:= b^{{\rm c}}_{i,m_i}\BB$. Then it is straightforward to verify that
\begin{equation}\label{eq:proofWB_kernel}
  {\rm ker}(\tilde{\BB}^{m_i})= {\rm ker}(\BB^{m_i})\qquad \text{ for }0\leq i \leq n.
\end{equation}  
Similar to \eqref{eq:BTr} we write
\begin{equation*}
  \tilde{\BB}^{m_i}=\sum_{j=0}^{m_i}\tilde{b}_{i,j}\Tr_j\qquad\text{ where } \tilde{b}_{i,j}:= b^{{\rm c}}_{i,m_i}b_{i,j}.
\end{equation*}
We claim that $\pi_{m_i}:=\tilde{b}_{i,m_i}:\RR^{d-1}\to \mc{L}(X)$ is a projection. Indeed, for $0\leq j\leq m_i$, we have
\begin{equation}\label{eq:proofWB_proj}
  \pi_{m_i}\tilde{b}_{i,j}=b^{{\rm c}}_{i,m_i}b_{i,m_i}b^{{\rm c}}_{i,m_i}b_{i,j} =b^{{\rm c}}_{i,m_i}b_{i,j}=\tilde{b}_{i,j}.
\end{equation}
Moreover, \eqref{eq:proofWB_proj} implies that for $i\in\{0,\dots, n\}$ we have
\begin{equation}\label{eq:proofWB_projB}
  \pi_{m_i}\tilde{\BB}^{m_i}=\tilde{\BB}^{m_i}.
\end{equation}

\textit{Step 2a. }Assume that $k_0+k_1>m_n+\frac{\gam+1}{p}$. We start with the construction of an extended system of normal boundary operators $\mc{C}$ whose kernel coincides with $W^{k_0+k_1,p}_0$.

Let $a$ be the largest integer such that $a+\frac{\gam+1}{p}<k_0+k_1$ and define $\overline{a}=(0,1,\dots, a) $. Moreover, by setting $\overline{X}=(X,\dots, X)$ with length $a+1$, we can define a normal boundary operator of type $(p,k_0+k_1,\gam, \overline{a}, \overline{X})$ by
\begin{equation}\label{eq:op_C}
  \begin{aligned}
  \mc{C}&=(\mc{C}^0,\dots, \mc{C}^a): W^{k_0+k_1,p}\to \prod_{j=0}^a B^{k_0+k_1-j-\frac{\gam+1}{p}}_{p,p}(\RR^{d-1};X),\\
  \mc{C}^j&:=\begin{cases}
               \Tr_j & \mbox{if } j\notin \{m_0,\dots, m_n\}, \\
               (1-\pi_j)\Tr_j +  \tilde{B}^j & \mbox{if }j\in \{m_0,\dots, m_n\}.
             \end{cases}
\end{aligned}
\end{equation}
Note that by \eqref{eq:proofWB_projB} we have
\begin{equation}\label{eq:proofWB_BC}
  \pi_j\mc{C}^j=\pi_j \tilde{\BB}^j = \tilde{\BB}^j\qquad \text{ for }j\in\{m_0,\dots, m_n\}.
\end{equation}
We show that 
\begin{equation}\label{eq:proofWB_claim}
  \mc{C}v =0\quad \text{if and only if}\quad \overline{\Tr}_a v=0\qquad\text{ for }v\in W^{k_0+k_1,p}.
\end{equation}
Indeed, let $v\in W^{k_0+k_1,p} $ and $\mc{C}v=0$, then for $0\leq \tilde{a}\leq a$ with $\tilde{a}\notin \{m_0,\dots, m_n\}$, we have 
\begin{equation}\label{eq:proofWb_eq1}
       \Tr_{\tilde{a}}v=\mc{C}^{a}v=0.
\end{equation}
If $\tilde{a}=m_i\in \{m_0,\dots, m_n\}$ for some $0\leq i\leq n$, then by \eqref{eq:proofWB_projB} and \eqref{eq:proofWb_eq1} it holds that
\begin{align*}
0=\mc{C}^{m_i}v &= (1-\pi_{m_i})\Tr_{m_i}v + \pi_{m_i}\tilde{\BB}^{m_i}v\\
&=\Tr_{m_i} v +\sum_{j=0}^{m_i-1}\tilde{b}_{i,j}\Tr_j v\\
&=\Tr_{m_i} v +\sum_{j\in\{m_0,\dots, m_{i-1}\}}\tilde{b}_{i,j}\Tr_j v.
\end{align*}
It follows that $\Tr_{m_0}v=\mc{C}^{m_0}v=0$ and successively we obtain $\Tr_{m_i}v=\mc{C}^{m_i}v=0$ for $0\leq i \leq n$ as well. The converse statement is trivial, so we have proved \eqref{eq:proofWB_claim}. 

From \eqref{eq:proofWB_claim} and the definition of $W^{k_0+k_1,p}_0$, it follows that
\begin{equation*}
  W^{k_0+k_1,p}_{\mc{C}}=W^{k_0+k_1,p}_{\overline{\Tr}_{a}}=W_0^{k_0+k_1,p}.
\end{equation*}
By Theorem \ref{thm:retract_W} there exists a right inverse 
\begin{equation*}
  \ext_{\mc{C}}:\prod_{j=0}^a B^{k_0+k_1-j-\frac{\gam+1}{p}}_{p,p}(\RR^{d-1};X)\to W^{k_0+k_1,p}
\end{equation*}
for $\mc{C}$, which is universal with respect to $k_0+k_1\in(a+\frac{\gam+1}{p}, k]$.

Let $u\in W^{k_0+\ell,p}_\BB=W^{k_0+\ell,p}_{\tilde{\BB}}$ (recall \eqref{eq:proofWB_kernel}) and we will now prove that $u\in[W^{k_0,p}, W^{k_0+k_1,p}_\BB]_{\theta}$ with $\theta=\frac{\ell}{k_1}$. 
Using the extension operator $\ext_{\mc{C}}$ from above, we construct a suitable $v \in [W^{k_0,p}, W^{k_0+k_1,p}_{\BB}]_{\theta}$ such that $u=v+(u-v)$ and all the traces of $u-v$ vanish.

Let $a_1$ be the largest integer such that $a_1+\frac{\gam+1}{p}<k_0+\ell$ and define for $j\leq a_1$
\begin{equation*}
  g_j:=\begin{cases}
    \mc{C}^j u & \mbox{if } j\notin \{m_0,\dots, m_n\}, \\
    (1-\pi_j)\mc{C}^j u & \mbox{if } j\in \{m_0,\dots, m_n\}.
  \end{cases}
\end{equation*}
By Lemma \ref{lem:boundB} and \cite[Theorem 14.4.30]{HNVW24}, we obtain
\begin{equation*}
  g_j \in B^{k_0+\ell-j-\frac{\gam+1}{p}}_{p,p}(\RR^{d-1};X)=\big[B^{k_0-j-\frac{\gam+1}{p}}_{p,p}(\RR^{d-1};X), B^{k_0+k_1-j-\frac{\gam+1}{p}}_{p,p}(\RR^{d-1};X)\big]_{\theta}.
\end{equation*}
Therefore, by definition of the complex interpolation method, there exists a $$\tilde{f}_j\in \mathscr{H}\big(B^{k_0-j-\frac{\gam+1}{p}}_{p,p}(\RR^{d-1};X), B^{k_0+k_1-j-\frac{\gam+1}{p}}_{p,p}(\RR^{d-1};X)\big)$$ such that $\tilde{f}_j(\theta)=g_j$ for $j\leq a_1$. In addition, for  $j\leq a_1$ we define
\begin{equation*}
  f_j:=\begin{cases}
    \tilde{f}_j & \mbox{if } j\notin \{m_0,\dots, m_n\}, \\
    (1-\pi_j)\tilde{f}_j & \mbox{if } j\in \{m_0,\dots, m_n\},
  \end{cases}
\end{equation*}
which satisfies the same properties as $\tilde{f}_j$. Then $$F:=\ext_{\mc{C}}(f_0,\dots, f_{a_1}, 0,\dots, 0)\in \mathscr{H}(W^{k_0,p}, W^{k_0+k_1,p})$$ and
\begin{equation*}
  v:=F(\theta)=\ext_{\mc{C}}\big(f_0(\theta),\dots, f_{a_1}(\theta), 0,\dots, 0\big)=\ext_{\mc{C}}(g_0,\dots, g_{a_1}, 0,\dots, 0).
\end{equation*}
Moreover, $F$ is bounded and continuous on $\{z\in \CC: \Re z=1\}$ with values in $W^{k_0+k_1,p}$ and by \eqref{eq:proofWB_BC}, properties of the right inverse $\ext_{\mc{C}}$ and the definition of $f_j$, we have for $j\in \{m_0,\dots, m_n\}$
\begin{align*}
  \tilde{\BB}^jF(1)&=\pi_j\mc{C}^j\ext_{\mc{C}}(f_0(1),\dots, f_{a_1}(1),0,\dots, 0)\\
  &=\begin{cases}\pi_j f_j(1)=0 & \text{ if } j\in \{m_0,\dots, m_n\}, \,j\leq a_1,\\0&\text{ if }j\in \{m_0,\dots, m_n\},\,a_1< j\leq a.\end{cases}
\end{align*}
This implies $F\in \mathscr{H}(W^{k_0,p}, W^{k_0+k_1,p}_{\tilde{\BB}})$ and thus by \eqref{eq:proofWB_kernel} we have
\begin{equation}\label{eq:v_intpW}
  v=F(\theta)\in \big[W^{k_0,p}, W^{k_0+k_1,p}_{\BB}\big]_{\theta}.
\end{equation} 
To continue, we prove that $u-v \in W^{k_0+\ell,p}_{\mc{C}}$. Indeed, by definition of $g_j$, \eqref{eq:proofWB_BC} and the fact that $\tilde{\BB}^j u =0$ for $j\in \{m_0,\dots, m_n\}$, we obtain
\begin{align*}
\mc{C}^j(u-v)&=\mc{C}^j u -\mc{C}^j \ext_{\mc{C}}(g_0,\dots, g_{a_1}, 0,\dots, 0)\\
& =  \mc{C}^j u - g_j= \begin{cases}
g_j-g_j =0& \mbox{if } j\notin \{m_0,\dots, m_n\},\, j\leq a_1,\\
 \mc{C}^j u - (1-\pi_j)\mc{C}^j u =0 & \mbox{if } j\in \{m_0,\dots, m_n\}, \,j\leq a_1.
 \end{cases}
\end{align*}
It follows from Proposition \ref{prop:intp_W_0} that
\begin{align}\label{eq:v_intp2W}
u-v\in  W^{k_0+\ell,p}_{\mc{C}}= W^{k_0+\ell,p}_0=\big[W^{k_0,p}_0,W^{k_0+k_1,p}_0\big]_{\theta}\hookrightarrow \big[W^{k_0,p},W^{k_0+k_1,p}_{\BB}\big]_{\theta}.
\end{align}
Then combining \eqref{eq:v_intpW} and \eqref{eq:v_intp2W} implies $u=v+(u-v)\in [W^{k_0,p},W^{k_0+k_1,p}_{\BB}]_{\theta}$ with $\theta=\frac{\ell}{k_1}$. This completes the proof of the embedding ``$\hookrightarrow$" in \eqref{it:thm:intp_WB_1} if $k_0+k_1 > m_n+\frac{\gam+1}{p}$.

\textit{Step 2b. }We prove the embedding ``$\hookrightarrow$" in \eqref{it:thm:intp_WB_1} if  $k_0+k_1<m_n+\frac{\gam+1}{p}$. Fix $a\in\NN_0$ such that $k_0+a\in (m_n+\frac{\gam+1}{p},k]$. Then by reiteration for the complex interpolation method (see \cite[Theorem 4.6.1]{BL76}) and Step 2a twice, we obtain
\begin{equation*}
  \big[W^{k_0,p}, W^{k_0+k_1,p}_{\BB}\big]_{\frac{\ell}{k_1}} =  \big[W^{k_0,p}, [W^{k_0,p}, W^{k_0+a,p}_{\BB}]_{\frac{k_1}{a}}\big]_{\frac{\ell}{k_1}} = \big[W^{k_0,p}, W^{k_0+a,p}_{\BB}\big]_{\frac{\ell}{a}} = W^{k_0+\ell,p}_{\BB}.
\end{equation*}
Combining Steps 1 and 2 completes the proof of \eqref{it:thm:intp_WB_1}.

\textit{Step 3. }It remains to prove that 
\begin{align*}
W^{k_0+\ell,p}_{\BB}(\RRdh,w_{\gam};X)
    =\big[W^{k_0,p}_{\mc{B}}(\RRdh,w_{\gam};X), W_{\mc{B}}^{k_0+k_1,p}(\RRdh,w_{\gam};X)\big]_{\frac{\ell}{k_1}}.
  \end{align*}
Fix $a\in\NN_0$ such that $a<k_0$ and define $\theta_1:=(k_0-a)/(k_0+k_1-a)$. Then by reiteration and \eqref{it:thm:intp_WB_1} twice, we obtain
\begin{align*}
  \big[W^{k_0,p}_{\BB}, W^{k_0+k_1,p}_{\BB}\big]_{\frac{\ell}{k_1}} &=  \big[[W^{a,p}, W^{k_0+k_1,p}_{\BB}]_{\theta_1}, W^{k_0+k_1,p}_{\BB}\big]_{\frac{\ell}{k_1}}\\
  &= \big[W^{a,p}, W^{k_0+k_1,p}_{\BB}\big]_{(1-\frac{\ell}{k_1})\theta_1+\frac{\ell}{k_1}} = W^{k_0+\ell,p}_{\BB}.
\end{align*}
This completes the proof of the theorem.
\end{proof}

We conclude with the proof of Theorem \ref{thm:intp_HB}.
\begin{proof}[Proof of Theorem \ref{thm:intp_HB}]
We first prove in Steps 1 and 2 that 
\begin{equation}\label{it:thm:intp_HB_1}
   H_{\mc{B}}^{s_\theta,p}(\RRdh,w_{\gam};X)  =\big[H^{s_0,p}(\RRdh,w_{\gam};X), H_{\mc{B}}^{s_1,p}(\RRdh,w_{\gam};X)\big]_\theta.
\end{equation}
For notational convenience we write $H^{s,p}:=H^{s,p}(\RRdh,w_{\gam};X)$.

\textit{Step 1: the embedding ``$\hookleftarrow$". } By \cite[Proposition 5.6]{LMV17} it holds that
\begin{equation*}
  \big[H^{s_0,p},H_\BB^{s_1,p}\big]_\theta \hookrightarrow \big[H^{s_0,p},H^{s_1,p}\big]_\theta=H^{s_\theta,p}.
\end{equation*}
Thus the embedding $[H^{s_0,p}, H_\BB^{s_1,p}]_\theta\hookrightarrow H^{s_\theta, p}_\BB$ can be proved similarly as in Step 1 of the proof of Theorem \ref{thm:intp_WB}.

\textit{Step 2: the embedding ``$\hookrightarrow$".} We adopt the same notation as in Step 2 of the proof of Theorem \ref{thm:intp_WB}. 

\textit{Step 2a. }Assume that $s_1>m_n+\frac{\gam+1}{p}$ and $s_1\notin \NN_0+\frac{\gam+1}{p}$. Note that the second condition on $s_1$ is new compared to the proof of Theorem \ref{thm:intp_WB}. This condition will be removed in Step 2b.

 Let $a$ be the largest integer such that $a+\frac{\gam+1}{p}<s_1$. The operator
\begin{equation*}
  \mc{C}=(\mc{C}^0,\dots, \mc{C}^a): H^{s_1,p}\to \prod_{j=0}^a B^{s_1-j-\frac{\gam+1}{p}}_{p,p}(\RR^{d-1};X),
\end{equation*}
as defined in \eqref{eq:op_C} satisfies \eqref{eq:proofWB_claim} for $v\in H^{s_1,p}$. Therefore, we have
\begin{equation*}
  H^{s_1,p}_{\mc{C}}= H^{s_1,p}_{\bTr_a} = H_0^{s_1,p}.
\end{equation*}
We can now proceed similarly as in Step 2a of the proof of Theorem \ref{thm:intp_WB} using Proposition \ref{prop:H_0_int} to obtain (cf. \eqref{eq:v_intp2W})
\begin{equation*}
  H^{s_\theta,p}_{\mc{C}} = H_0^{s_\theta,p} = \big[H_0^{s_0,p},H_0^{s_1,p}\big]_\theta \hookrightarrow \big[H^{s_0,p}, H^{s_1,p}_{\BB}\big]_\theta.
\end{equation*}

\textit{Step 2b. }Assume that $s_1>m_n+\frac{\gam+1}{p}$ and there exists an $\ell\in \NN_0$ such that $s_1= \ell+\frac{\gam+1}{p}$. Fix $t\in (s_1, s]$ such that $t\notin \NN_0+\frac{\gam+1}{p}$. Furthermore, let $\theta_1:=(s_1-s_0)/(t-s_0)$. Then by reiteration for the complex interpolation method (see \cite[Theorem 4.6.1]{BL76}) and Step 2a twice, we obtain
\begin{equation*}
  \big[H^{s_0,p}, H^{s_1,p}_{\BB}\big]_\theta =  \big[H^{s_0,p}, [H^{s_0,p}, H^{t,p}_{\BB}]_{\theta_1}\big]_\theta= \big[H^{s_0,p}, H^{t,p}_{\BB}\big]_{\theta\theta_1} = H^{s_\theta,p}_{\BB}.
\end{equation*}

\textit{Step 2c. }Assume that $s_1<m_n+\frac{\gam+1}{p}$. Fix $t\in (m_n+\frac{\gam+1}{p},s]$ and let $\theta_1:=(s_1-s_0)/(t-s_0)$. Then again by reiteration and Step 2a, we obtain $[H^{s_0,p}, H^{s_1,p}_{\BB}]_\theta=H^{s_\theta,p}_{\BB}$. This completes the proof of \eqref{it:thm:intp_HB_1}.

\textit{Step 3. }It remains to prove 
\begin{equation*}
    H_{\mc{B}}^{s_\theta,p}(\RRdh,w_{\gam};X) =\big[H^{s_0,p}_{\mc{B}}(\RRdh,w_{\gam};X), H_{\mc{B}}^{s_1,p}(\RRdh,w_{\gam};X)\big]_\theta.
\end{equation*}
Fix $\tilde{s}<s_0$, with $\tilde{s}>0$ if $s_0>0$ and $\tilde{s}>-1+\frac{\gam+1}{p}$ otherwise. Let $\theta_1:=(s_0-\tilde{s})/(s_1-\tilde{s})$. Then by reiteration and \eqref{it:thm:intp_HB_1} twice, we obtain
\begin{equation*}
  \big[H^{s_0,p}_{\BB}, H^{s_1,p}_{\BB}\big]_\theta =  \big[[H^{\tilde{s},p}, H^{s_1,p}_{\BB}]_{\theta_1}, H^{s_1,p}_{\BB}\big]_\theta= \big[H^{\tilde{s},p}, H^{s_1,p}_{\BB}\big]_{(1-\theta)\theta_1+\theta} = H^{s_\theta,p}_{\BB}.
\end{equation*}
This completes the proof of the theorem.
\end{proof}

\bibliographystyle{plain}
\bibliography{Bibliography_Roodenburg25}
\end{document}